\documentclass{article}

\usepackage{algorithm,algorithmic}
\usepackage{amsfonts}
\usepackage{amsmath}
\usepackage{amsrefs}
\usepackage{amssymb}
\usepackage{amsthm}
\usepackage{bm}
\usepackage{booktabs}
\usepackage{enumitem}
\usepackage[margin=1.25in]{geometry}
\usepackage{graphicx}
\usepackage{hyperref}
\usepackage{lineno}
\usepackage{mathrsfs}
\usepackage{mathtools}
\usepackage{ulem}
\usepackage{xcolor}
\usepackage{subfig}
\usepackage{multirow}
\usepackage{multicol}

\DefineSimpleKey{bib}{primaryclass}{}
\DefineSimpleKey{bib}{archiveprefix}{}
\newcommand{\squarebracket}[1]{[#1]}
\newcommand{\arxivurl}[1]{\href{https://arxiv.org/abs/#1}{\texttt{arXiv: #1}}}

\BibSpec{arXiv}{%
	+{}{\PrintAuthors}{author}
	+{,}{ \textit}{title}
	+{,}{ \arxivurl}{eprint}
	+{}{ \squarebracket}{primaryclass}
	+{,}{ }{date}
	+{.}{}{transition}
}

\BibSpec{inproceedings}{%
	+{}{\PrintAuthors}{author}
	+{,}{ \textit}{title}
	+{,}{ }{booktitle}
	+{}{ }{volume}
	+{,}{ }{publisher}
	+{,}{ }{date}
	+{,}{ pp.~}{pages}
	+{.}{}{transition}
}

\allowdisplaybreaks[4]

\hypersetup{colorlinks,breaklinks}
\hypersetup{
	colorlinks=true,
	allcolors={green!50!black},
	urlcolor={red!50!black}
}

\newcommand{\cmg}{\textup{\textsc{cmg}}}
\newcommand{\ee}{\mathrm{ee}}
\newcommand{\st}{\mathrm{s.t.}}
\newcommand{\lrbrace}[1]{\left\{#1\right\}}
\newcommand{\lrbracket}[1]{\left(#1\right)}
\newcommand{\lrsquare}[1]{\left[#1\right]}
\newcommand{\T}{\top}
\newcommand{\trace}{\mathrm{Tr}}
\newcommand{\eps}{\varepsilon}

\newcommand{\calI}{\mathcal{I}}
\newcommand{\calO}{\mathcal{O}}

\newcommand{\calT}{\mathcal{T}}
\newcommand{\calU}{\mathcal{U}}

\newcommand{\scrT}{\mathscr{T}}
\newcommand{\mb}[1]{\mathbf{#1}}
\newcommand{\norm}[1]{\left\Vert#1\right\Vert}
\newcommand{\snorm}[1]{\Vert#1\Vert}
\newcommand{\abs}[1]{\left|#1\right|}
\newcommand{\inner}[1]{\left\langle#1\right\rangle}
\newcommand{\R}{\mathbb{R}}
\newcommand{\N}{\mathbb{N}}

\newcommand{\errobj}{\mathrm{err\_obj}}
\newcommand{\errsce}{\mathrm{err\_sce}}
\newcommand{\inter}{\mathrm{inter}}
\newcommand{\KL}{\mathrm{KL}}

\newcommand{\mba}{\mathbf{a}}
\newcommand{\mbb}{\mathbf{b}}
\renewcommand{\d}{\mathbf{d}}

\newcommand{\rr}{\mathbf{r}}
\newcommand{\uu}{\mathbf{u}}

\newcommand{\vvv}{\mathbf{v}}

\newcommand{\brho}{\bm{\varrho}}
\newcommand{\E}{\mathsf{E}}
\newcommand{\PP}{\mathsf{P}}
\newcommand{\var}{\mathsf{Var}}
\newcommand{\dd}{\,\mathrm{d}}
\newcommand{\Diag}{\mathrm{Diag}}

\newcommand{\one}{\mathbf{1}}

\newcommand{\bcd}{\textup{\textsc{bcd}}}
\newcommand{\bcg}{\textup{\textsc{bcg}}}

\newcommand{\palm}{\textup{\textsc{palm}}}

\newcommand{\klalm}{\textup{\textsc{klalm}}}
\newcommand{\sklalm}{\textup{\textsc{S-klalm}}}
\newcommand{\eralm}{\textup{\textsc{eralm}}}
\newcommand{\seralm}{\textup{\textsc{S-eralm}}}

\newcommand{\sklalmgr}{\textup{\textsc{S-klalm-cmg}}}

\makeatletter
\newcommand{\labeltarget}[1]{\Hy@raisedlink{\hypertarget{#1}{}}}
\makeatother
\newcommand{\tabincell}[2]{\begin{tabular}{@{}#1@{}}#2\end{tabular}}
\makeatletter

\DeclareMathOperator*{\minimize}{minimize}
\DeclareMathOperator*{\argmin}{arg\,min}

\newtheorem{assumption}{Assumption}
\newtheorem{corollary}{Corollary}
\newtheorem{definition}{Definition}
\newtheorem{lemma}{Lemma}
\newtheorem{remark}{Remark}
\newtheorem{theorem}{Theorem}

\numberwithin{equation}{section}
\graphicspath{{figures}}

\title{Sampling-Based Methods for Multi-Block Optimization Problems over Transport Polytopes\thanks{{\bf Funding:} The work of the first author was supported by the National Key R\&D Program of China (2020YFA0711900, 2020YFA0711904). The work of the second author was supported by the Outstanding Innovative Talents Cultivation Funded Programs 2021 of Renmin University of China. The work of the third author was supported in part by the National Natural Science Foundation of China (12125108, 11971466, 12226008, 11991021, 11991020, 12021001, 12288201), Key Research Program of Frontier Sciences, Chinese Academy of Sciences (ZDBS-LY-7022), and CAS-Croucher Funding Scheme for Joint Laboratories ``CAS AMSS-PolyU Joint Laboratory of Applied Mathematics: Nonlinear Optimization Theory, Algorithms and Applications''. The work of the fourth author was supported by Beijing Municipal Natural Science Foundation (1232019), the National Natural Science Foundation of China (12101606), and Renmin University of China research fund program for young scholars.}}
\date{}
\author{
    Yukuan Hu\thanks{State Key Laboratory of Scientific and Engineering Computing, Academy of Mathematics and Systems Science, Chinese Academy of Sciences, and University of Chinese Academy of Sciences, Beijing, China (\href{mailto:ykhu@lsec.cc.ac.cn}{ykhu@lsec.cc.ac.cn}, \href{mailto:liuxin@lsec.cc.ac.cn}{liuxin@lsec.cc.ac.cn}).}
    \and Mengyu Li\thanks{Institute of Statistics and Big Data, Renmin University of China, Beijing, China (\href{mailto:limengyu516@ruc.edu.cn}{limengyu516@ruc.edu.cn}).}
    \and Xin Liu\footnotemark[2]
    \and Cheng Meng\thanks{Center for Applied Statistics, Institute of Statistics and Big Data, Renmin University of China, Beijing, China (\href{mailto:chengmeng@ruc.edu.cn}{chengmeng@ruc.edu.cn}).}
}

\begin{document}
	
	\maketitle

    \begin{abstract}
        This paper focuses on multi-block optimization problems over transport polytopes, which underlie various applications including strongly correlated quantum physics and machine learning. Conventional block coordinate descent-type methods for the general multi-block problems store and operate on the matrix variables directly, resulting in formidable expenditure for large-scale settings. On the other hand, optimal transport problems, as a special case, have attracted extensive attention and numerical techniques that waive the use of the full matrices have recently emerged. However, it remains nontrivial to apply these techniques to the multi-block, possibly nonconvex problems with theoretical guarantees. In this work, we leverage the benefits of both sides and develop novel sampling-based block coordinate descent-type methods, which are equipped with either entropy regularization or Kullback-Leibler divergence. Each iteration of these methods solves subproblems restricted on the sampled degrees of freedom. Consequently, they involve only sparse matrices, which amounts to considerable complexity reductions. We explicitly characterize the sampling-induced errors and establish convergence and asymptotic properties for the methods equipped with the entropy regularization. Numerical experiments on typical strongly correlated electron systems corroborate their superior scalability over the methods utilizing full matrices. The advantage also enables the first visualization of approximate optimal transport maps between electron positions in three-dimensional contexts.
    \end{abstract}
	
	\section{Introduction}
	
	In this work, we consider multi-block optimization problems over transport polytopes as follows:
	\begin{equation}
		\begin{array}{rcl}
			\minimize & (\min) & f(X_1,\ldots,X_N),\\
			\text{subject to} & (\st) & X_i\in\calU(\mba_i,\mbb_i),~i=1,\ldots,N,
		\end{array}
		\label{eqn:multi-block OTP}
	\end{equation}
	where, for any $i\in\{1,\ldots,N\}$ ($1\le N\in\N$), $\mba_i\in\R_+^{m_i}$, $\mbb_i\in\R_+^{n_i}$,
	$$\calU(\mba_i,\mbb_i):=\lrbrace{T\in\R_+^{m_i\times n_i}\mid T\one_{n_i}=\mba_i,~T^\T\one_{m_i}=\mbb_i}$$
	is called the transport polytope. The notation ``$\one_{n}$'' refers to the all-ones vector in $\R^n$. The unknown matrix variables are $X_i\in\R^{m_i\times n_i}$ ($i=1,\ldots,N$). The objective function $f:\bigtimes_{i=1}^N\R^{m_i\times n_i}\to\R$ is assumed to be block Lipschitz smooth over $\bigtimes_{i=1}^N\calU(\mba_i,\mbb_i)$ (for the definition see section \ref{sec:convergence analyses}), possibly nonconvex. Problem \eqref{eqn:multi-block OTP} finds its applications in several fields. For example, in quantum physics, it provides a promising route for treating the elusive strongly correlated electron systems (e.g., transition metal oxides \cite{dagotto2005complexity}), describing the electron-electron correlation explicitly \cite{chen2014numerical,hu2023global}. It can also act as a subproblem in finding the Wasserstein barycenter among several discrete probability distributions \cite{carlier2015numerical}, which has gained popularity so far in statistics \cite{bigot2012consistent} and machine learning \cite{cuturi2014fast}, as well as in label distribution learning \cite{zhao2018label}, which reflects the relative importance of different labels in supervised learning \cite{geng2016label}. 
	
	\par For solving general multi-block optimization problems, block coordinate descent (\bcd)-type methods rank among the top choices. These methods fully exploit the separability of the feasible region, in that each subproblem involves only one variable block and is much easier to solve than the original problem. Representatives of the \bcd-type methods are the block coordinate descent (\bcd) methods \cite{bertsekas2016nonlinear,fercoq2016optimization,wright2015coordinate}, block conditional gradient (\bcg) methods \cite{beck2015cyclic,braun2022conditional}, proximal alternating linearized minimization (\palm) methods \cite{bolte2014proximal,hu2023convergence}, as well as their stochastic versions \cite{beck2015cyclic,chen2019extended,driggs2021stochastic,hertrich2022inertial,lacoste2013block,sun2020efficiency}, where randomness is introduced to the gradient calculations or update order. Nevertheless, for the specific problem \eqref{eqn:multi-block OTP}, all the existing \bcd-type methods store and operate on the matrix variables directly, requiring at least quadratically growing complexities per iteration. This forms formidable memory and computation burdens when $\{m_i\}_{i=1}^N$ and/or $\{n_i\}_{i=1}^N$ are of large magnitude. Taking the aforementioned quantum physics application \cite{hu2023global} for instance, $m_i$ ($=n_i$) stands for the number of grid points and can be of order $10^4$ or $10^5$ even for crude discretization. 
	
	\par When $N$ equals one ($m_i=m$, $n_i=n$) and $f$ is affine, problem \eqref{eqn:multi-block OTP} reduces to the Kantorovich formulation of the classical optimal transport (OT) problem \cite{villani2003topics}. The exploration of this problem dates back to Monge's pioneering work in the 18th century \cite{monge1781memoire}, followed by Kantorovich's relaxation in the 20th century \cite{kantorovich1942transfer}. Since then, a plethora of numerical methods for solving OT problems have been constantly emerging. Traditional ones solve differential equations \cite{benamou2002monge,brenier1997homogenized} or turn to linear programming solvers \cite{pele2009fast,rubner1997earth}, resulting in unacceptable cubic complexities. Nowadays, the widest usage may go to the entropy regularization-based methods \cite{cuturi2013sinkhorn,peyre2019computational}, which allow for the approximations of solutions in $\calO(t_{\max}mn)$ scaling time with the Sinkhorn algorithm \cite{sinkhorn1967concerning}, where $t_{\max}$ is the number of iterations; see section \ref{subsec:regularized ot} for more discussions. In recent years, motivated by the need in large-scale contexts, there have been works dedicated to alleviating the per-iteration quadratic costs by the conventional Sinkhorn algorithm; for example, the low-rank approximation-based \cite{altschuler2019massively} and the entrywise sampling-based \cite{li2023importance} variants of the Sinkhorn algorithm. Remarkably, the latter variant essentially deals with a \textit{restricted} OT problem:
	\begin{equation}
		\min_{X}~~\inner{\hat C,X},~~\st~X\in\calU(\mba,\mbb),~X_{\calI^c}=0,
		\label{eqn:restricted OT}
	\end{equation}
	where $\hat C\in\R^{m\times n}$ is an effective cost matrix (defined later), $X\in\R^{m\times n}$, $\mba\in\R^m$, $\mbb\in\R^n$, $\calI\subseteq\{(j,k)\mid j=1,\ldots,m,~k=1,\ldots,n\}$ contains the indices sampled from the beginning according to some probability distribution related to $\mba$ and $\mbb$, and $\calI^c$ denotes its complementary set. The constraint ``$X_{\calI^c}=0$'' enforces the entries in $X$ indexed by $\calI^c$ to be zero, which distinguishes the algorithm from the well-known stochastic optimization methods. As a result, only $\abs{\calI}$ entries in $X$ get involved in the calculations and updates, leading to a nice scaling when $\abs{\calI}=o(mn)$. However, it remains unclear whether the sampling technique can be adapted to handle the multi-block, possibly nonconvex problem \eqref{eqn:multi-block OTP}, while maintaining favorable convergence properties. One possible way is to integrate the sampling technique into the \bcd-type methods and to solve restricted subproblems like \eqref{eqn:restricted OT} in each iteration. Analyzing the accumulation of errors induced by sampling will then become subtle.
	
	\subsection{Contributions and organization.} 
	We develop in this paper sampling-based \bcd-type methods for problem \eqref{eqn:multi-block OTP}, which are equipped with either entropy regularization or Kullback-Leibler divergence. In particular, importance samplings are performed conforming to the probability distributions associated with the \textit{previous iterates} and each iteration solves subproblems restricted over sampled supports. Consequently, only $o(m_in_i)$ entries in $X_i$ ($i=1,\ldots,N$) take part in the updates and derivatives calculations, which amounts to considerable computational saving in large-scale contexts. 
	
	\par Following the theoretical results about randomized matrix sparsification, we analyze the convergence and asymptotic properties for the methods equipped with entropy regularization. We explicitly characterize the sampling-induced errors and establish upper bounds for the average stationarity violations. The average violation is further shown to vanish in the limit $\sum_{i=1}^N(m_i+n_i)\to+\infty$ (with probability going to 1). Notably, to the best of our knowledge, our work is the first attempt in applying the matrix entrywise sampling technique to multi-block nonconvex settings with theoretical guarantees.
	
	\par We demonstrate the efficiency of the newly designed methods via numerical simulations of typical strongly correlated electron systems. 
	Their better scalability enables the first visualization of the approximate OT maps between electron positions in three-dimensional contexts.
	
	\par The paper is organized as follows. We provide preliminaries in section \ref{sec:notation and preliminary} and elaborate on the algorithmic developments in section \ref{sec:algorithmic development}. Section \ref{sec:convergence analyses} contains our theoretical results, whose proofs are deferred to the appendix. Numerical experiments and results are described in section \ref{sec:numerical experiments}. Finally, we conclude in section \ref{sec:conclusion}.

	\section{Preliminaries}\label{sec:notation and preliminary}
	
	This section offers some preliminaries, including notations, tools from OT, and bibliographical notes on entrywise matrix sparsification. 
	
	\subsection{Notations}
	
	This paper presents scalars, vectors, and matrices by regular-font, bold lower-case, and upper-case letters, respectively. We denote the rounding down operation by ``$\lfloor\cdot\rfloor$''. The notation ``$\one_n$'' stands for the all-ones vector in $\R^n$. The notations ``$\inner{\cdot,\cdot}$'' and ``$\snorm{\cdot}$'' calculate, respectively, the standard inner product and norm of vectors or matrices in the ambient Euclidean space. We use ``$\snorm{\cdot}_2$'' particularly for the 2-norm of matrices. The notation ``$\kappa(\cdot)$'' refers to the spectral condition number of a matrix. We use ``$\Diag(\cdot)$'' to form a diagonal matrix with the input vector. We denote the entries or sub-blocks of vectors by single subscripts (e.g., $\varrho_k$ or $\mba_i$), the sub-blocks of matrices by single subscripts (e.g., $X_i$), and the entries of matrices by double subscripts (e.g., $x_{i,jk}$). A matrix with a set subscript refers to the entries indexed by the set (e.g., $X_{\calI}$). Sometimes, we make abbreviations for the aggregation of the sub-blocks of a matrix (e.g., $X_{\le i}:=(X_1,\ldots,X_i)$, $X_{>i}:=(X_{i+1},\ldots,X_N)$). These abbreviations become null if the index sets in the subscripts are empty. 
	The entrywise product and division of two vectors or matrices are denoted by ``$\odot$'' and ``$\oslash$'', respectively. Univariate functions, such as ``$\exp(\cdot)$'' and ``$\log(\cdot)$'', are extended to vectors and matrices as entrywise operations. 
	
	\par For a multivariate function $g$, $\nabla g$ is the gradient of $g$ at the points where $g$ is differentiable. We add a subscript to indicate the block with respect to which the derivative is taken (e.g., $\nabla_ig$). 
	
	\par We use $\R_{++}$ to denote the set of positive real numbers. Given a set, its measure or cardinality is represented using ``$\abs{\cdot}$''. We use ``$\bigtimes$'' or ``$\times$'' to refer to the Cartesian product of sets or spaces (e.g., $\bigtimes_{i=1}^N\R^{m_i\times n_i}$ or $\R^{m_i}\times\R^{n_i}$), use exponents to represent the Cartesian product of identical sets or spaces (e.g., $(\R^{K\times K})^{N_e-1}$). The complementary set is noted with a superscript ``$c$'' (e.g., $\calI^c$). 
	
	\par When describing algorithms, we use superscripts within brackets to refer to the iteration numbers (e.g., $X_i^{(t)}$). 
	
	\subsection{Negative entropy and Kullback-Leibler divergence}\label{subsec:entropy and KL}
	
	\begin{definition}[\cite{kullback1951information,shannon1948mathematical}]\label{def:entropy and KL divergence}
		For any $T=(t_{ij})\in\R_+^{m\times n}$, its negative entropy is defined as $h(T):=\sum_{ij}t_{ij}(\log t_{ij}-1)$. Given any $T=(t_{ij})$, $T'=(t_{ij}')\in\R_+^{m\times n}$, the Kullback-Leibler (KL) divergence between $T$ and $T'$ is defined as
		\begin{equation}
			\KL(T;T'):=\sum_{i,j}\lrsquare{t_{ij}(\log t_{ij}-\log t_{ij}')-(t_{ij}-t_{ij}')}.
			\label{eqn:KL divergence}
		\end{equation}
		If $t_{ij}>0$ and $t_{ij}'=0$ for some pair $(i,j)$, then $\KL(T;T')=+\infty$.
	\end{definition}
	
	
	\par The negative entropy has been adopted in thermodynamics as a measure of disorder in a system, or a measure of uncertainty in information theory. The KL divergence can be treated as the Bregman distance \cite{bregman1967relaxation} associated with the negative entropy; it has been used as a measure of the disparity between probability distributions \cite{peyre2019computational}. 
	
	\subsection{Entropy regularized optimal transport}\label{subsec:regularized ot}
	
	\par The Kantorovich formulation of the discrete OT problem is in general
	\begin{equation}
		\min_{T}~\inner{W, T},~~\st~T\in\calU(\mb{p},\mb{q}),
		\label{eqn:ot}
	\end{equation}
	where $T=(t_{ij})\in\R^{m\times n}$ is the transport plan, $W=(w_{ij})\in\R^{m\times n}$ is the cost matrix, and $\mb{p}\in\R^m$, $\mb{q}\in\R^n$ are discrete probability distributions. The solution to problem \eqref{eqn:ot} is referred to as the OT plan, which achieves minimal transportation efforts. Nowadays, OT has attracted extensive attention from applications (e.g., \cite{arjovsky2017wasserstein,ma2019optimal,meng2019large,xu2019learning}).
	
	\par The computational complexity of directly solving problem \eqref{eqn:ot} as a linear programming usually grows cubically as $m$ and $n$ increase, which severely hinders the wide applications of OT. To approximate the solution efficiently within certain tolerance, the author of \cite{cuturi2013sinkhorn} adds a (negative) entropy regularization penalty, obtaining
	\begin{equation}
		\min_{T}~\inner{W, T} + \lambda h(T),~~\st~T\one_n=\mb{p},~T^\T\one_m=\mb{q},
		\label{eqn:rot}
	\end{equation}
	where $\lambda>0$ is the regularization parameter. Note that the nonnegativity requirement is unnecessary due to the definition of $h$.
	
	\par On account of the entropy regularizer, problem \eqref{eqn:rot} becomes strongly convex and thus admits a unique optimal solution. Furthermore, as $\lambda\to0$, this optimal solution converges to an optimal solution of problem \eqref{eqn:ot} \cite{peyre2019computational}. 
	Computationally, the dual of problem \eqref{eqn:rot} can be solved by an alternating minimization scheme, known as the Sinkhorn algorithm in the OT community \cite{sinkhorn1967concerning}. The alternating scheme involves only matrix-vector multiplications and entrywise divisions between vectors (see section \ref{subsec:eralm}), 
	particularly suited for GPU executions \cite{cuturi2013sinkhorn}. 
	
	
	\subsection{Entrywise matrix sparsification}\label{subsec:importance sampling}
	
	\par Entrywise matrix sparsification is pioneered by Achlioptas and McSherry \cite{achlioptas2007fast}, later developed in \cite{braverman2021near,drineas2011note,kundu2017recovering}, where sampling-based algorithms are described to select a small number of entries from an input matrix to construct a sparse sketch; the sketch is close to the original one in the operator norm with a high probability guarantee. The entries are sampled following some probability distribution associated with the original matrix. Among others, importance sampling, as a statistical technique, constructs the sampling probability distribution in the spirit of variance reduction \cite{elvira2021advances,liu1996metropolized,liu2008monte,owen2013monte}. Owing to the substantial computational complexity reduction thereby, entrywise matrix sparsification has been adopted in various scenarios, e.g., for computing approximate eigenvectors \cite{achlioptas2007fast}, solving OT problems \cite{li2023importance}, and computing Gromov-Wasserstein distances \cite{li2023efficient}.  
	
	
	\section{Algorithmic developments}\label{sec:algorithmic development}
	
	
	\par In this part, we develop two classes of methods for the multi-block problem over the transport polytopes \eqref{eqn:multi-block OTP}, with the \bcg~and \palm~methods as starting points. In particular, we add entropy regularizers to the subproblems in the \bcg~method, while replacing the proximal term with the KL divergence in the \palm~method. We further equip them with importance sampling-based entrywise matrix sparsification.
	
	\subsection{Entropy regularized alternating linearized minimization}\label{subsec:eralm}
	
	\par In each iteration, the \bcg~method obtains search directions via solving subproblems as
	\begin{equation}
		\min_{X_i}~\inner{C_i^{(t)},X_i-X_i^{(t)}},~\st~X_i\in\calU(\mba_i,\mbb_i),
		\label{eqn:subprob in bcg}
	\end{equation}
	where 
	\begin{equation}
		C_i^{(t)}:=\nabla_if\big(X_{<i}^{(t+1)},X_{\ge i}^{(t)}\big)\in\R^{m_i\times n_i}.
		\label{eqn:subprob cost matrix full}
	\end{equation} 
	Subproblem \eqref{eqn:subprob in bcg} can be identified as an OT problem with $C_i^{(t)}$ being the cost matrix. As mentioned in section \ref{subsec:regularized ot}, solving subproblem \eqref{eqn:subprob in bcg} entails cubic complexities using linear programming methods, which forms an unacceptable computational burden. To this end, 
	we add an entropy regularization and instead resort to 
	\begin{equation}
		\min_{X_i}~\inner{C_i^{(t)},X_i-X_i^{(t)}}+\lambda_i^{(t)}h(X_i),~\st~X_i\one_{n_i}=\mba_i,~X_i^\T\one_{m_i}=\mbb_i
		\label{eqn:subprob in eralm}
	\end{equation}
	in each iteration, where $\lambda_i^{(t)}>0$ is the regularization parameter and $h$ is the negative entropy in Definition \ref{def:entropy and KL divergence}. Consequently, we obtain the entropy regularized alternating linearized minimization (\eralm) method; see Algorithm \ref{alg:eralm}. 
	\begin{algorithm}[!t]
		\caption{The \eralm~method for solving problem \eqref{eqn:multi-block OTP}.}
		\label{alg:eralm}
		\begin{algorithmic}[1]
			\REQUIRE{$X_i^{(0)}\in\R^{m_i\times n_i}$, $\mba_i\in\R^{m_i}$, $\mbb_i\in\R^{n_i}$ ($i=1,\ldots,N$), $t_{\max}\in\N$.}
			\STATE{Set $t:=0$.}
			\WHILE{\textit{certain conditions are not satisfied} \textbf{and} $t<t_{\max}$}
			\FOR{$i=1,\ldots,N$}
			\STATE{Select a regularization parameter $\lambda_i^{(t)}>0$ and a step size $\alpha_i^{(t)}\in(0,1]$.}
			\STATE{Compute $C_i^{(t)}\in\R^{m_i\times n_i}$ as in the formula \eqref{eqn:subprob cost matrix full}.}
			\STATE{Solve subproblem \eqref{eqn:subprob in eralm} or \eqref{eqn:dual of eralm} to obtain $\tilde X_i^{(t+1)}\in\R^{m_i\times n_i}$.}
			\STATE{Update $X_i^{(t+1)}:=(1-\alpha_i^{(t)})X_i^{(t)}+\alpha_i^{(t)}\tilde X_i^{(t+1)}\in\R^{m_i\times n_i}$.}
			\ENDFOR
			\STATE{Set $t:=t+1$.}
			\ENDWHILE
			\ENSURE{Approximate solution $(X_1^{(t)},\ldots,X_N^{(t)})\in\bigtimes_{i=1}^N\R^{m_i\times n_i}$.}
		\end{algorithmic}
	\end{algorithm}
	
	\par 
	Observe that the number of variables in subproblem \eqref{eqn:subprob in eralm} scales quadratically, while the number of equality constraints grows linearly. Therefore, it is more advantageous to work from the dual perspective, especially when $m_i$ and/or $n_i$ is of large magnitude. From \cite{peyre2019computational}, the dual of subproblem \eqref{eqn:subprob in eralm} is
	\begin{equation}
		\min_{\tilde\uu_i,\tilde\vvv_i}~q_i(\tilde\uu_i,\tilde\vvv_i;\lambda_i^{(t)},\Psi_i^{(t)}):=\lambda_i^{(t)}\exp\lrbracket{\frac{\tilde\uu_i}{\lambda_i^{(t)}}}^\T\Psi_i^{(t)}\exp\lrbracket{\frac{\tilde\vvv_i}{\lambda_i^{(t)}}}-\tilde\uu_i^\T\mba_i-\tilde\vvv_i^\T\mbb_i,
		\label{eqn:dual of eralm}
	\end{equation}
	where $\tilde\uu_i\in\R^{m_i}$, $\tilde\vvv_i\in\R^{n_i}$ are the dual variables associated with the equality constraints, and
	\begin{equation}
		\Psi_i^{(t)}:=\exp\lrbracket{-C_i^{(t)}/\lambda_i^{(t)}}\in\R^{m_i\times n_i}
		\label{eqn:kernel of eralm}
	\end{equation}
	is called the kernel matrix. To tackle the block-structured problem \eqref{eqn:dual of eralm}, one natural choice is the \bcd~method, which is also known as the \textit{Sinkhorn algorithm} \cite{sinkhorn1967concerning} in the context of OT. Starting from a given $\tilde\vvv_i^{(t,0)}\in\R^{n_i}$, the Sinkhorn algorithm repeats the following two steps until fulfilling certain criteria:
	\begin{align*}
		\tilde\uu_i^{(t,s+1)}:&=\lambda_i^{(t)}\log\lrbracket{\mba_i\oslash\lrbracket{\Psi_i^{(t)}\exp\big(\tilde\vvv_i^{(t,s)}/\lambda_i^{(t)}\big)}},\\
		\tilde\vvv_i^{(t,s+1)}:&=\lambda_i^{(t)}\log\lrbracket{\mbb_i\oslash\lrbracket{\Psi_i^{(t)\T}\exp\big(\tilde\uu_i^{(t,s+1)}/\lambda_i^{(t)}\big)}},
	\end{align*}
	where $s$ indicates the subiteration number. After letting $\check{\uu}_i^{(t,s)}:=\exp\big(\tilde\uu_i^{(t,s)}/\lambda_i^{(t)}\big)\in\R^{m_i}$ and $\check{\vvv}_i^{(t,s)}:=\exp\big(\tilde\vvv_i^{(t,s)}/\lambda_i^{(t)}\big)\in\R^{n_i}$, the above schemes can be rewritten as
	\begin{equation}
		\check{\uu}_i^{(t,s+1)}:=\mba_i\oslash\lrbracket{\Psi_i^{(t)}\check{\vvv}_i^{(t,s)}},~~\check{\vvv}_i^{(t,s+1)}:=\mbb_i\oslash\lrbracket{\Psi_i^{(t)\T}\check{\uu}_i^{(t,s+1)}},
		\label{eqn:sinkhorn}
	\end{equation}
	involving matrix-vector multiplications and entrywise divisions between vectors, thus favoring high parallel scalability \cite{cuturi2013sinkhorn}. The linear convergence rate of the Sinkhorn algorithm has been established \cite{luo1993convergence}. Further acceleration can be gained via warm starts \cite{xie2020fast}.
	
	
	\subsection{Sampling-based variant of the {\sc eralm} method}
	
	\par The \eralm~method still works on matrix variables and requires calculating $C_i^{(t)}$ \eqref{eqn:kernel of eralm} explicitly. Below, we use a sparse matrix $\hat\Psi_i^{(t)}\in\R^{m_i\times n_i}$ to approximate $\Psi_i^{(t)}=(\psi_{i,jk}^{(t)})$, rendering most of the computational costs dispensable. 
	
	\par The idea of the sparse approximation largely originates from the following multiplicative expression for the unique optimal solution of subproblem \eqref{eqn:subprob in eralm}:
	\begin{equation}
		\tilde X_i^{(t+1,\star)}:=\Diag\lrbracket{\exp\lrbracket{\frac{\tilde\uu_i^{(t,\star)}}{\lambda_i^{(t)}}}}\Psi_i^{(t)}\Diag\lrbracket{\exp\lrbracket{\frac{\tilde\vvv_i^{(t,\star)}}{\lambda_i^{(t)}}}}\in\R^{m_i\times n_i},
		\label{eqn:optimal solution eralm subproblem}
	\end{equation}
	where $\big(\tilde\uu_i^{(t,\star)},\tilde\vvv_i^{(t,\star)}\big)\in\R^{m_i}\times\R^{n_i}$ is an optimal solution of the dual \eqref{eqn:dual of eralm}. The expression \eqref{eqn:optimal solution eralm subproblem} indicates that $\tilde x_{i,jk}^{(t+1,\star)}=0$ whenever $\psi_{i,jk}^{(t)}=0$. For another, it has been theoretically established in \cite{hosseini2022intrinsic} that the solutions of subproblem \eqref{eqn:subprob in bcg} can be sparse. When $f$ is multi-affine, there even exist sparse solutions with $\calO(\sum_{i=1}^N(m_i+n_i))$ nonzero entries to problem \eqref{eqn:multi-block OTP} \cite{hu2023exactness,hu2023global}. These two points together motivate us to compute only a small portion of the entries in $\Psi_i^{(t)}$ to eliminate most of the storage and computation overhead. 
	
	\par For this purpose, it suffices to estimate the sparsity pattern along iterations. In this work, we make an attempt through entrywise sampling \cite{achlioptas2013near,braverman2021near,drineas2011note,kundu2017recovering,li2023efficient}; that is, randomly pick a small portion from $\{(j,k):j=1,\ldots,m_i,~k=1,\ldots,n_i\}$ according to certain sampling probability distribution. Specifically, the entrywise sampling is implemented via Poisson sampling framework following the recent works \cite{braverman2021near, ai2021optimal, li2023importance}, which independently evaluate each index for inclusion in the sampled set. Compared to the sampling with replacement, the Poisson sampling has been shown to provide higher accuracy in some scenarios and is more practical for distributed systems \cite{wang2022sampling}.
	
	\par In light of the importance sampling (see section \ref{subsec:importance sampling}), the optimal sampling probabilities should be $p_{i,jk}^{(t,\star)}\propto \tilde x_{i,jk}^{(t+1,\star)}$ \cite{li2023efficient}. However, $\tilde X_i^{(t+1,\star)}$ is completely unknown without the knowledge of $\tilde\uu_i^{(t,\star)}$ and $\tilde\vvv_i^{(t,\star)}$ (see the formula \eqref{eqn:optimal solution eralm subproblem}). An alternative is to sample conforming to the values in the previous iteration, i.e., $p_{i,jk}^{(t)\prime}\propto x_{i,jk}^{(t)}$. This becomes reasonable when the procedure gets close to the optimum. In addition, to recover entries $(j,k)$ from the optimal sparsity pattern that are missing in $X_i^{(t)}$, in the sense that $x_{i,jk}^{(t)}=0$, we linearly interpolate between $p_{i,jk}^{(t)\prime}$ and the sampling probability proposed by \cite{li2023efficient}, i.e., $p_{i,jk}^{\prime\prime}\propto \sqrt{a_{i,j} b_{i,k}}$. Specifically, 
	\begin{equation}
		p_{i,jk}^{(t)}:=\gamma p_{i,jk}^{(t)\prime}+(1-\gamma)p_{i,jk}^{\prime\prime} = \gamma\frac{x_{i,jk}^{(t)}}{\sum_{j',k'}x_{i,j'k'}^{(t)}}+(1-\gamma)\frac{\sqrt{a_{i,j} b_{i,k}}}{\sum_{j',k'} \sqrt{a_{i,j'} b_{i,k'}}}
		\label{eqn:subsampling probability}
	\end{equation}
	for $j=1,\ldots,m_i$, $k=1,\ldots,n_i$. Here, $\gamma\in[0,1]$ stands for the interpolation factor. 
	Such a shrinkage strategy is widely adopted in the subsampling literature \cite{ma2014statistical,yu2022optimal}.

	\par Given a sampling parameter $n_{s,i}\in\N$, the sparse approximation $\hat\Psi_i^{(t)}=(\hat{\psi}_{i,jk}^{(t)})\in\R^{m_i\times n_i}$ for the kernel matrix $\Psi_i^{(t)}$ is constructed in accordance with the Poisson sampling principle, i.e.,
	\begin{equation}
		\hat{\psi}_{i,jk}^{(t)}:=\left\{\begin{array}{ll}
			\psi_{i,jk}^{(t)}/p_{i,jk}^{(t)\ast}, & \text{with probability}~p_{i,jk}^{(t)\ast} := \min\lrbrace{1, n_{s,i}\cdot p_{i,jk}^{(t)}},\\
			0, & \text{otherwise},
		\end{array}\right. 
		\label{eqn:sampled kernel}
	\end{equation}
	where the denominators ensure the unbiasedness of the random approximation (see Appendix \ref{appsec:convergence seralm}). We denote the sampled set of indices as $\calI_i^{(t)}\subseteq\{(j,k):j=1,\ldots,m_i,~k=1,\ldots,n_i\}$. Note that $\E(|\calI_i^{(t)}|) = \sum_{j,k} p_{i,jk}^{(t)\ast} \le n_{s,i} \sum_{j,k} p_{i,jk}^{(t)} = n_{s,i}$, which indicates that $n_{s,i}$ is an upper bound for the expected number of nonzero entries in $\hat\Psi_i^{(t)}$. Although $|\calI_i^{(t)}|$ fluctuates, it still concentrates near its expectation with high probability \cite{ai2021optimal}, ensuring that the computational cost remains manageable.
	
	\par Replacing $\Psi_i^{(t)}$ in subproblem \eqref{eqn:dual of eralm} with $\hat\Psi_i^{(t)}$, we obtain
	\begin{equation}
		\min_{\tilde\uu_i,\tilde\vvv_i}~q_i(\tilde\uu_i,\tilde\vvv_i;\lambda_i^{(t)},\hat\Psi_i^{(t)}),
		\label{eqn:dual of seralm}
	\end{equation}
	which is the dual of\footnote{The strong duality holds if subproblem \eqref{eqn:subprob of seralm} is feasible with $\calI_i^{(t)}$. The same will be true later, for subproblem \eqref{eqn:subprob of sklalm}.} 
	\begin{equation}
		\begin{array}{cl}
			\min\limits_{X_i} & \inner{\hat C_i^{(t)},X_i-X_i^{(t)}}+\lambda_i^{(t)}h(X_i),\\
			\st & X_i\one_{n_i}=\mba_i,~X_i^\T\one_{m_i}=\mbb_i,~(X_i)_{(\calI_i^{(t)})^c}=0.
		\end{array}
		\label{eqn:subprob of seralm}
	\end{equation}
	Here, $\hat C_i^{(t)}=(\hat c_{i,jk}^{(t)})\in\R^{m_i\times n_i}$ is the \textit{effective cost matrix}, defined as
	\begin{equation}
		\hat c_{i,jk}^{(t)}:=\left\{\begin{array}{ll}
			c_{i,jk}^{(t)}+\lambda_i^{(t)}\log\lrbracket{p_{i,jk}^{(t)\ast}}, & \text{if}~(j,k)\in\calI_i^{(t)},\\
			c_{i,jk}^{(t)}, & \text{otherwise}.
		\end{array}\right.
		\label{eqn:subprob cost matrix sampling}
	\end{equation}
	Note that in $\hat C_i^{(t)}$ only the entries indexed by $\calI_i^{(t)}$ are required to form $\hat\Psi_i^{(t)}$. To solve subproblem \eqref{eqn:dual of seralm} and \eqref{eqn:subprob of seralm}, it suffices to substitute $\Psi_i^{(t)}$ with $\hat\Psi_i^{(t)}$ in the Sinkhorn algorithm \eqref{eqn:sinkhorn}.
	
	\par We call the \eralm~method equipped with sampling the \seralm~method, which is summarized in Algorithm \ref{alg:seralm}. The maximum expected numbers of sampled indices are controlled by $\{n_{s,i}\}_{i=1}^N\subseteq\N$. 
	\begin{algorithm}[!t]
		\caption{The \seralm~method for solving problem \eqref{eqn:multi-block OTP}.}
		\label{alg:seralm}
		\begin{algorithmic}[1]
			\REQUIRE{$X_i^{(0)}\in\R^{m_i\times n_i}$, $\mba_i\in\R^{m_i}$, $\mbb_i\in\R^{n_i}$ ($i=1,\ldots,N$), $\gamma\in[0,1]$, $\{n_{s,i}\}_{i=1}^N\subseteq\N$, $t_{\max}\in\N$.}
			\STATE{Set $t:=0$.}
			\WHILE{\textit{certain conditions are not satisfied} \textbf{and} $t<t_{\max}$}
			\FOR{$i=1,\ldots,N$}
			\STATE{Select a regularization parameter $\lambda_i^{(t)}>0$ and a step size $\alpha_i^{(t)}\in(0,1]$.}
			\STATE{Randomly pick a subset $\calI_i^{(t)}\subseteq\{(j,k):j=1,\ldots,m_i,~k=1,\ldots,n_i\}$ following the Poisson sampling with $P_i^{(t)}=(p_{i,jk}^{(t)})\in\R_+^{m_i\times n_i}$ given by \eqref{eqn:subsampling probability} and $n_{s,i}$.}
			\STATE{Construct $\hat\Psi_i^{(t)}=(\hat{\psi}_{i,jk}^{(t)})\in\R^{m_i\times n_i}$ as in the formula \eqref{eqn:sampled kernel}.}
			\STATE{Solve subproblem \eqref{eqn:dual of seralm} or \eqref{eqn:subprob of seralm} to obtain $\tilde X_i^{(t+1)}\in\R^{m_i\times n_i}$.}
			\STATE{Update $X_i^{(t+1)}:=(1-\alpha_i^{(t)})X_i^{(t)}+\alpha_i^{(t)}\tilde X_i^{(t+1)}\in\R^{m_i\times n_i}$.}
			\ENDFOR
			\STATE{Set $t:=t+1$.}
			\ENDWHILE
			\ENSURE{Approximate solution $(X_1^{(t)},\ldots,X_N^{(t)})\in\bigtimes_{i=1}^N\R^{m_i\times n_i}$.}
		\end{algorithmic}
	\end{algorithm}
	
	\subsection{KL divergence-based alternating linearized minimization}\label{subsec:klalm}
	
	\par We now turn to methods involving a direct penalization of the distance between $X_i$ and $X_i^{(t)}$. In each iteration, the \palm~method solves several proximal linearized subproblems of the form
	\begin{equation}
		\min_{X_i}~\inner{C_i^{(t)},X_i-X_i^{(t)}}+\frac{\mu_i^{(t)}}{2}\snorm{X_i-X_i^{(t)}}^2,~\st~X_i\in\calU(\mba_i,\mbb_i),
		\label{eqn:subprob of palm}
	\end{equation}
	where $\mu_i^{(t)}>0$ refers to the proximal parameter and $C_i^{(t)}$ is defined in equation \eqref{eqn:subprob cost matrix full}. Solving subproblem \eqref{eqn:subprob of palm} is equivalent to projecting $X_i^{(t)}-C_i^{(t)}/\mu_i^{(t)}$ onto $\calU(\mba_i,\mbb_i)$. In this part, we replace the proximal term $\snorm{X_i-X_i^{(t)}}^2/2$ by the KL divergence $\KL(X_i;X_i^{(t)})$, which is the Bregman distance induced by the negative entropy (see section \ref{subsec:entropy and KL}), and obtain
	\begin{equation}
		\min_{X_i}~\inner{C_i^{(t)},X_i-X_i^{(t)}}+\mu_i^{(t)}\KL(X_i;X_i^{(t)}),~\st~X_i\one_{n_i}=\mba_i,~X_i^\T\one_{m_i}=\mbb_i.
		\label{eqn:subprob of klalm}
	\end{equation}
	This is motivated by the Bregman proximal point algorithms for convex optimization \cite{xie2020fast,yang2022bregman} and the Bregman distance-based \palm~methods for nonconvex optimization with relatively smooth objectives \cite{ahookhosh2021multi,li2019provable}. We describe the KL divergence-based alternating linearized minimization (\klalm) method in Algorithm \ref{alg:klalm}. 
	\begin{algorithm}[!t]
		\caption{The \klalm~method for solving problem \eqref{eqn:multi-block OTP}.}
		\label{alg:klalm}
		\begin{algorithmic}[1]
			\REQUIRE{$X_i^{(0)}\in\R^{m_i\times n_i}$, $\mba_i\in\R^{m_i}$, $\mbb_i\in\R^{n_i}$ ($i=1,\ldots,N$), $t_{\max}\in\N$.}
			\STATE{Set $t:=0$.}
			\WHILE{\textit{certain conditions are not satisfied} \textbf{and} $t<t_{\max}$}
			\FOR{$i=1,\ldots,N$}
			\STATE{Select a proximal parameter $\mu_i^{(t)}>0$.}
			\STATE{Compute $C_i^{(t)}\in\R^{m_i\times n_i}$ as in the formula \eqref{eqn:subprob cost matrix full}.}
			\STATE{Solve subproblem \eqref{eqn:subprob of klalm} or \eqref{eqn:dual of klalm} to obtain $X_i^{(t+1)}\in\R^{m_i\times n_i}$.}
			\ENDFOR
			\STATE{Set $t:=t+1$.}
			\ENDWHILE
			\ENSURE{Approximate solution $(X_1^{(t)},\ldots,X_N^{(t)})\in\bigtimes_{i=1}^N\R^{m_i\times n_i}$.}
		\end{algorithmic}
	\end{algorithm}
	
	\par 
	As in the previous subsection, we write out the dual of subproblem \eqref{eqn:subprob of klalm}:
	\begin{equation}
		\min_{\uu_i,\vvv_i}~q_i(\uu_i,\vvv_i;\mu_i^{(t)},\Phi_i^{(t)}),
		\label{eqn:dual of klalm}
	\end{equation}
	where $q_i$ is defined in problem \eqref{eqn:dual of eralm}, $\uu_i\in\R^{m_i}$, $\vvv_i\in\R^{n_i}$ are the dual variables associated with the equality constraints and the kernel matrix
	\begin{equation}
		\Phi_i^{(t)}:=\exp\lrbracket{-C_i^{(t)}/\mu_i^{(t)}}\odot X_i^{(t)}\in\R^{m_i\times n_i}.
		\label{eqn:kernel of klalm}
	\end{equation}
	The Sinkhorn algorithm \eqref{eqn:sinkhorn} thus still applies, yet with $\Psi_i^{(t)}$ changed to $\Phi_i^{(t)}$. 
	
	
	\subsection{Sampling-based variant of the {\sc klalm} method}
	\par Since subproblem \eqref{eqn:subprob of klalm} in the \klalm~method admits the following multiplicative expression for the unique optimal solution:
	\begin{equation}
		X_i^{(t+1,\star)}:=\Diag\lrbracket{\exp\lrbracket{\frac{\uu_i^{(t,\star)}}{\mu_i^{(t)}}}}\Phi_i^{(t)}\Diag\lrbracket{\exp\lrbracket{\frac{\vvv_i^{(t,\star)}}{\mu_i^{(t)}}}}\in\R^{m_i\times n_i},
		\label{eqn:optimal solution bpalm subproblem}
	\end{equation}
	where $\big(\uu_i^{(t,\star)},\vvv_i^{(t,\star)}\big)\in\R^{m_i}\times\R^{n_i}$ is an optimal solution of the dual \eqref{eqn:dual of klalm}, we could likewise employ sparse approximation on $\Phi_i^{(t)}$ for better scalability. We define the sparse approximation $\hat\Phi_i^{(t)}=(\hat{\varphi}_{i,jk}^{(t)})\in\R^{m_i\times n_i}$ for $\Phi_i^{(t)}=(\varphi_{i,jk}^{(t)})$ as
	\begin{equation}
		\hat{\varphi}_{i,jk}^{(t)}:=\left\{\begin{array}{ll}
			\varphi_{i,jk}^{(t)}/p_{i,jk}^{(t)\ast}, & \text{with probability}~p_{i,jk}^{(t)\ast},\\
			0, & \text{otherwise},
		\end{array}\right.
		\label{eqn:sampled kernel sklalm}
	\end{equation}
	where $p_{i,jk}^{(t)\ast}$ is given as in equation \eqref{eqn:sampled kernel}. Replacing $\Phi_i^{(t)}$ in subproblem \eqref{eqn:dual of klalm} with $\hat\Phi_i^{(t)}$, we obtain 
	\begin{equation}
		\min_{\uu_i,\vvv_i}~q_i(\uu_i,\vvv_i;\mu_i^{(t)},\hat\Phi_i^{(t)}),
		\label{eqn:dual of sklalm}
	\end{equation}
	which is the dual of
	\begin{equation}
		\begin{array}{cl}
			\min\limits_{X_i} & \inner{\hat C_i^{(t)},X_i-X_i^{(t)}}+\mu_i^{(t)}\KL(X_i;X_i^{(t)}),\\
			\st & X_i\one_{n_i}=\mba_i,~X_i^\T\one_{m_i}=\mbb_i,~(X_i)_{(\calI_i^{(t)})^c}=0.
		\end{array}
		\label{eqn:subprob of sklalm}
	\end{equation}
	The Sinkhorn algorithm \eqref{eqn:sinkhorn} then applies with $\Psi_i^{(t)}$ replaced by $\hat\Phi_i^{(t)}$.
	
	\par While everything seems to go smoothly, we shall point out some distinctions: samplings over iterations will not improve the supports to be optimized and can even result in infeasible subproblems. Recalling the definition \eqref{eqn:kernel of klalm} of $\Phi_i^{(t)}$, one knows from relation \eqref{eqn:optimal solution bpalm subproblem} that $x_{i,jk}^{(t)}=0$ implies $x_{i,jk}^{(t+1,\star)}=0$. If the subproblems are exactly solved and we perform samplings in two successive iterations (say, $t$ and $t+1$), the $i$th support to be optimized in the $(t+1)$th iteration is $\calI_i^{(t-1)}\cap\calI_i^{(t)}$, which is a subset of $\calI_i^{(t-1)}$. In implementation, we adopt the Sinkhorn algorithm to inexactly solve the subproblems. Its iterative schemes \eqref{eqn:sinkhorn} also imply the inheritance of zero entries. In both cases, samplings over iterations do not improve the supports to be optimized and can lead to infeasibility due to repeated intersections.
	To this end, we choose to perform sampling only in some critical iteration, say $\hat t\in\N$. For $t<\hat t$, the kernel matrices $\Phi_i^{(t)}$ are fully computed. With properly chosen $\hat t$, we can expect $X_i^{(\hat t)}$ to capture well the sparsity pattern. Then, for $t>\hat t$, the selected indices are fixed and no sampling occurs. 
	
	\par We summarize the \klalm~method with sampling (the \sklalm~method) in Algorithm \ref{alg:sbpalm}.
	\begin{algorithm}[!t]
		\caption{The \sklalm~method for solving problem \eqref{eqn:multi-block OTP}.}
		\label{alg:sbpalm}
		\begin{algorithmic}[1]
			\REQUIRE{$X_i^{(0)}\in\R^{m_i\times n_i}$, $\mba_i\in\R^{m_i}$, $\mbb_i\in\R^{n_i}$ ($i=1,\ldots,N$), $\gamma\in[0,1]$, $\{n_{s,i}\}_{i=1}^N\subseteq\N$, $\hat t$, $t_{\max}\in\N$.}
			\STATE{Set $t:=0$.}
			\WHILE{\textit{certain conditions are not satisfied} \textbf{and} $t<t_{\max}$}
			\FOR{$i=1,\ldots,N$}
			\STATE{Select a proximal parameter $\mu_i^{(t)}>0$.}
			\IF{$t=\hat t$}
			\STATE{Randomly pick a subset $\calI_i^{(t)}\subseteq\{(j,k):j=1,\ldots,m_i,~k=1,\ldots,n_i\}$ following the Poisson sampling with $P_i^{(t)}=(p_{i,jk}^{(t)})\in\R_+^{m_i\times n_i}$ given by \eqref{eqn:subsampling probability} and $n_{s,i}$.}
			\ENDIF
			\IF{$t<\hat t$}
			\STATE{Let $\hat\Phi_i^{(t)}:=\Phi_i^{(t)}$ defined in the formula \eqref{eqn:kernel of klalm}.}
			\ELSE
			\STATE{Construct $\hat\Phi_i^{(t)}\in\R^{m_i\times n_i}$ as in the formula \eqref{eqn:sampled kernel sklalm} with $\calI_i^{(\hat t)}$ and $P_i^{(\hat t)}$.}
			\ENDIF
			\STATE{Solve subproblem \eqref{eqn:dual of sklalm} or \eqref{eqn:subprob of sklalm} to obtain $X_i^{(t+1)}\in\R^{m_i\times n_i}$.}
			\ENDFOR
			\STATE{Set $t:=t+1$.}
			\ENDWHILE
			\ENSURE{Approximate solution $(X_1^{(t)},\ldots,X_N^{(t)})\in\bigtimes_{i=1}^N\R^{m_i\times n_i}$.}
		\end{algorithmic}
	\end{algorithm}
	
	\subsection{Computational complexities}
	
	\par We compare the single-iteration computational complexities of the \eralm, \seralm, \klalm, and \sklalm~methods. In particular, we focus on the cost of computing the (sparse) kernel matrices, performing importance sampling, and the subiterations within the Sinkhorn algorithm \eqref{eqn:sinkhorn}. For a summary, see Table \ref{tab:complexities}, where $s_{\max}\in\N$ is the maximum subiteration number, and we assume $m_i\equiv m\in\N$, $n_i\equiv n\in\N$, $n_{s,i}\equiv n_s\in\N$ ($i=1,\ldots,N$) for better readability. 
	\begin{table}[!t]
		\centering
		\caption{A comparison of computational complexities. }
		\label{tab:complexities}
		\resizebox{.6\linewidth}{!}{\begin{tabular}{c|c|c}
				\toprule
				
				\tabincell{c}{Ingredients\\in one iteration} & {\large\eralm} & {\large\klalm} \\\midrule
				
				Kernel matrices & $Nmn$ entries & $Nmn$ entries \\\midrule[.01pt]
				
				Subiterations & $s_{\max}\times Nmn$ & $s_{\max}\times Nmn$ \\
				
				
				\midrule\midrule
				
				\tabincell{c}{Ingredients\\in one iteration} & {\large\seralm} & {\large\sklalm} \\\midrule
				
				Sampling & $\calO(Nmn)$ & $\calO(Nmn)$ ($t=\hat t$) \\\midrule[.01pt]
				
				Kernel matrices & $Nn_s$ entries & \tabincell{c}{$Nmn$ entries ($t<\hat t$)\\$Nn_s$ entries ($t\ge\hat t$)} \\ \midrule[.01pt]
				
				Subiterations & $s_{\max}\times Nn_s$ & \tabincell{c}{$s_{\max}\times Nmn$ ($t<\hat t$)\\$s_{\max}\times Nn_s$ ($t\ge\hat t$)} \\
				
				
				\bottomrule
		\end{tabular}}
	\end{table}
	
	\par Armed with warm starts, the Sinkhorn algorithm usually terminates after few subiterations. As such, the \sklalm~method enjoys the lowest complexity per iteration when $t>\hat t$. Given $n_s\sim(m+n)^{1+\tau}$ with $\tau\in(0,1)$, this advantage becomes more evident as $\tau$ tends to $0$ or $m+n$ goes to $+\infty$.
	
	\section{Convergence analysis}\label{sec:convergence analyses}
	
	\par In this section, we present the convergence and asymptotic properties of both the \eralm~and \seralm~methods. Since the KL divergence lacks local Lipschitz smoothness, the analysis for the \klalm~and \sklalm~methods can be highly nontrivial and is left over as future work. Nevertheless, the theoretical results obtained for the \eralm~and \seralm~methods already deserve a whistle. To the best of our knowledge, our work is the first attempt in incorporating randomized matrix sparsification into numerical methods for multi-block nonconvex optimization, while establishing the convergence and asymptotic properties. 
	
	\par All the results are obtained by assuming that the subproblems are exactly solved. Furthermore, the subproblems of the \seralm~method are assumed to be feasible in their primal forms in all iterations, so that the strong duality holds. Although at present no analysis is present on the conditions under which the latter assumption is fulfilled, we remark that sampling parameters without careful selection (say, $n_{s,i}=\lfloor (m_i+n_i)^{1.5}\rfloor$) already meets the demand in numerical simulations (see section \ref{sec:numerical experiments}). 
	
	\par To characterize the stationarity violation for problem \eqref{eqn:multi-block OTP}, we define the residual functions: for any $X:=(X_1,\ldots,X_N)\in\bigtimes_{i=1}^N\R^{m_i\times n_i}$ and $i=1,\ldots,N$,
	\begin{equation}
		R_i(X):=\max_{T\in\calU(\mba_i,\mbb_i)}\inner{\nabla_if(X),X_i-T}.
		\label{eqn:residual function}
	\end{equation}
	Moreover, let $R:=\sum_{i=1}^NR_i$. It is not hard to verify that $R(X)\ge0$ for any $X\in\bigtimes_{i=1}^N\calU(\mba_i,\mbb_i)$ and that $X\in\bigtimes_{i=1}^N\calU(\mba_i,\mbb_i)$ is a Karush-Kuhn-Tucker point of problem \eqref{eqn:multi-block OTP} if and only if $R(X)=0$ holds. For the iterate $X^{(t)}$ generated by the \eralm~or \seralm~method, $R(X^{(t)})$ can thus characterize the stationarity violation at $X^{(t)}$. 
	
	\par We assume Lipschitz smoothness of $f$ over the feasible region, which holds automatically, e.g., for the quantum physics application (see section \ref{subsec:prob descrip}).
	
	\begin{assumption}\label{assume:Lip grad}
		The gradient of the function $f$ is Lipschitz continuous over $\bigtimes_{i=1}^N\calU(\mba_i,\mbb_i)$, i.e., there exists an $L\ge0$ such that, for $i=1,\ldots,N$,
		$$\snorm{\nabla_if(X)-\nabla_if(X')}\le L\snorm{X-X'}~~\text{for all}~X,X'\in\bigtimes_{i=1}^N\calU(\mba_i,\mbb_i).$$
	\end{assumption}
	
	
	
	
	\par We first give the convergence and asymptotic properties of the \eralm~method. The proofs of the theorem and corollary below can be found in Appendix \ref{appsec:convergence eralm}. 
	
	\begin{theorem}\label{theorem:convergence of eralm}
		Suppose that Assumption \ref{assume:Lip grad} holds. Let $\{X^{(t)}\}$ be the iterate sequence generated by the \eralm~method when
			\begin{equation}
				t_{\max}\ge\frac{f(X^{(0)})-\underline{f}}{2\bar d^2LN(2\sqrt{N}+1)},\quad\alpha_1^{(t)}=\cdots=\alpha_N^{(t)}\equiv\alpha:=\frac{1}{\bar d}\sqrt{\frac{f(X^{(0)})-\underline{f}}{2LN(2\sqrt{N}+1)t_{\max}}},
				\label{eqn:optimal step size full}
			\end{equation}
			and $\lambda_1^{(t)}=\cdots=\lambda_N^{(t)}\equiv\lambda$ for $0\le t\le t_{\max}$ and subproblems \eqref{eqn:subprob in eralm} are exactly solved, where $\underline{f}\in\R$ is less than or equal to the optimal value of problem \eqref{eqn:multi-block OTP}, 
		$$d_i:=\min\lrbrace{\sqrt{m_i}\norm{\mba_i}_\infty,\sqrt{n_i}\norm{\mbb_i}_\infty}~(i=1,\ldots,N),\quad\bar d:=\max_{i=1}^Nd_i.$$
		Then
		\begin{equation}\label{eq:thm1}
			0\le\frac{1}{t_{\max}}\sum_{t=0}^{t_{\max}-1}R(X^{(t)})\le2\bar d(2N+1)\sqrt{\frac{L(f(X^{(0)})-\underline{f})}{t_{\max}}}+N\lambda\bar h,
		\end{equation}
		where $\bar h:=-\min_{i=1}^Nh(\mba_i\mbb_i^\T)$.
	\end{theorem}
	
	\begin{remark}
		\par In Theorem \ref{theorem:convergence of eralm}, as well as the subsequent Corollary \ref{corollary:asymptotic full}, Theorem \ref{theorem:convergence of seralm} and Corollary \ref{corollary:asymptotic sample}, the requirements $\alpha_1^{(t)}=\cdots=\alpha_N^{(t)}$ and $\lambda_1^{(t)}=\cdots=\lambda_N^{(t)}$ are assumed for better readability. The theorems and corollaries can be extended without too much difficulty to the case where $\alpha_i^{(t)}$ and $\lambda_i^{(t)}$ vary across $i\in\{1,\ldots,N\}$ and $0\le t\le t_{\max}$.
	\end{remark}
	
	\par The errors related to the entropy terms are inevitable because the objective function in the subproblem of the \eralm~method is not a local approximation for $f$. Nevertheless, the right-hand side of inequality \eqref{eq:thm1} vanishes in the limit $\sum_{i=1}^N(m_i+n_i)\to+\infty$ after choosing proper $t_{\max}$ and $\lambda$ and imposing Assumption \ref{assumption:discretization} below. 
	This is of particular importance for large-scale applications.
	
	
	
	\begin{assumption}
		\begin{enumerate}[label=(\roman*),itemsep=-0.05cm]
			\item There exists an $\underline{f}\in\R$ such that the optimal value of problem \eqref{eqn:multi-block OTP} is lower bounded by $\underline{f}$ for any $\{m_i\}_{i=1}^N$, $\{n_i\}_{i=1}^N\subseteq\N$.
			
			\item There exists a $q>0$ such that, for any $\{m_i\}_{i=1}^N$, $\{n_i\}_{i=1}^N\subseteq\N$, $\mba_i^\T\one_{m_i}=\mbb_i^\T\one_{n_i}=1$, $\max_ja_{i,j}\le q\cdot\min_ja_{i,j}$, and $\max_kb_{i,k}\le q\cdot\min_kb_{i,k}$.
			
			\item There exists a $\theta\ge0$ such that, for any $\{m_i\}_{i=1}^N$, $\{n_i\}_{i=1}^N\subseteq\N$, the block Lipschitz constant $L=\calO(\sum_{i=1}^N(m_i+n_i)^\theta)$.
			
			\item There exists a $\xi\ge0$ such that $\max_{i=1}^N(m_i+n_i)/\min_{i=1}^N(m_i+n_i)\le\xi$ for any $\{m_i\}_{i=1}^N$, $\{n_i\}_{i=1}^N\subseteq\N$.
		\end{enumerate}
		\label{assumption:discretization}
	\end{assumption}
	
	\begin{remark}
		\par Items (i) and (ii) in Assumption \ref{assumption:discretization} are motivated by the application of interest, where a continuous problem arises and one seeks to solve its discretized version. For example, in the application of strongly correlated quantum physics (see section \ref{subsec:prob descrip} and also \cite{chen2014numerical,hu2023global}), the discretized problem has a natural objective lower bound $\underline{f}=0$ due to the nonnegativity of energy, $\mba_i$ and $\mbb_i$ are discretization of the so-called single-particle density, whose integral is a prescribed constant. Incidentally, up to normalizing $f$, we assume the total mass of marginals to be 1 in item (ii). For simplicity in theoretical analysis, we adopt items (iii) and (iv), which can be further relaxed to some extent.
	\end{remark}
	\begin{corollary}\label{corollary:asymptotic full}
		Suppose that Assumptions \ref{assume:Lip grad} and \ref{assumption:discretization} (i)-(iii) hold. Let $\{X^{(t)}\}$ be the iterate sequence generated by the \eralm~method when
		\begin{align*}
			t_{\max}&\ge\max\lrbrace{\Omega\lrbracket{\sum_{i=1}^N(m_i+n_i)^\eta},\frac{f(X^{(0)})-\underline{f}}{2\bar d^2LN(2\sqrt{N}+1)}},\quad f(X^{(0)})\le M,\\
			\alpha_1^{(t)}&=\cdots=\alpha_N^{(t)}\equiv\alpha,\quad\lambda_1^{(t)}=\cdots=\lambda_N^{(t)}\equiv\lambda=o\lrbracket{\frac{1}{\sum_{i=1}^N\log m_in_i}}
		\end{align*}
		for $0\le t\le t_{\max}$ and subproblems \eqref{eqn:subprob in eralm} are exactly solved, where $\eta(>\theta)$ and $M$ are constants independent from $\{m_i\}_{i=1}^N$ and $\{n_i\}_{i=1}^N$ and $\alpha$ is defined in equation \eqref{eqn:optimal step size full}. 
		Then $\sum_{t=0}^{t_{\max}-1}R(X^{(t)})/t_{\max}\to0$ as $\sum_{i=1}^N(m_i+n_i)\to+\infty$. 
	\end{corollary}
	
	\par With sampling, the analysis for the \seralm~method is much more involved, in that the entrywise matrix sparsification results in additional errors. 
	We establish the convergence and asymptotic results with the help of the theory of randomized matrix sparsification as well as the following Assumption \ref{assumption:kernel and sampling}; details are provided in Appendix \ref{appsec:convergence seralm}. 
	
	\begin{assumption}\label{assumption:kernel and sampling}
		
		\begin{enumerate}[label=(\roman*),itemsep=-0.05cm]
			\item There exist constants $\nu\in(1/2,1]$, $c_1$, $c_2$, $\hat c_2>0$ such that, for any $0\le t\le t_{\max}$ and $i=1,\ldots,N$, 
			$$\snorm{\Psi_i^{(t)}}_2\ge \frac{(m_i+n_i)^\nu}{c_1},\quad\kappa(\Psi_i^{(t)})\le c_2,\quad\kappa(\hat\Psi_i^{(t)})\le\hat c_2.$$ 
			
			\item The interpolation factor $\gamma$ is less than 1 and there exists an $\eps>0$ such that, for $i=1,\ldots,N$, 
			$$\frac{1}{\max_{j,k,t} p_{i,jk}^{(t)}} \ge n_{s,i}\ge \frac{8(m_i+n_i)^{1-2\nu}\log^4(m_i+n_i)}{(1-\gamma)w_i\log^4(1+\eps)},$$ 
			where $w_i:=\min_{j,k}\sqrt{a_{i,j}b_{i,k}}/\sum_{j',k'}\sqrt{a_{i,j'}b_{i,k'}}$.
		\end{enumerate}
	\end{assumption}
	
	\begin{remark}\label{remark:justification on sample size}
		\par Assumption \ref{assumption:kernel and sampling} is adopted for simplicity in theoretical analysis. In particular, since $w_i\le1/m_in_i$, item (ii) implies the following lower bound for $n_{s,i}$:
			$$\frac{8(m_i+n_i)^{1-2\nu}\log^4(m_i+n_i)}{(1-\gamma)w_i\log^4(1+\eps)}\ge\frac{8}{(1-\gamma)\log^4(1+\eps)}\frac{m_in_i}{(m_i+n_i)^{2\nu-1}}\log^4(m_i+n_i),$$
			which is of lower order than $m_in_i$ because $\nu\in(1/2,1]$. The condition $n_{s,i}p_{i,jk}^{(t)}\le 1$ is a convention prevalent in the literature of Poisson sampling \cite{wang2021comparative, wang2022sampling}. Note that under Assumption \ref{assumption:discretization} (ii), the entries in $\mba_i$ or $\mbb_i$ are of the same order. Therefore,  $w_i=\Theta(1/m_in_i)$ and $p_{i,jk}^{(t)} = \Theta(1/m_in_i)$ after choosing a proper value for $\gamma$. Item (ii) then holds if we choose $n_{s,i}=\Theta(m_in_i\log^4(m_i+n_i)/(m_i+n_i)^{2\nu-1})$ and the values of $\eps$ and $\gamma$ are independent from $\{m_i\}_{i=1}^N$ and $\{n_i\}_{i=1}^N$.
	\end{remark}

	\begin{theorem}\label{theorem:convergence of seralm}
		Suppose that Assumption \ref{assume:Lip grad} holds. Let $\{X^{(t)}\}$ be the iterate sequence generated by the \seralm~method when 
			$$t_{\max}\ge\frac{f(X^{(0)})-\underline{f}}{2\bar d^2LN(2\sqrt{N}+1)},$$
			$\alpha_1^{(t)}=\cdots=\alpha_N^{(t)}\equiv\alpha$, and $\lambda_1^{(t)}=\cdots=\lambda_N^{(t)}\equiv\hat\lambda$ for $0\le t\le t_{\max}$, subproblems \eqref{eqn:subprob of seralm} are feasible and exactly solved, and Assumption \ref{assumption:kernel and sampling} is fulfilled, where $\underline{f}\in\R$ is less than or equal to the optimal value of problem \eqref{eqn:multi-block OTP} and $\alpha$ is defined in equation \eqref{eqn:optimal step size full}. Then, for any $m_i+n_i>\max\{152,e^{\sqrt{c_3}}\}~(i=1,\ldots,N)$, $\zeta>0$, and $\iota>0$, with probability no less than 
		$$\prod_{i=1}^N\lrbrace{\lrsquare{1-2\exp\lrbracket{-\frac{16\zeta^2}{\eps^4}\log^4(m_i+n_i)}}\lrsquare{1-\exp\lrbracket{-2\iota^2m_in_i}}}^{t_{\max}},$$ 
		there holds
		\begin{align}
			0&\le \frac{1}{t_{\max}}\sum_{t=0}^{t_{\max}-1}R(X^{(t)})\le2\bar d(2N+1)\sqrt{\frac{L(f(X^{(0)})-\underline{f})}{t_{\max}}}+2N\hat\lambda\bar h\label{eq:thm2}\\
			&+\hat\lambda\bar d \sum_{i=1}^N\sqrt{n_{s,i}+\iota\cdot m_in_i}\log\frac{1}{(1-\gamma)w_i  n_{s,i}}+\hat\lambda\sum_{i=1}^N\frac{\hat c_2c_3}{\log^2(m_i+n_i)-c_3},
			\nonumber
		\end{align}
		where $c_3:=c_1(1+\eps+\zeta)\log^2(1+\eps)$. 
	\end{theorem}
	
	
	
	
	\begin{corollary}\label{corollary:asymptotic sample}
		Suppose that Assumptions \ref{assume:Lip grad} and \ref{assumption:discretization} hold. Let $\{X^{(t)}\}$ be the iterate sequence generated by the \seralm~method when
		\begin{align*}
			t_{\max}&=\Theta\lrbracket{\sum_{i=1}^N(m_i+n_i)^\eta}~\text{satisfying}~t_{\max}\ge\frac{f(X^{(0)})-\underline{f}}{2\bar d^2LN(2\sqrt{N}+1)},\quad f(X^{(0)})\le M,\\
			\alpha_1^{(t)}&=\cdots=\alpha_N^{(t)}\equiv\alpha,~n_{s,i}=\Theta\lrbracket{\frac{m_in_i}{(m_i+n_i)^{2\nu-1}}\log^4(m_i+n_i)},\\
			\lambda_1^{(t)}&=\cdots=\lambda_N^{(t)}\equiv\hat\lambda=o\lrbracket{\frac{1}{\sum_{i=1}^N\sqrt{m_in_i}\log(m_i+n_i)}}
		\end{align*}
		for $0\le t\le t_{\max}$, subproblems \eqref{eqn:subprob of seralm} are feasible and exactly solved, and Assumption \ref{assumption:kernel and sampling} is fulfilled, where $c_1$, $c_2$, $\hat c_2$, $\eps$, $\eta(>\theta)$, $\gamma$, $\nu$, $\xi$, and $M$ are independent from $\{m_i\}_{i=1}^N$ and $\{n_i\}_{i=1}^N$ and $\alpha$ is defined in equation \eqref{eqn:optimal step size full}. Then $\sum_{t=0}^{t_{\max}-1}R(X^{(t)})/t_{\max}\to0$ as $\sum_{i=1}^N(m_i+n_i)\to+\infty$ with probability going to $1$.
	\end{corollary}
	
	\section{Numerical experiments}\label{sec:numerical experiments}
	
	\par We demonstrate the efficiency of the proposed methods via numerical results on model one/two/three-dimensional strongly correlated electron systems. We first describe the related optimization problem of the form \eqref{eqn:multi-block OTP} mathematically and provide experimental details, including systems to be simulated and default algorithmic settings. Then numerical comparisons are conducted among the \palm~method \cite{bolte2014proximal,hu2023convergence} and the newly designed four methods on the one-dimensional systems. We integrate those with favorable performances into a cascadic multigrid optimization framework for the simulations of two/three-dimensional systems. A first visualization of approximate OT maps between electron positions in three-dimensional contexts is provided. Finally, we test the scalability of methods with respect to the problem size as well as the number of variable blocks.
	
	\subsection{Problem description}\label{subsec:prob descrip}
	
	From the strong-interaction limit of density functional theory \cite{friesecke2023strong}, the strongly correlated quantum systems in the strictly correlated regime can be well understood by solving the multimarginal optimal transport problems with Coulomb cost (MMOT) \cite{alfonsi2021approximation,alfonsi2022constrained,benamou2016numerical,buttazzo2012optimal,chen2014numerical,cotar2013density,friesecke2022genetic,khoo2019convex,khoo2020semidefinite,mendl2013kantorovich}: 
		\begin{equation}
			\begin{array}{cl}
				\min\limits_{\pi} & \displaystyle\int_{(\R^d)^{N_e}}c_{\ee}(\rr_1,\ldots,\rr_{N_e})\dd\pi(\rr_1,\ldots,\rr_N), \\
				\st & \displaystyle\int_{(\R^d)^{N_e-1}}\dd\pi(\rr_1,\ldots,\rr_{i-1},\cdot,\rr_{i+1},\ldots,\rr_{N_e})=\frac{\rho}{N_e},~~i=1,\ldots,N_e.
			\end{array}
			\label{eqn:MMOT}
		\end{equation}
		Here, $d\in\{1,2,3\}$ is the system dimension, $N_e\in\N$ is the number of electrons, $\rr_i\in\R^d$ refers to the position of the $i$th electron ($i=1,\ldots,N_e$), $\rho:\R^d\to\R_+$ is the single-particle density, satisfying $\int\rho=N_e$, $c_{\ee}(\rr_1,\ldots,\rr_{N_e}):=\sum_{1\le i<j}1/\norm{\rr_i-\rr_j}$ stands for the $N$-particle Coulomb potential, and $\pi$ is a joint probability measure of $N_e$ electron positions. 
		
		\par Since the dimension of the search space in the MMOT \eqref{eqn:MMOT} scales exponentially with the number of electrons, one could adopt a Monge-like ansatz \cite{chen2014numerical,hu2023global}, which characterizes the electron-electron couplings explicitly, to transform the MMOT into the following problem:
		\begin{align}
			&\min\limits_{\{\gamma_i\}_{i=2}^{N_e}}\hspace{-3mm} && \displaystyle\sum_{i=2}^{N_e}\int_{(\R^d)^2}\frac{\rho(\rr)\gamma_i(\rr,\rr')}{\norm{\rr-\rr'}}\dd\rr\dd\rr'+\sum_{2\le i<j}\int_{(\R^d)^3}\frac{\rho(\rr)\gamma_i(\rr,\rr')\gamma_j(\rr,\rr'')}{\norm{\rr'-\rr''}}\dd\rr\dd\rr'\dd\rr'', \label{eqn:MMOT under the Monge-like ansatz}\\
			&\quad\st && \displaystyle\int_{\R^d}\gamma_i(\cdot,\rr_i)\dd\rr_i=1,~~\int_{\R^d}\gamma_i(\rr_1,\cdot)\rho(\rr_1)\dd\rr_1=\rho,~~\gamma_i\ge0,~~i=2,\ldots,N_e.\nonumber
		\end{align}
		Here, $\gamma_i:\R^d\times\R^d\to\R_+$ encodes the coupling between the first and $i$th electrons, $\rho\gamma_i$ can be understood as the joint probability density between their positions ($i=2,\ldots,N_e$). 
		
		\par To numerically solve problem \eqref{eqn:MMOT under the Monge-like ansatz}, we confine the integral domain to some bounded $\Omega\subseteq\R^d$ and adopt finite elements-like discretization. In particular, we first define a mesh $\calT=\{e_k\}_{k=1}^K$ to divide $\Omega$ into $K$ non-overlapping elements, i.e., $\cup_{k=1}^Ke_k=\Omega$ and $e_k\cap e_{k'}=\emptyset$ whenever $k\ne k'$. Then we use a finite summation of Dirac measures to approximate $\rho$ as $\rho\approx\sum_{k=1}^K\varrho_k\delta_{\d_k}$, where $\varrho_k:=\int_{e_k}\rho$, $\d_k\in\R^d$ is the barycenter of the element $e_k$ ($k=1,\ldots,K$). Let $\brho:=[\varrho_1,\ldots,\varrho_K]\in\R_{++}^K$ and $\Lambda:=\Diag(\brho)\in\R^{K\times K}$. The two-particle Coulomb potential and couplings are respectively discretized into $K\times K$ matrices $C=(c_{kl})\in\R^{K\times K}$ and $X_i=(x_{i,kl})_{kl}\in\R^{K\times K}$ ($i=2,\ldots,N_e$), where, for $i=2,\ldots,N_e$ and $k,l=1,\ldots,K$,
		$$c_{kl}:=\left\{\begin{array}{ll}
			\norm{\d_k-\d_l}^{-1}, & \text{if}~k\ne l, \\
			0, & \text{otherwise},
		\end{array}\right.\quad
		x_{i,kl}:=\frac{1}{\abs{e_k}}\int_{e_k}\int_{e_l}\gamma_i(\rr,\rr')\dd\rr'\dd\rr.$$
		Note that the diagonal entries in $C$ are set to 0 to avoid numerical instability. We impose the following extra constraints on $\{X_i\}_{i=2}^{N_e}$ to maintain the model equivalence:
		\begin{equation}
			\trace(X_i)=0,~i=2,\ldots,N_e;~\inner{X_i,X_j}=0,~i,j=2,\ldots,N_e:i\ne j.
			\label{eqn:extra constraints}
		\end{equation}
		Intuitively, the constraints \eqref{eqn:extra constraints} exclude the cases where two electrons collide. After penalizing the constraints \eqref{eqn:extra constraints} in $\ell_1$ form and transforming $X_i$ to $Y_i:=\Lambda X_i\in\R^{K\times K}$ ($i=2,\ldots,N_e$), we obtain a multi-block optimization problem over the transport polytopes of the form \eqref{eqn:multi-block OTP} (with $N=N_e-1$):
		\begin{equation}
			\begin{array}{cl}
				\min\limits_{\{Y_i\}_{i=2}^{N_e}} & \displaystyle\sum_{i=2}^{N_e}\inner{Y_i,C+\beta \Lambda^{-1}}+\sum_{2\le i<j}\inner{Y_i,\Lambda^{-1}Y_jC+\beta\Lambda^{-2}Y_j}, \\
				\st & Y_i\in\calU(\brho,\brho)\subseteq\R^{K\times K},~~i=2,\ldots,N_e.
			\end{array}
			\label{eqn:MPGCC l1}
		\end{equation}
		The matrix variable $Y_i$ can be understood as the transport plan between the positions of the first and $i$th electrons ($i=2,\ldots,N_e$). It has been shown in \cite{hu2023exactness} that there exists a $\hat\beta\ge0$ such that the optimal solutions of \eqref{eqn:MPGCC l1} satisfy the constraints \eqref{eqn:extra constraints} whenever $\beta\ge\hat\beta$.

	\subsection{Systems under simulations}
	
	\par We consider eight one/two/three-dimensional (1D/2D/3D) systems. Table \ref{tab:system info} contains their single-particle densities, domains of interest, and numbers of electrons. 
	\begin{table}[!t]
		\centering
		\caption{$1$D/$2$D/$3$D systems used for simulations. The second column lists the unnormalized single-particle densities $\rho$, the third gives the domains $\Omega$, and the last indicates the numbers of electrons $N_e$ in systems.}
		\label{tab:system info}
		\resizebox{\linewidth}{!}{
			\begin{tabular}{c|l|l|c}
					\toprule
					System No. & \multicolumn{1}{c|}{$\rho$} & \multicolumn{1}{c|}{$\Omega$} & $N_e$ \\\midrule\midrule
					\multicolumn{4}{c}{$1$D systems} \\\midrule\midrule
					
					1 & $\cos(\pi r)+1$ & $[-1,1]$ & 3 \\\midrule[0.2pt]
					
					2 & $2\rho_6(r;-0.5)+1.5\rho_4(r;0.5)$ & $[-1.5,1.5]$ & 3 \\\midrule[0.2pt]
					
					3 & $\rho_{1/\sqrt{\pi}}(r)$ & $[-2,2]$ & 7 \\\midrule[0.2pt]
					
					4 & \tabincell{l}{$\rho_4(r;-2)+\rho_4(r;-1.5)+\rho_4(r;-1)+\rho_4(r;-0.5)$\\$+\rho_4(r;2/3)+\rho_4(r;4/3)+\rho_4(r;2)$} & $[-3,3]$ & 7 \\\midrule\midrule
					
					\multicolumn{4}{c}{$2$D systems} \\\midrule\midrule
					
					5 & \tabincell{l}{$\rho_3(\rr;[0,0.96]^\T)+\rho_3(\rr;[1.032,-0.84]^\T)+\rho_3(\rr;[-1.032,-0.84]^\T)$} & $[-3,3]^2$ & 3 \\\midrule[0.2pt]
					
					6 & \tabincell{l}{$2\rho_3(\rr;[0,1.2]^\T)+\rho_3(\rr;[1.29,-1.05]^\T)+\rho_3(\rr;[-1.29,-1.05]^\T)$} & $[-3,3]^2$ & 4 \\\midrule\midrule
					
					\multicolumn{4}{c}{$3$D systems} \\\midrule\midrule
					
					7 & $\rho_3(\rr;[-1,-1,-1]^\T)+\rho_3(\rr;[1,1,-1]^\T)+\rho_3(\rr;[-1,1,1]^\T)$ & $[-2,2]^3$ & 3 \\\midrule[0.2pt]
					
					8 & $3\rho_4(\rr;[-1,0,0]^\T)+\rho_4(\rr;[1,0,0]^\T)$ & $[-2,2]\times[-1,1]^2$ & 4 \\
					
					\bottomrule
		\end{tabular}}
	\end{table}
	The component function $\rho_{\alpha}(\cdot;\mathbf{c})$ ($\alpha>0$, $\mathbf{c}\in\R^d$) is defined as
		$$\rho_{\alpha}(\rr;\mathbf{c}):=\exp\lrbracket{-\alpha\norm{\rr-\mathbf{c}}^2},~~\forall~\rr\in\R^d.$$
		We illustrate the single-particle densities in Figure \ref{fig:single-particle densities}.
		\begin{figure}[!t]
			\centering
			\includegraphics[width=\linewidth]{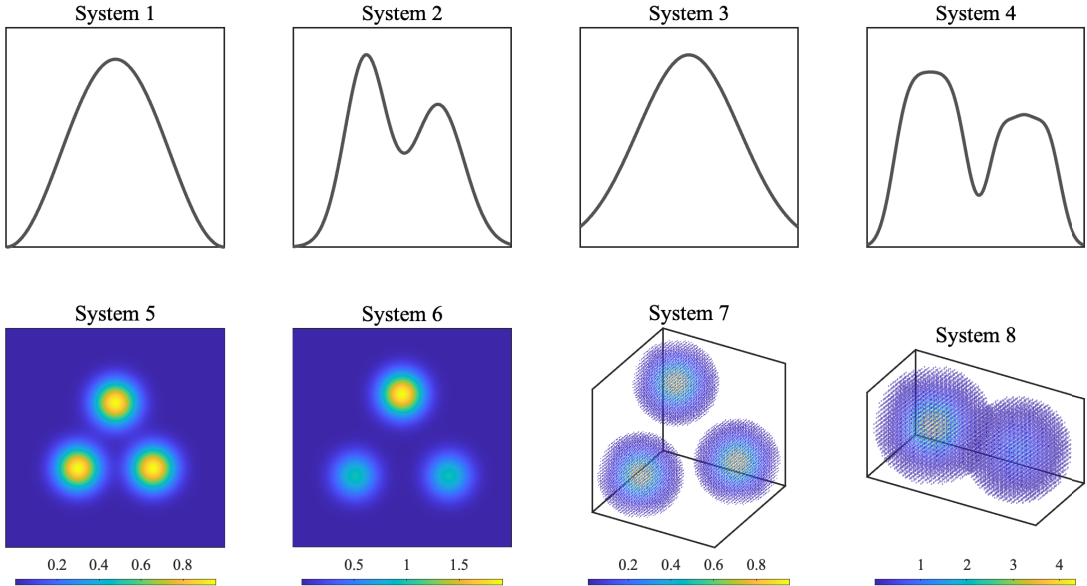}
			\caption{The illustrations of the single-particle densities listed in Table \ref{tab:system info}. For the 3D systems, we only show the regions where the values of the single-particle densities are larger than 0.01.}
			\label{fig:single-particle densities}
		\end{figure}
		
		\begin{remark}
			Nearly all the single-particle densities in Table \ref{tab:system info} comprise of Gaussian functions. These settings are reasonable because electrons tend to concentrate around the nuclei, which are represented by the potential wells. For example, the single-particle density of system 8 comprises of two Gaussian functions with weights 1 and 3, respectively. It can describe a dissociating lithium hydride \cite{filatov2015spin}, which has two nuclei with charge numbers 1 and 3, respectively. We shall note that the applicability of the proposed methods is independent from the constructions of the single-particle densities.
	\end{remark}
	
	\subsection{Default settings}\label{subsec:default settings}
	
	\par We employ either equimass or equisize discretization (to be fixed later) for problem \eqref{eqn:MPGCC l1}. 
	We set $\beta=1$ for any $K$ following \cite{hu2023global}, which is reasonable because the nonzero entries in both $C$ and $\Lambda^{-1}$ are of order $K$. For both the \eralm~and \seralm~methods, we adopt decreasing step sizes $\alpha_i^{(t)} = 1/(t+1)^{0.75}$, $i=2,\ldots,N_e$, to pursue convergence. In the \seralm~and \sklalm~methods, we let 
	$n_{s,i}=\lfloor K^{1.5} \rfloor$ ($i=2,\ldots,N_e$). 
	All the proposed methods set the regularization or proximal parameters adaptively as 
	\begin{equation}
		\lambda_i^{(t)}=\sigma\snorm{\tilde\vvv_i^{(t)}}_\infty/(20\log(K)),~~\mu_i^{(t)}=\sigma\snorm{\vvv_i^{(t)}}_\infty/(20\log(K)),~~i=2,\ldots,N_e,
		\label{eqn:proximal parameter}
	\end{equation}
	where $\sigma=1$. The selections of $\gamma$ and $\hat t$ will be presented in section \ref{subsec:algorithm comparison}. We shall point out that better performances of the proposed methods can be expected with more careful parameter tuning. All the subproblems are solved using the Sinkhorn algorithm \eqref{eqn:sinkhorn}, with warm starts for acceleration. Incidentally, to circumvent possible underflow and achieve better convergence rates \cite{dvurechensky2018computational} when using the Sinkhorn algorithm, we discard the entries in $\brho$ that are smaller than 0.1\% of $\snorm{\brho}_\infty$; this is also practically reasonable, in that the regions of low probabilities are far less important. With the abuse of notations, we still denote the truncated vector by $\brho$ so that other symbols remain unchanged. Regarding the stopping criteria, we terminate the Sinkhorn algorithm whenever the feasibility violation $\snorm{Y_i^{(t+1)}\one_K-\brho}_\infty$ is less than $10^{-6}$ or the subiteration number arrives at $s_{\max}=20$ ($i=2,\ldots,N_e$). We stop the outer loop if 
	$$\Delta^{(t)}:=\frac{1}{N_e-1}\sum_{i=2}^{N_e}\norm{\Lambda^{-1}\big(Y_i^{(t)}-Y_i^{(t-1)}\big)}$$
	falls below a prescribed $tol>0$ or the iteration number reaches a prescribed $t_{\max}\in\N$. The specific values of $tol$ and $t_{\max}$ are detailed in the subsequent subsections. All the experiments presented are run in a platform with Intel(R) Xeon(R) Gold 6242R CPU @ 3.10GHz and 510GB RAM running \textsc{Matlab} R2019b under Ubuntu 20.04.
	
	
	
	\par For quantities of interest, we monitor the converged objective value (obj) and approximate the so-called strictly-correlated-electrons (SCE) potential \cite{chen2014numerical,di2020optimal,hu2023global} with the output dual variables. The SCE potential is the functional derivative of the optimal value of the MMOT with respect to $\rho$ and is important for electronic structure calculations \cite{chen2014numerical}. Taking the \eralm~method for example, we approximate the SCE potential by $\vvv:=\tilde{\vvv}-\min_{k=1}^K\{\tilde v_k\}\cdot\one_K\in\R^K$, where
	$$\tilde{\vvv}:=\frac{1}{N_e-1}\sum_{i=2}^{N_e}\tilde{\vvv}_i\in\R^K,$$
	and $\{\tilde{\vvv}_i\}_{i=2}^{N_e}$ are the dual solutions yielded by the Sinkhorn algorithm. In addition, for the cases where explicit constructions of the optimal solutions to problem \eqref{eqn:MPGCC l1} and the SCE potentials of the MMOT are available (e.g., in one-dimensional settings \cite{colombo2015multimarginal,hu2023exactness}), we also evaluate the qualities of the converged solutions via the relative errors of the objective values (err\_obj) and SCE potentials (err\_sce). They are defined respectively as
	$$\errobj:=\abs{\frac{\text{obj}-\text{obj}^{\star}}{\text{obj}^{\star}}},~\errsce:=\frac{\snorm{\vvv-\vvv^{\star}}_\infty}{\snorm{\vvv^{\star}}_\infty}.$$
	Here, $\text{obj}^{\star}\in\R$ denotes the optimal objective value of problem \eqref{eqn:MPGCC l1} and $\vvv^{\star}\in\R^K$ refers to the vector made up by the values of the SCE potential at barycenters. For the efficiency comparison, we record the CPU time in seconds (T).
	
	
	\subsection{Algorithm comparisons}\label{subsec:algorithm comparison}
	
	\par We conduct comparisons among the \palm~method \cite{bolte2014proximal,hu2023convergence} and the proposed four methods on the 1D systems in Table \ref{tab:system info}. In particular, we first test the \seralm~methods with different sampling probabilities and the \sklalm~methods with different choices of $\hat t$. As a byproduct, we select default values of $\gamma$ and $\hat t$. Secondly, we compare the proposed four methods. Those with favorable performances are tested against the \palm~method in the third part and used for the simulations of 2D and 3D systems in section \ref{subsec:large-scale simulations}.
	
	\subsubsection{Comparisons among the {\sc S-eralm}~methods with different sampling probabilities}\label{subsubsec:different gamma}
	
	We consider randomly generated sampling probabilities and importance sampling-based probabilities \eqref{eqn:subsampling probability} ($\gamma\in\{0.1,0.3,0.5,0.7,0.9,0.99,0.999\}$) on system 1 with $K=90$ (equimass discretization). For each setting of sampling probability, 10 random trials are generated by the built-in function ``\texttt{rand}'' in \textsc{Matlab}. The stopping parameters are $tol=5\times10^{-3}$ and $t_{\max}=+\infty$. We record the achieved err\_obj, err\_sce, and required T averaged over 10 trials for each setting in Table \ref{tab:compare sampling probabilities}.
	\begin{table}[htb]
		\centering
		\caption{The achieved err\_obj, err\_sce, and required T averaged over 10 trials given by the \seralm~methods with different sampling probabilities on system 1 with $K=90$ (equimass discretization).}
		\label{tab:compare sampling probabilities}
		\begin{tabular}{c||ccr}
				\toprule
				Sampling Prob. & err\_obj & err\_sce & \multicolumn{1}{c}{T} \\\midrule
				Random & 0.4184 & 0.86 & 128.84\\
				$\gamma=0.1$ & 0.3098 & 0.69 & 120.16\\
				$\gamma=0.3$ & 0.1729 & 0.54 & 93.30 \\
				$\gamma=0.5$ & 0.1044 & 0.42 & 65.10\\
				$\gamma=0.7$ & 0.0693 & 0.39 & 48.13\\
				$\gamma=0.9$ & 0.0597 & 0.37 & 35.05\\
				$\gamma=0.99$ & 0.0525 & 0.36 & 21.76 \\
				$\gamma=0.999$ & 0.1118 & 0.34 & 5.66\\\bottomrule
		\end{tabular}
	\end{table}
	
	\par Though lacking theoretical justifications, the sampling probabilities incorporating information of previous iterates are found to yield lower errors within less CPU time than randomly generated ones. Increasing the value of $\gamma$ contributes to less CPU time for fulfilling the stopping criterion, yet worsening the accuracy once surpassing some threshold. We select $\gamma=0.99$ in the \seralm~and \sklalm~methods for a compromise between accuracy and efficiency in the subsequent experiments.
	
	\subsubsection{Comparisons among the {\sc S-klalm}~methods with different $\hat t$}
	
	\par We conduct numerical comparisons on system 1 in Table \ref{tab:system info} with equimass discretization and $K\in\{90,180,360,720\}$. We call the \sklalm~method with $\hat t\in\{0,5,10\}$. For each pair of $(K,\hat t)$, 10 random trials are performed. The stopping parameters are $tol=10^{-3}\times\sqrt{2}^{\log_2(K/90)}$\footnote{We increase $tol$ since the size of $\calU(\brho,\brho)$ decreases as $\calO(1/\sqrt{K})$; see Lemma \ref{lemma:diameter}.} and $t_{\max}=+\infty$. We depict the achieved err\_obj, err\_sce, and required T averaged over 10 trials for each pair of $(K,\hat t)$ in Figure \ref{fig:hat t not sensitive}.
		\begin{figure}[!t]
			\centering
			\includegraphics[width=\linewidth]{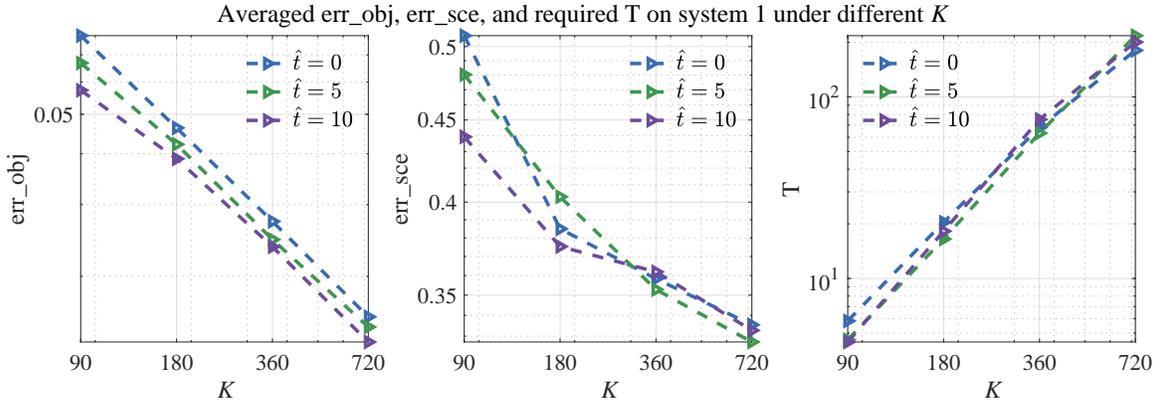}
			\caption{The achieved err\_obj, err\_sce, and required T averaged over 10 trials for each pair of $(K,\hat t)$ given by the \sklalm~method with different values of $\hat t$ on system 1 (equimass discretization). The blue, green, and purple dashed lines with right-pointing triangle markers are the results of the \sklalm~method with $\hat t=0,5,10$, respectively. Left: err\_obj. Middle: err\_sce. Right: T.}
			\label{fig:hat t not sensitive}
		\end{figure}
		
		\par From Figure \ref{fig:hat t not sensitive}, we observe that the \sklalm~method with a positive $\hat t$ does yield higher-quality solutions within comparable CPU time. But the advantage becomes less obvious as $K$ grows large. Since large-scale problems are common in the applications of interest and a positive $\hat t$ indicates the computations and storage of full matrices during the first iterations, we choose $\hat t=0$ in the \sklalm~method for the ensuing numerical experiments. 
	
	\subsubsection{Comparisons among the {\sc S-eralm}~and {\sc S-klalm}~methods}
	
	\par We first test the proposed four methods with $\sigma$ taking its value in $\{2^0,2^2,2^4,2^6,2^8\}$ on system 1 with $K=90$ and equimass discretization. For each value of $\sigma$, 10 random trials are performed. The stopping parameters are $tol=5\times10^{-3}/\sqrt{\sigma}$ and $t_{\max}=+\infty$. We record the achieved err\_obj, err\_sce and required T averaged over 10 random trials for each value of $\sigma$ in Figure \ref{fig:regproxparam}. 
		\begin{figure}[!t]
			\centering
			\includegraphics[width=\linewidth]{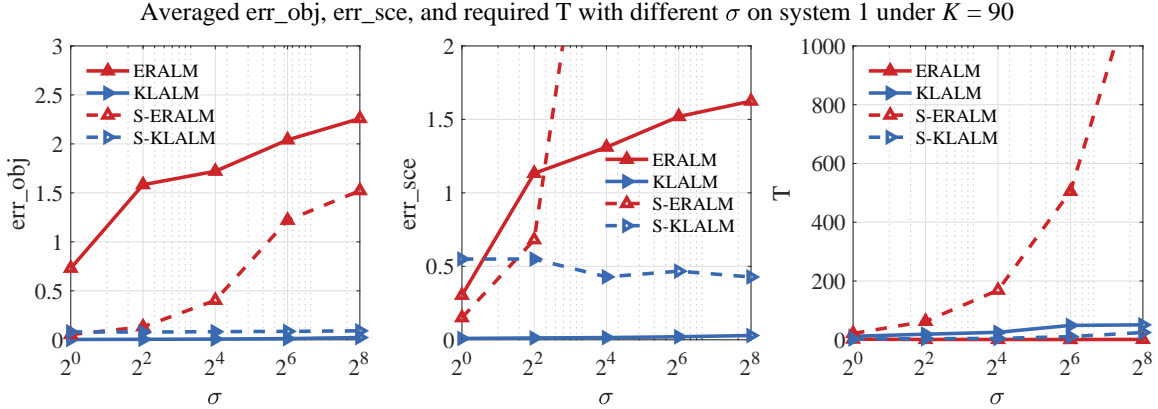}
			\caption{The achieved err\_obj, err\_sce, and required T averaged over 10 trials for each value of $\sigma$ given by the \eralm, \klalm, S-\eralm, and S-\klalm~methods on system 1 with $K=90$ (equimass discretization). The red solid and dashed lines with triangle markers represent the results of the \eralm~and S-\eralm~methods, respectively. The blue solid and dashed lines with right-pointing triangle markers represent the results of the \klalm~and S-\klalm~methods, respectively. Left: err\_obj. Middle: err\_sce. Right: T.}
			\label{fig:regproxparam}
		\end{figure}
		
		\par From Figure \ref{fig:regproxparam}, we observe that (i) the objective errors of the \eralm~and \seralm~methods rise quickly as $\sigma$ increases, which conforms to the theoretical results in section \ref{sec:convergence analyses}; (ii) the \klalm~and S-\klalm~methods yield high-quality solutions regardless of the choices of $\sigma$, demonstrating their robustness to the choices of proximal parameters. The robustness is practically desirable because a tiny $\sigma$, as is needed by the (S-)\eralm~methods to achieve high accuracy, can result in numerical underflow. 
		
		\par In the above settings, we also notice that the S-\eralm~method arrives at better solutions than those given by the S-\klalm~method when a small $\sigma$ is used. But we shall point out that the advantage no longer persists as $K$ increases. We test the S-\eralm~and S-\klalm~methods with $\sigma=1$ on system 1 with $K\in\{90,180,360,720\}$ (equimass discretization). For each value of $K$, 10 random trials are performed. The stopping parameters are $tol=10^{-3}\times\sqrt{2}^{\log_2(K/90)}$ and $t_{\max}=+\infty$. We depict the achieved err\_obj, err\_sce, and required T averaged over 10 random trials for each value of $K$ in Figure \ref{fig:exclude seralm}. 
		
		\begin{figure}[htb]
			\centering
			\includegraphics[width=\linewidth]{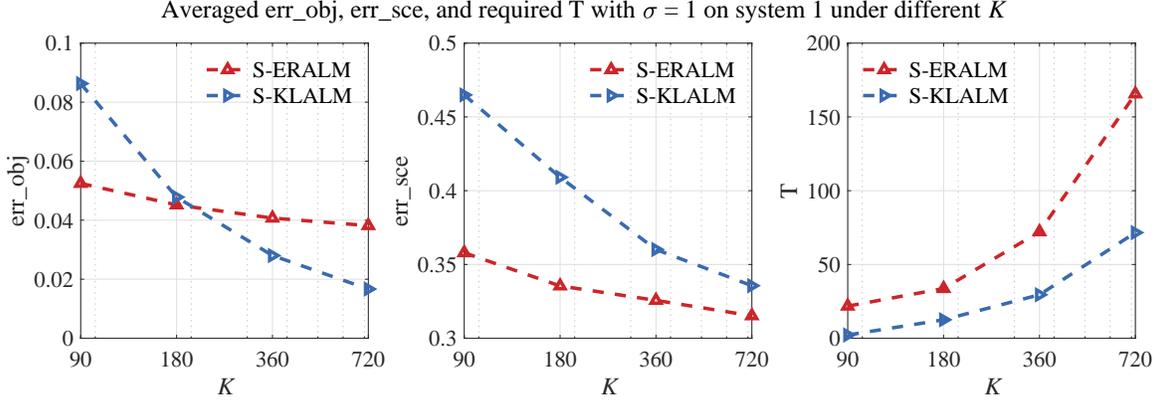}
			\caption{The achieved err\_obj, err\_sce, and required T averaged over 10 trials for each value of $K$ given by the S-\eralm~and S-\klalm~methods with $\sigma=1$ on system 1 (equimass discretization). The red dashed lines with triangle markers represent the results of the S-\eralm~method. The blue dashed lines with right-pointing triangle markers represent the results of the S-\klalm~method. Left: err\_obj. Middle: err\_sce. Right: T.}
			\label{fig:exclude seralm}
		\end{figure}

		\par From Figure \ref{fig:exclude seralm}, we conclude that the advantage of the S-\eralm~method over the S-\klalm~method in terms of accuracy disappears as the problem size increases, along with fast growing computational time. This is in line with Table \ref{tab:complexities}, i.e., the matrix entrywise sampling consumes quadratic complexity in each iteration of the S-\eralm~method.
		
		\par Finally, we compare the \klalm~and S-\klalm~methods. Aiming at a problem restricted on the sampled support, it is impractical for the S-\klalm~method to outperform the \klalm~method in accuracy. However, if the computational budget is limited, the S-\klalm~method can yield relatively high-quality solutions within much less CPU time, particularly when the problem size goes large. We demonstrate this point by numerical comparisons on system 1. We employ equimass discretization with $K\in\{90,180,360,720,1440,2880\}$. We call the \klalm~and S-\klalm~methods for each value of $K$ with 10 random trials. The stopping parameters are $tol=10^{-3}\times\sqrt{2}^{\log_2(K/90)}$ and $t_{\max}=+\infty$. We depict the convergence curves of err\_obj along with the CPU time averaged over 10 trials for each value of $K$ in Figure \ref{fig:1D cos Tinter}. 
		\begin{figure}[!t]
			\centering
			\includegraphics[width=\linewidth]{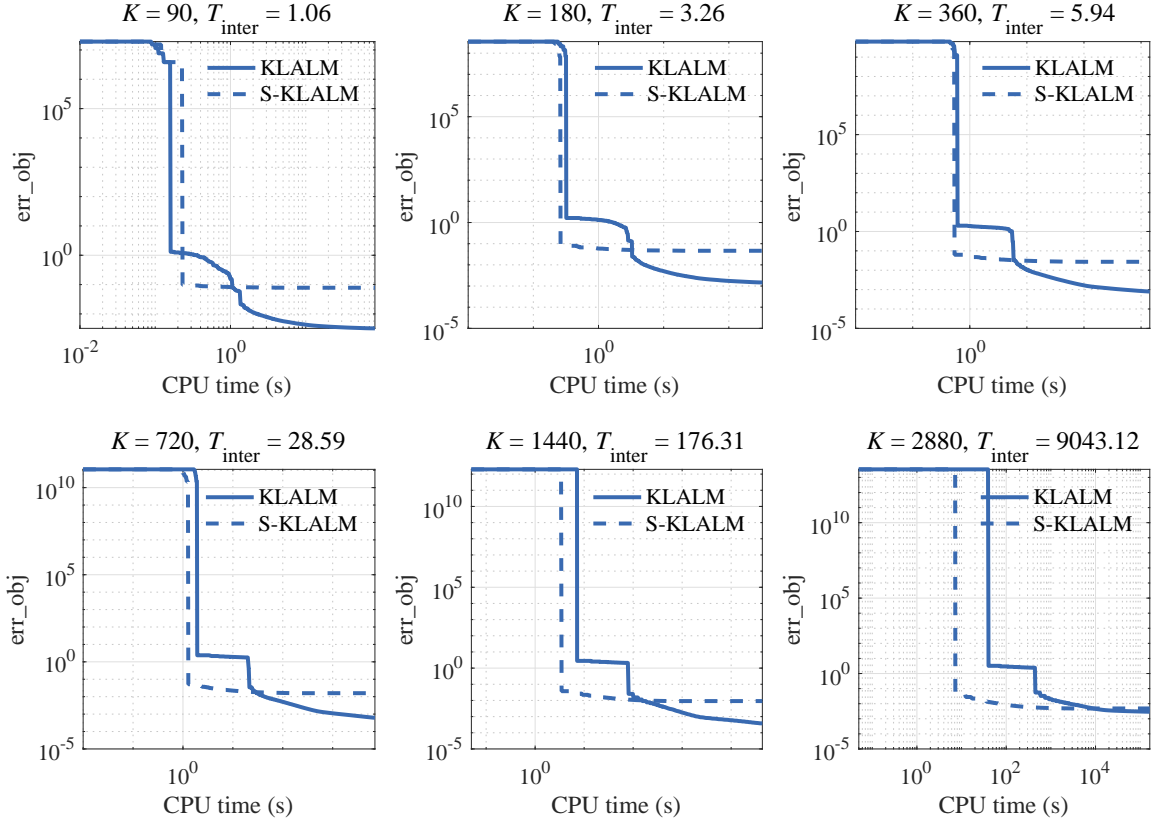}
			\caption{The convergence curves of err\_obj along with the CPU time averaged over 10 trials for each value of $K$ given by the \klalm~and S-\klalm~methods on system 1 (equimass discretization). The blue solid and dashed lines stand for the results of the \klalm~and S-\klalm~methods, respectively. The notation $T_{\inter}$ refers to the CPU time where the curves of two methods intersect for the last time.}
			\label{fig:1D cos Tinter}
		\end{figure}
		The notation $T_{\inter}$ refers to the CPU time where the curves of two methods intersect for the last time.
		
		\par From Figure \ref{fig:1D cos Tinter}, we observe that the \klalm~method attains worse accuracy than the S-\klalm~method until the CPU time touches $T_{\inter}$. Moreover, as we increase the value of $K$, $T_{\inter}$ grows at a cubic rate; see Figure \ref{fig:fitted 1D cos Tinter} for an illustration of cubic polynomial fitting.
		\begin{figure}[!t]
			\centering
			\includegraphics[width=.5\linewidth]{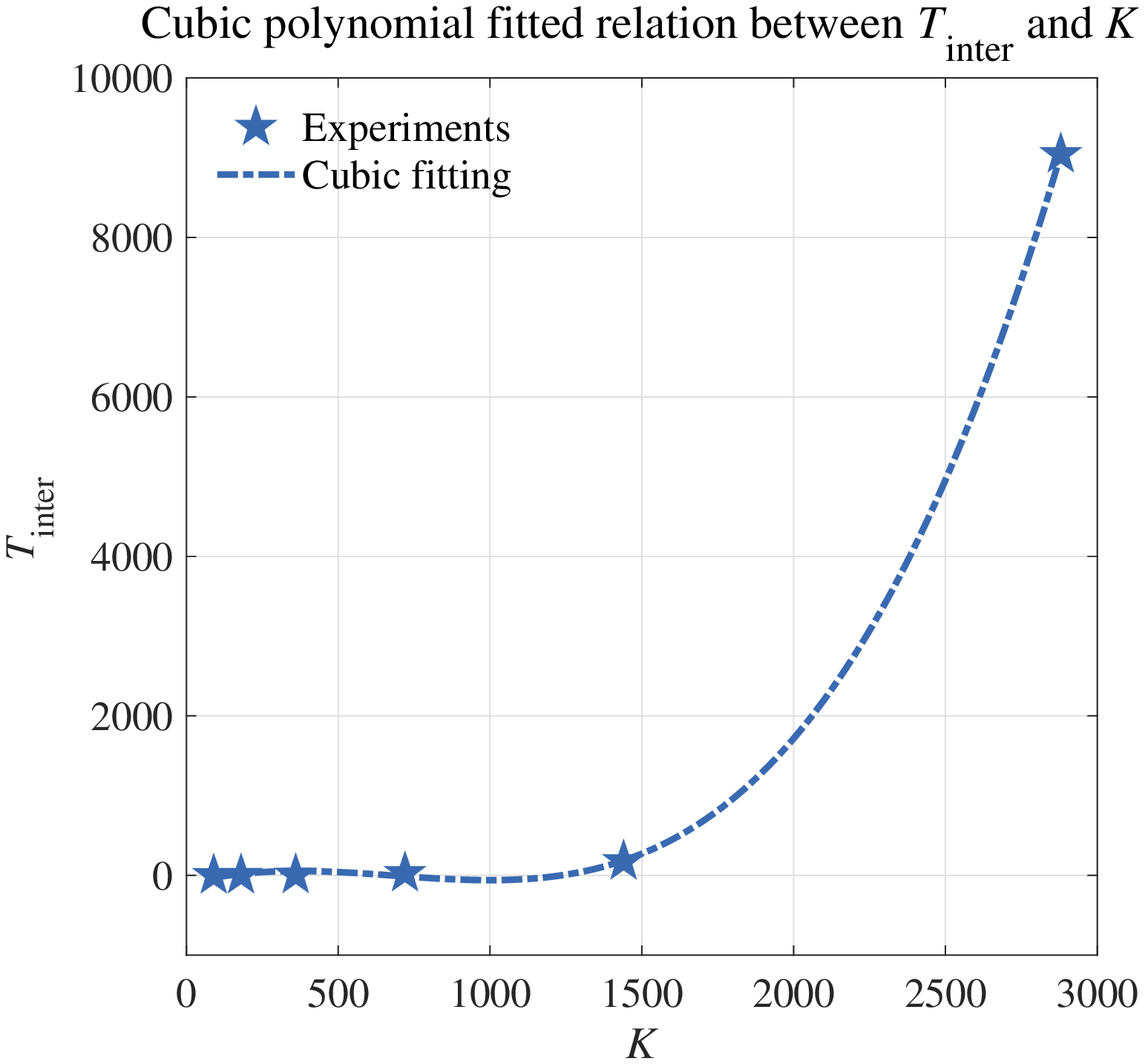}
			\caption{Cubic polynomial fitted relation between $T_{\inter}$ and $K$ on system 1 (equimass discretization). The fitted polynomial is $9.09\times10^{-7}K^3-1.86\times10^{-3}K^2+1.00K-106.38$. The blue pentagons are the obtained $T_{\inter}$ in experiments and the blue dashdotted line is the fitted relation.}
			\label{fig:fitted 1D cos Tinter}
		\end{figure}
		These results imply that (i) if the problem size is small (e.g., of order $10^2$), the \klalm~method achieves high accuracy within acceptable CPU time; (ii) if the problem size is relatively large (e.g., of order $10^3$ or higher), as is usually the case in the applications of interest, and the computational budget is limited, the S-\klalm~method is more preferable from the practical perspective.
		
		\par Based on the above numerical findings, we choose the \klalm~and \sklalm~methods for the comparisons with the \palm~method as well as the simulations of the 2D and 3D systems in section \ref{subsec:large-scale simulations}.
		
	\subsubsection{Comparisons among the {\sc palm}, {\sc klalm}, and {\sc S-klalm}~methods}
	
	\par We perform numerical comparisons on the 1D systems in Table \ref{tab:system info} with equimass discretization. The implementation of the \palm~method follows \cite{hu2023global}, except that we tune the proximal parameters adaptively as in equation \eqref{eqn:proximal parameter} for fair comparisons. For systems 1 and 2, we consider $K\in\{90,180,360,720\}$. For each value of $K$, 10 random trials are performed. The stopping parameters are $tol=10^{-3}\times\sqrt{2}^{\log_2(K/90)}$ and $t_{\max}=+\infty$. For systems 3 and 4, we consider $K\in\{140,280,560,1120\}$. For each value of $K$, 10 random trials are performed. The stopping parameters are $tol=10^{-3}\times\sqrt{2}^{\log_2(K/140)}$ and $t_{\max}=+\infty$.
		The achieved err\_obj, err\_sce, and required T averaged over 10 random trials given by the three methods are gathered in Figures \ref{fig:local compare N3} and \ref{fig:local compare N7}.
		\begin{figure}[!t]
			\centering
			\subfloat[System 1]{\includegraphics[width=\linewidth]{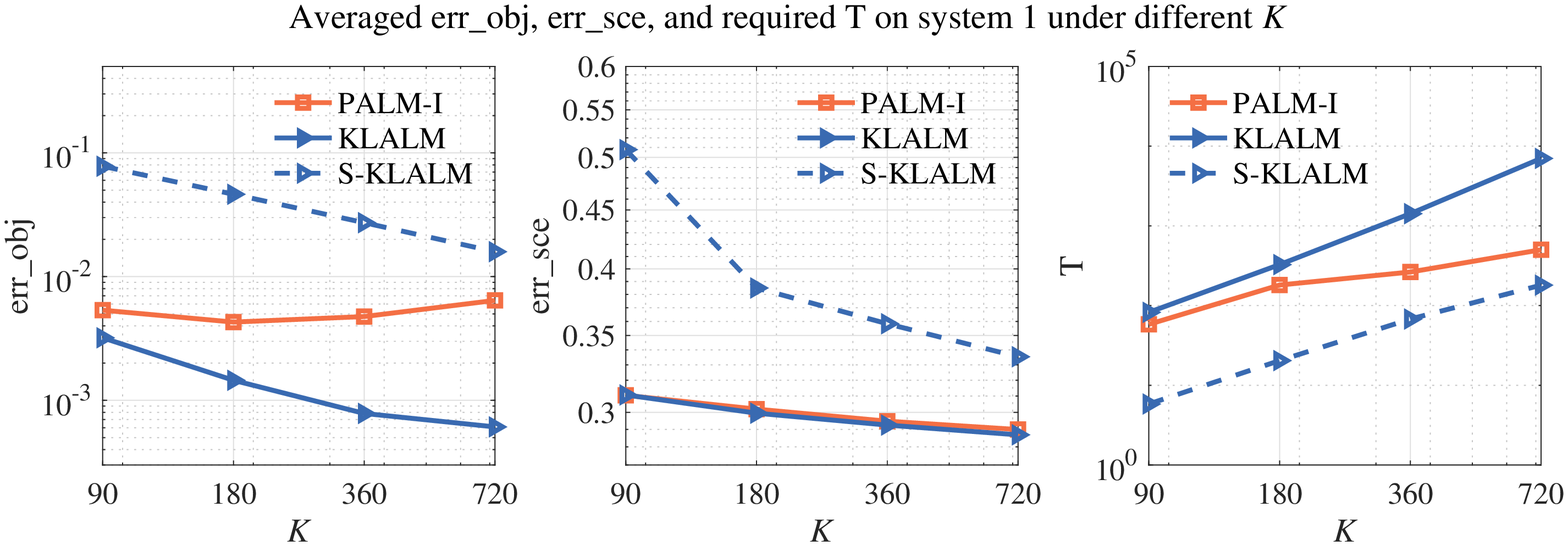}}\\
			\subfloat[System 2]{\includegraphics[width=\linewidth]{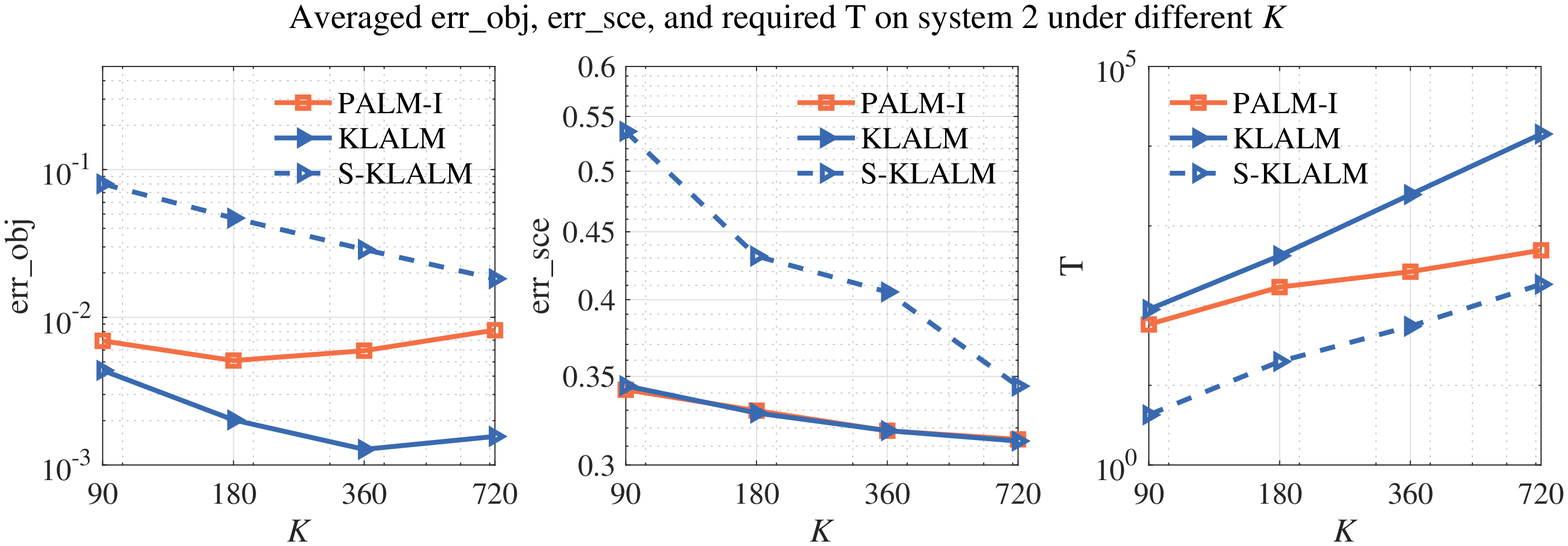}}
			\caption{The achieved err\_obj, err\_sce, and required T averaged over 10 trials for each value of $K$ given by the \palm, \klalm,~and \sklalm~methods on the three-electron 1D systems (equimass discretization). The orange solid lines with square markers are the results of the \palm~method. The blue solid and dashed lines with right-pointing triangle markers are the results of the \klalm~and \sklalm~methods, respectively. From left to right: err\_obj, err\_sce, and T. (a) System 1. (b) System 2.}
			\label{fig:local compare N3}
		\end{figure}
		\begin{figure}[!t]
			\centering
			\subfloat[System 3]{\includegraphics[width=\linewidth]{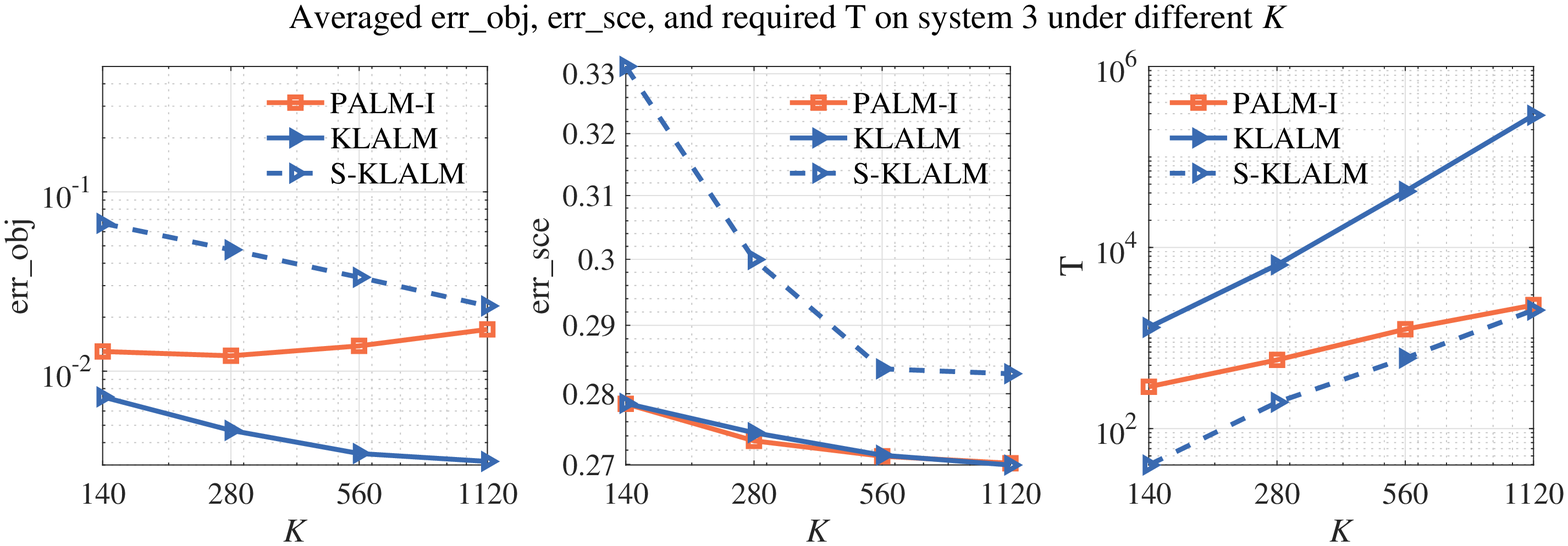}}\\
			\subfloat[System 4]{\includegraphics[width=\linewidth]{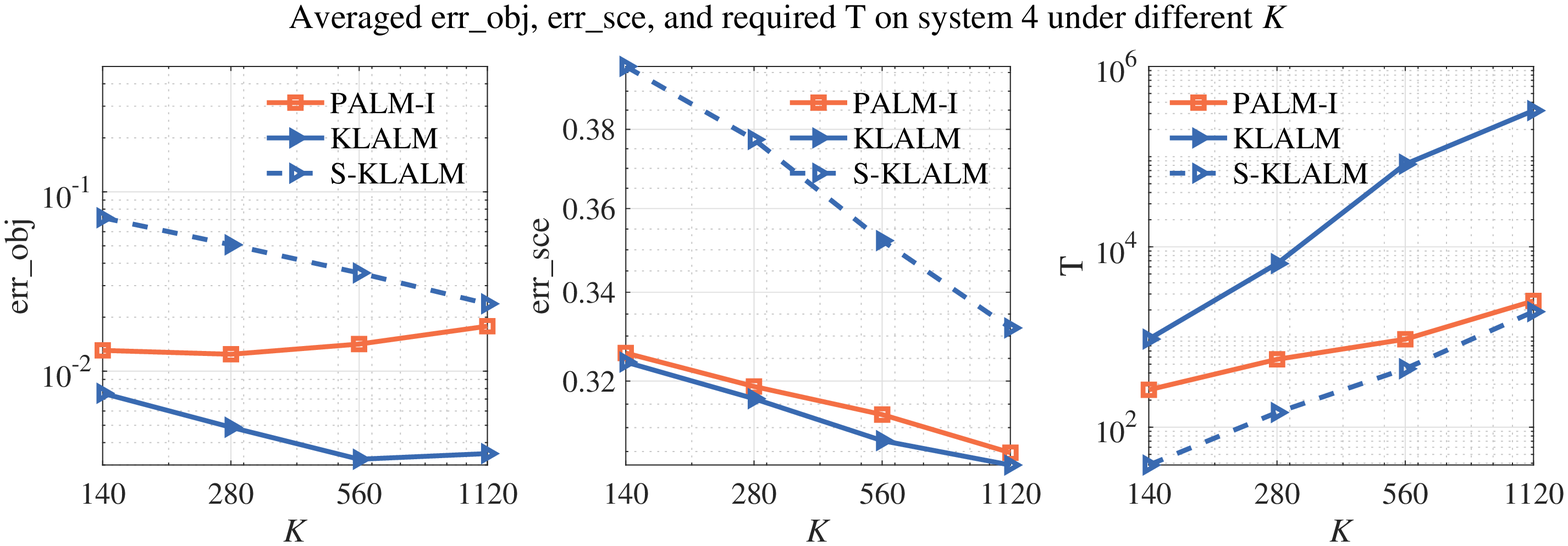}}
			\caption{The achieved err\_obj, err\_sce, and required T averaged over 10 trials for each value of $K$ given by the \palm, \klalm,~and \sklalm~methods on the seven-electron 1D systems (equimass discretization). The orange solid lines with square markers are the results of the \palm~method. The blue solid and dashed lines with right-pointing triangle markers are the results of the \klalm~and \sklalm~methods, respectively. From left to right: err\_obj, err\_sce, and T. (a) System 3. (b) System 4.}
			\label{fig:local compare N7}
		\end{figure}
		
		\par Under the same settings of proximal and stopping parameters, the \palm~method suffers from premature convergence, while the \klalm~method yields much smaller objective errors across the 1D systems. The multiplicative expression \eqref{eqn:optimal solution bpalm subproblem} further enables the use of randomized matrix sparsification, leading to the \sklalm~method with better scalability.
		
	
	\subsection{A cascadic multigrid optimization framework}
	
	\par Before the simulations of 2D and 3D systems, we introduce a cascadic multigrid (\cmg) optimization framework for problem \eqref{eqn:MPGCC l1}. The \cmg~optimization framework has demonstrated its power in accelerating the numerical solution of large-scale scientific problems, especially those with governing partial differential equations; see, e.g., \cite{borzi2008multigrid,chen2017efficient,xia2018cascadic,liu2021multilevel}. In general, the framework begins with a coarse mesh, with which a small-scale problem is associated. An accurate solver is then called to solve the small-scale problem and yields a solution. Later, it basically repeats the following three steps: (i) refine the previous coarse mesh to a finer one; (ii) prolongate the previous solution and construct an initial point over the current mesh; (iii) start a local solver from the initial point and obtain a solution over the current mesh. With carefully designed prolongation operator and local solver, the framework can fully utilize the solution information in the coarser meshes and considerably accelerate the numerical solution of the large-scale problems over the finer meshes. 
		
		\par In our context, we call accurate solvers, such as the \klalm~method, to tackle the problems over the initial coarse meshes. After uniform or adaptive mesh refinements, we prolongate the previous solution (say, $\big(Y_2^{(\ell-1,\star)},\ldots,Y_{N_e}^{(\ell-1,\star)}\big)$) to the current mesh (say, $\{e_k^{(\ell)}\}_{k=1}^{K^{(\ell)}}$) as follows:
		\begin{equation}
			y_{i,kl}^{(\ell,0)}:=\frac{1}{K_{k'l'}^{(\ell)}}y_{i,k'l'}^{(\ell-1,\star)},\quad k,l=1,\ldots,K^{(\ell)},~~i=2,\ldots,N_e,
			\label{eqn:prolongation operator}
		\end{equation}
		where $k'$, $l'\in\{1,\ldots,K^{(\ell-1)}\}$ are such that $e_{k}^{(\ell)}\subseteq e_{k'}^{(\ell-1)}$, $e_l^{(\ell)}\subseteq e_{l'}^{(\ell-1)}$, and 
		$$K_{k'l'}^{(\ell)}:=\abs{\lrbrace{(k,l):e_k^{(\ell)}\subseteq e_{k'}^{(\ell-1)},~e_l^{(\ell)}\subseteq e_{l'}^{(\ell-1)},~k,l\in\{1,\ldots,K^{(\ell)}\}}}.$$
		The design of the prolongation operator follows that $\rho\gamma_i$ can be understood as a joint probability density between the positions of the first and $i$th electrons ($i=2,\ldots,N_e$). Finally, cheap local solvers, such as the S-\klalm~method, start from $\big(Y_2^{(\ell,0)},\ldots,Y_{N_e}^{(\ell,0)}\big)$ and solve the problem over the current mesh. Consequently, apart from the previously mentioned warm starts, the prolongation operator also enables ``warm samplings''; please refer to the importance sampling probability \eqref{eqn:subsampling probability}. We summarize our \cmg~optimization framework in Framework \ref{frame:gr}.
		\begin{algorithm}[!t]
			\floatname{algorithm}{Framework}
			\caption{The \cmg~optimization framework for problem \eqref{eqn:MPGCC l1}.}
			\label{frame:gr}
			\begin{algorithmic}[1]
				\REQUIRE{Discretization and refinement oracles, accurate and cheap local solvers, initial number of finite elements $K^{(0)}\in\N$.} 
				\STATE{Set $\ell:=0$.}
				\WHILE{\textit{certain conditions are not satisfied}}
				\IF{$\ell=0$}
				\STATE{\textbf{Discretization:} discretize the MMOT into problem \eqref{eqn:MPGCC l1} with $K^{(0)}$ finite elements $\{e_k^{(0)}\}_{k=1}^{K^{(0)}}\subseteq\R^d$.}
				\STATE{Construct a random initial point $(Y_2^{(0,0)},\ldots,Y_{N_e}^{(0,0)})\in(\R^{K^{(\ell)}\times K^{(\ell)}})^{N_e-1}$.}
				\STATE{\textbf{Local solution: }start the accurate local solver from $(Y_2^{(0,0)},\ldots,Y_{N_e}^{(0,0)})$ for problem \eqref{eqn:MPGCC l1} at level 0 and obtain $(Y_2^{(0,\star)},\ldots,Y_{N_e}^{(0,\star)})$.}
				\ELSE
				\STATE{\textbf{Grid refinement:} refine $\{e_k^{(\ell-1)}\}_{k=1}^{K^{(\ell-1)}}$ to $\{e_k^{(\ell)}\}_{k=1}^{K^{(\ell)}}\subseteq\R^d$ with $K^{(\ell)}\in\N$.}
				\STATE{\textbf{Prolongation:} prolongate $(Y_2^{(\ell-1,\star)},\ldots,Y_{N_e}^{(\ell-1,\star)})$ to the current mesh through equation \eqref{eqn:prolongation operator} and obtain $(Y_2^{(\ell,0)},\ldots,Y_{N_e}^{(\ell,0)})\in(\R^{K^{(\ell)}\times K^{(\ell)}})^{N_e-1}$.}
				\STATE{\textbf{Local solution:} start the cheap local solver from $(Y_2^{(\ell,0)},\ldots,Y_{N_e}^{(\ell,0)})$ for problem \eqref{eqn:MPGCC l1} at level $\ell$ and obtain $(Y_2^{(\ell,\star)},\ldots,Y_{N_e}^{(\ell,\star)})$.}
				\ENDIF
				\STATE{Set $\ell:=\ell+1$.}
				\ENDWHILE
				\ENSURE{Approximate solution $(Y_2^{(\ell-1,\star)},\ldots,Y_{N_e}^{(\ell-1,\star)})$.}
			\end{algorithmic}
		\end{algorithm}
		
		\par To demonstrate the utility of Framework \ref{frame:gr}, we include numerical experiments on the 1D systems. We call the \cmg~optimization method, which takes respectively the \klalm~and \sklalm~methods as the accurate and cheap local solvers, the S-\klalm-\cmg~method. We compare the numerical performances of the \klalm, \sklalm,~and S-\klalm-\cmg~methods. For systems 1 and 2, the \klalm~and S-\klalm~methods directly solve the problems with $K=720$ and equimass discretization starting from 10 random trials. The stopping parameters are $tol=2\sqrt{2}\times10^{-3}$ and $t_{\max}=+\infty$. The S-\klalm-\cmg~method starts 10 random trials from $K^{(0)}=90$ with equimass discretization and reaches the desired the mesh after equimass refinements for three times. The stopping parameters at level $\ell\in\{0,1,2,3\}$ are $tol=10^{-3}\times\sqrt{2}^{\log_2(K^{(\ell)}/K^{(0)})}$ and $t_{\max}=+\infty$. The settings for the simulations of systems 3 and 4 are analogous, except that the problem size handled by the \klalm~and \sklalm~methods is $K=1120$ and the initial problem size handled by the S-\klalm-\cmg~method is $K^{(0)}=140$. We report the achieved err\_obj, err\_sce, and required T averaged over 10 random trials given by the three methods on the 1D systems in Table \ref{tab:GR is useful}. 
		\begin{table}[!t]
			\centering
			\caption{The achieved err\_obj, err\_sce, and required T averaged over 10 trials given by the \klalm, S-\klalm, and S-\klalm-\cmg~methods on the 1D systems (equimass discretization).\\}
			\label{tab:GR is useful}
			\begin{tabular}{c|ccr||ccr}
					\toprule
					\multirow{2}{*}{Algorithms} & \multicolumn{3}{c||}{System 1 ($K=720$)} & \multicolumn{3}{c}{System 2 ($K=720$)}\\\cmidrule{2-7}
					& err\_obj & err\_sce & \multicolumn{1}{c||}{T} & err\_obj & err\_sce & \multicolumn{1}{c}{T} \\\midrule
					\klalm~& 0.0006 & 0.29 & 7014.57 & 0.0016 & 0.31 & 14190.02 \\
					S-\klalm~& 0.0159 & 0.34 & 180.51 & 0.0183 & 0.34 & 184.96\\\midrule
					S-\klalm-\cmg & 0.0032 & 0.34 & 96.32 & 0.0044 & 0.38 & 109.39 \\\midrule\midrule
					
					\multirow{2}{*}{Algorithms} & \multicolumn{3}{c||}{System 3 ($K=1120$)} & \multicolumn{3}{c}{System 4 ($K=1120$)}\\\cmidrule{2-7}
					& err\_obj & err\_sce & \multicolumn{1}{c||}{T} & err\_obj & err\_sce & \multicolumn{1}{c}{T} \\\midrule
					\klalm~& 0.0031 & 0.27 & 288688.75 & 0.0035 & 0.30 & 323734.12 \\
					S-\klalm~& 0.0231 & 0.28 & 2034.49 & 0.0238 & 0.33 & 1910.36 \\\midrule
					S-\klalm-\cmg & 0.0074 & 0.32 & 1204.31 & 0.0079 & 0.36 & 1252.50 \\
					\bottomrule
			\end{tabular}
		\end{table}
		The S-\klalm-\cmg~method is found to yield relatively high-quality solutions within the least CPU time.

	\subsection{Simulations on 2D and 3D systems}\label{subsec:large-scale simulations}
	
	We use the \sklalmgr~method for simulating the $2$D and $3$D systems in Table \ref{tab:system info}. The initial step employs equisize discretization and the latter ones refine the meshes uniformly. For the $2$D systems, the initial mesh size is $K^{(0)}=900$ before truncation. After grid refinements for three times, we get $K^{(3)}=57600$ before truncation. For the $3$D system 7, we set $K^{(0)}=1728$ before truncation and perform grid refinements twice to arrive at $K^{(2)}=110592$ before truncation. For the $3$D system 8, we set $K^{(0)}=1000$ before truncation and perform grid refinements twice to arrive at $K^{(2)}=64000$ before truncation. The stopping parameters are $t_{\max}=10^4$ and
	$$tol=\left\{\begin{array}{ll}
		5\times10^{-3}, & \ell=0,\\
		10^{-2}\times(\sqrt{2^d})^{\log_2(K/K^{(0)})}, & \ell>0.
	\end{array}\right.$$
	Note that for these systems, no explicit constructions of the optimal solutions are available. Hence, we monitor instead the evolution of objective values and the approximate SCE potentials. Moreover, following \cite{hu2023global}, we approximate the OT maps between electron positions in the so-called Monge ansatz \cite{seidl1999strictly,seidl1999strong,seidl2000simulation}, which are of particular physical interest, as 
	\begin{equation*}
		\scrT_i(\mb{d}_j):=\sum_{1\le k\le K}y_{i,jk}\mb{d}_k/\varrho_j,~j=1,\ldots,K,~i=2,\ldots,N_e,
	\end{equation*}
	where $\scrT_i:\R^d\to\R^d$ represents the transport map between the positions of the first and $i$th electrons. We collect the results in Table \ref{tab:2D/3D obj} as well as Figures \ref{fig:sce potential 2D/3D} and \ref{fig:approximate transport maps 2D/3D}, where $K_{\mathrm{trunc}}\in\N$ refers to the dimension of the truncated $\brho$. 
	\begin{table}[!t]
		\centering
		\caption{The objective values calculated along the iteration of the \sklalmgr~method when solving the $2$D and $3$D systems. The notation $K_{\mathrm{trunc}}\in\N$ is the dimension of the truncated $\brho$.}
		\label{tab:2D/3D obj}
		\begin{tabular}{c|rr|c||rr|c}
			\toprule
			\multirow{2}{*}{Step} & \multicolumn{3}{c||}{System 5} & \multicolumn{3}{c}{System 6} \\\cmidrule{2-7}
			& $K$ & $K_{\mathrm{trunc}}$ & obj & $K$ & $K_{\mathrm{trunc}}$ & obj \\\midrule
			
			0 & 900 & 424 & 1.1339 & 900 & 408 & 3.0690 \\
			
			1 & 3600 & 1622 & 1.1337 & 3600 & 1534 & 3.0690 \\
			
			2 & 14400 & 6410 & 1.1335 & 14400 & 6068 & 3.0677 \\
			
			3 & 57600 & 25562 & 1.1334 & 57600 & 24176 & 3.0667 \\
			
			\midrule\midrule
			
			\multirow{2}{*}{Step} & \multicolumn{3}{c||}{System 7} & \multicolumn{3}{c}{System 8} \\\cmidrule{2-7}
			& $K$ & $K_{\mathrm{trunc}}$ & obj & $K$ & $K_{\mathrm{trunc}}$ & obj \\\midrule
			
			0 & 1728 & 780 & 1.0202 & 1000 & 720 & 4.6193 \\
			
			1 & 13824 & 5628 & 1.0209 & 8000 & 5272 & 4.6716 \\
			
			2 & 110592 & 42936 & 1.0209 & 64000 & 40764 & 4.6833 \\
			
			\bottomrule
			
		\end{tabular}
	\end{table}
	\begin{figure}[!t]
		\centering
		\subfloat[System 5]{\includegraphics[width=\linewidth]{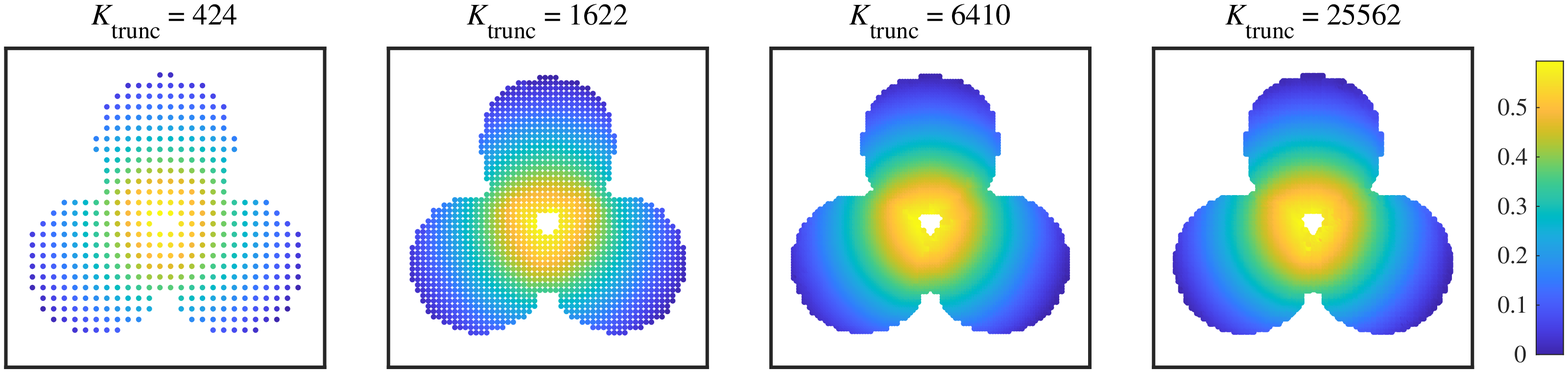}}\\
		\subfloat[System 6]{\includegraphics[width=\linewidth]{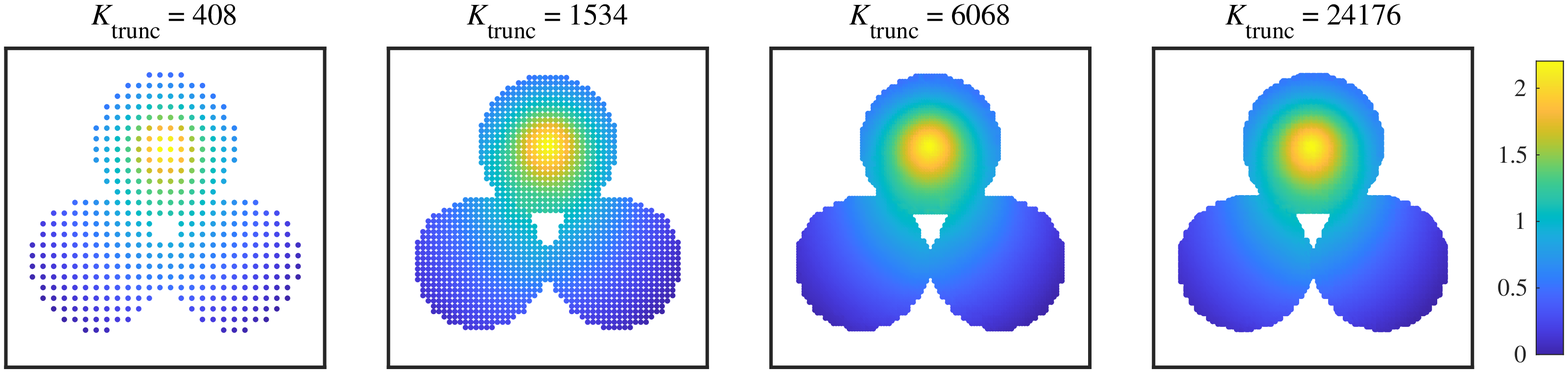}}\\
		\subfloat[System 7]{\includegraphics[width=\linewidth]{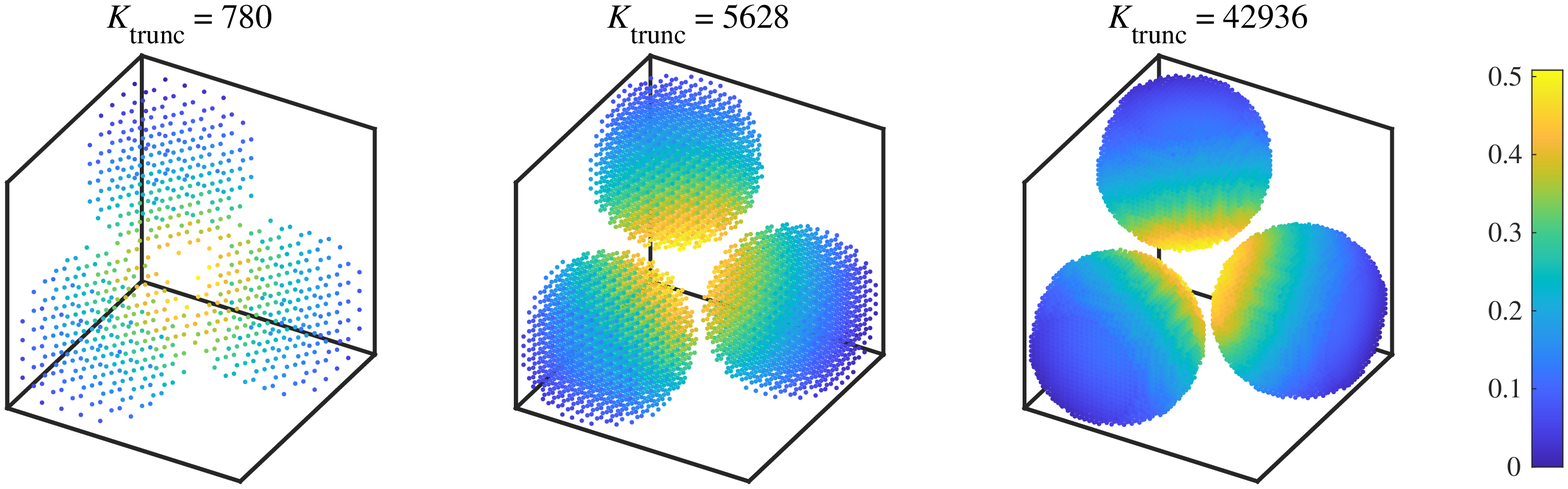}}\\
		\subfloat[System 8]{\includegraphics[width=\linewidth]{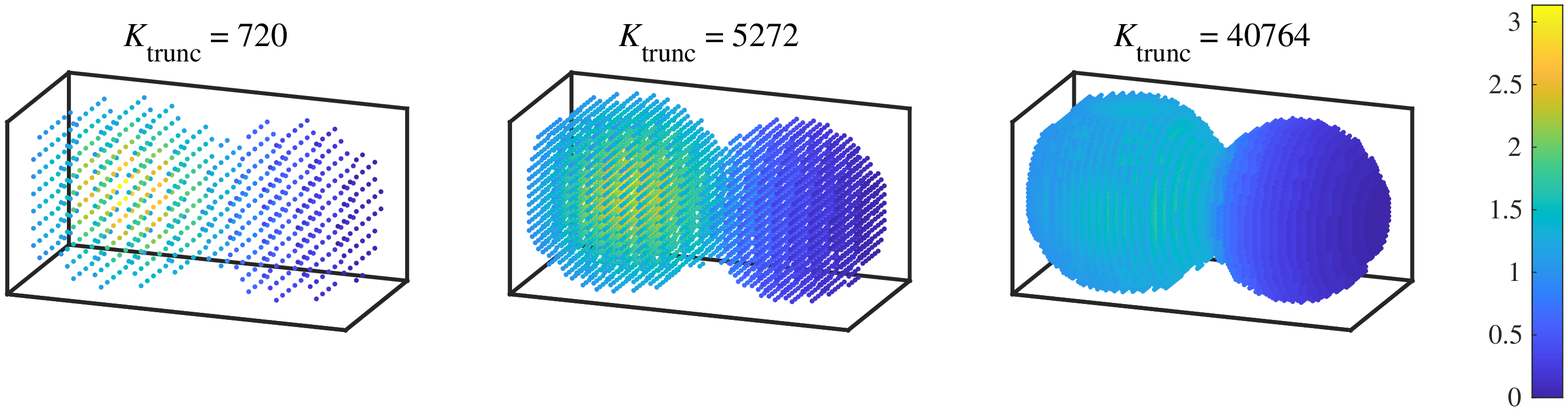}}
		\caption{The evolution of the approximate SCE potentials for the $2$D and $3$D systems given by the \sklalmgr~method. (a) System $5$. (b) System $6$.(c) System $7$. (d) System $8$.}
		\label{fig:sce potential 2D/3D}
	\end{figure}
	\begin{figure}[!t]
		\centering
		\subfloat[System 5]{\includegraphics[width=.48\linewidth]{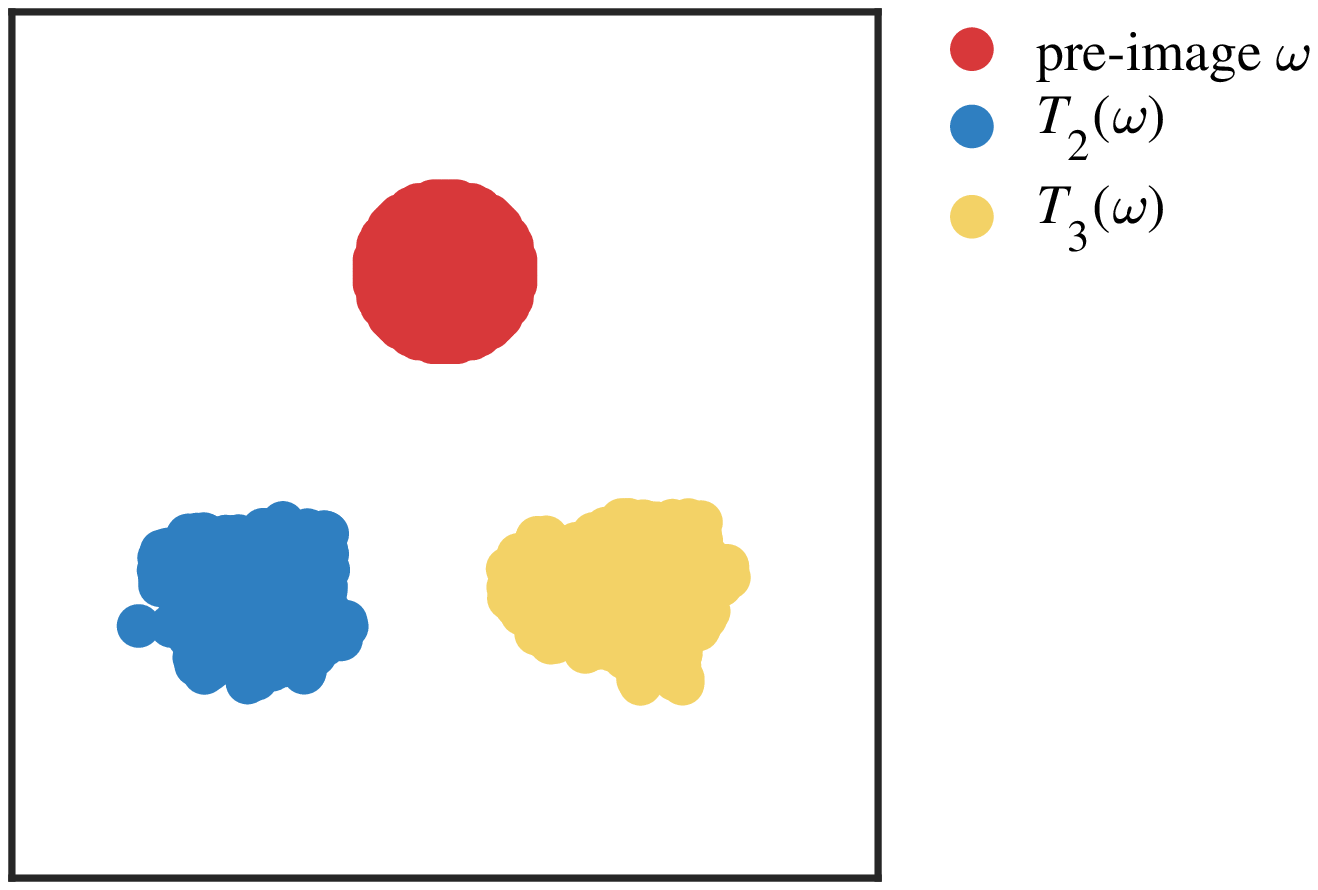}}\quad
		\subfloat[System 6]{\includegraphics[width=.48\linewidth]{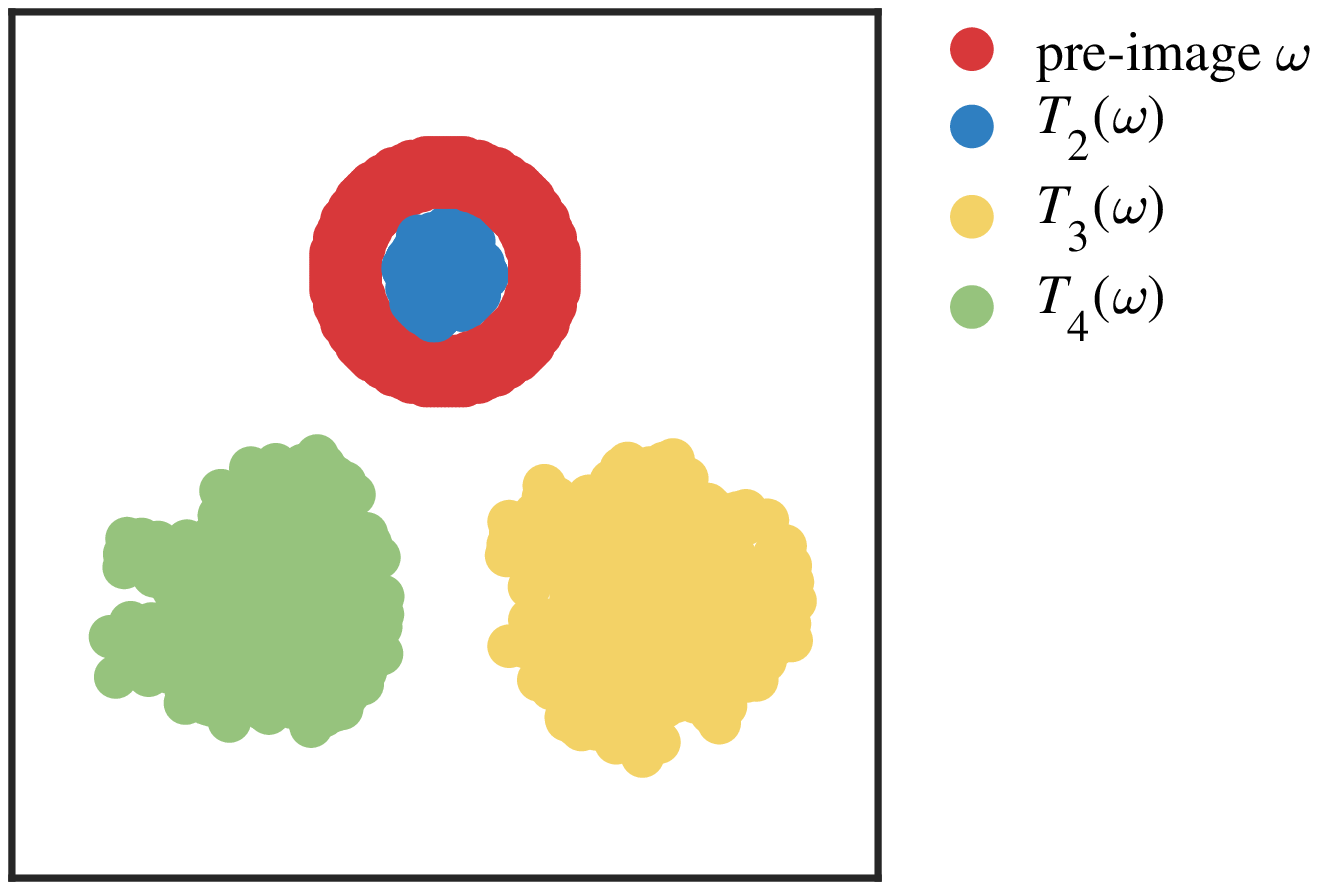}}\\[0.1cm]
		\subfloat[System 7]{\includegraphics[width=.48\linewidth]{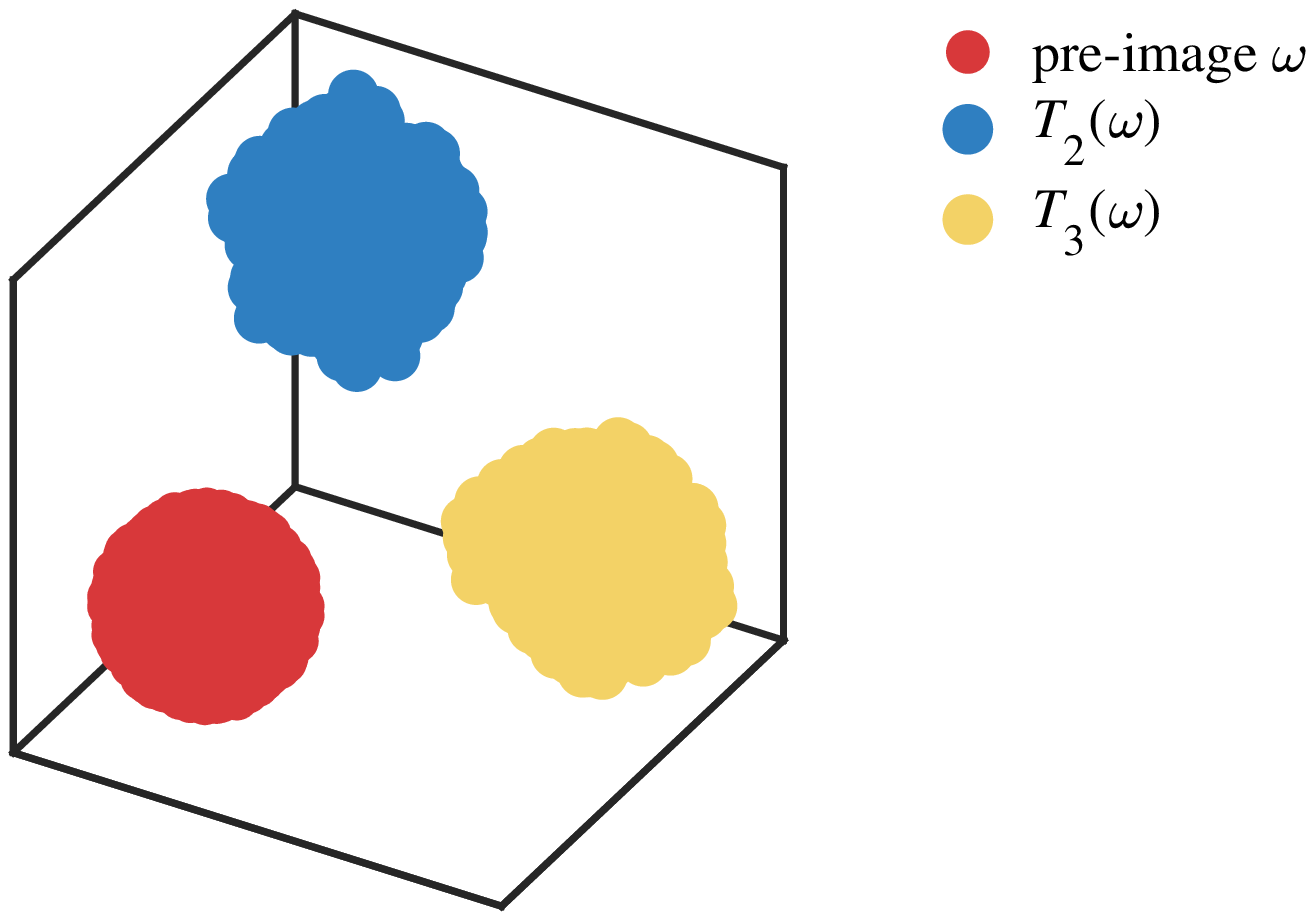}}\quad
		\subfloat[System 8]{\includegraphics[width=.48\linewidth]{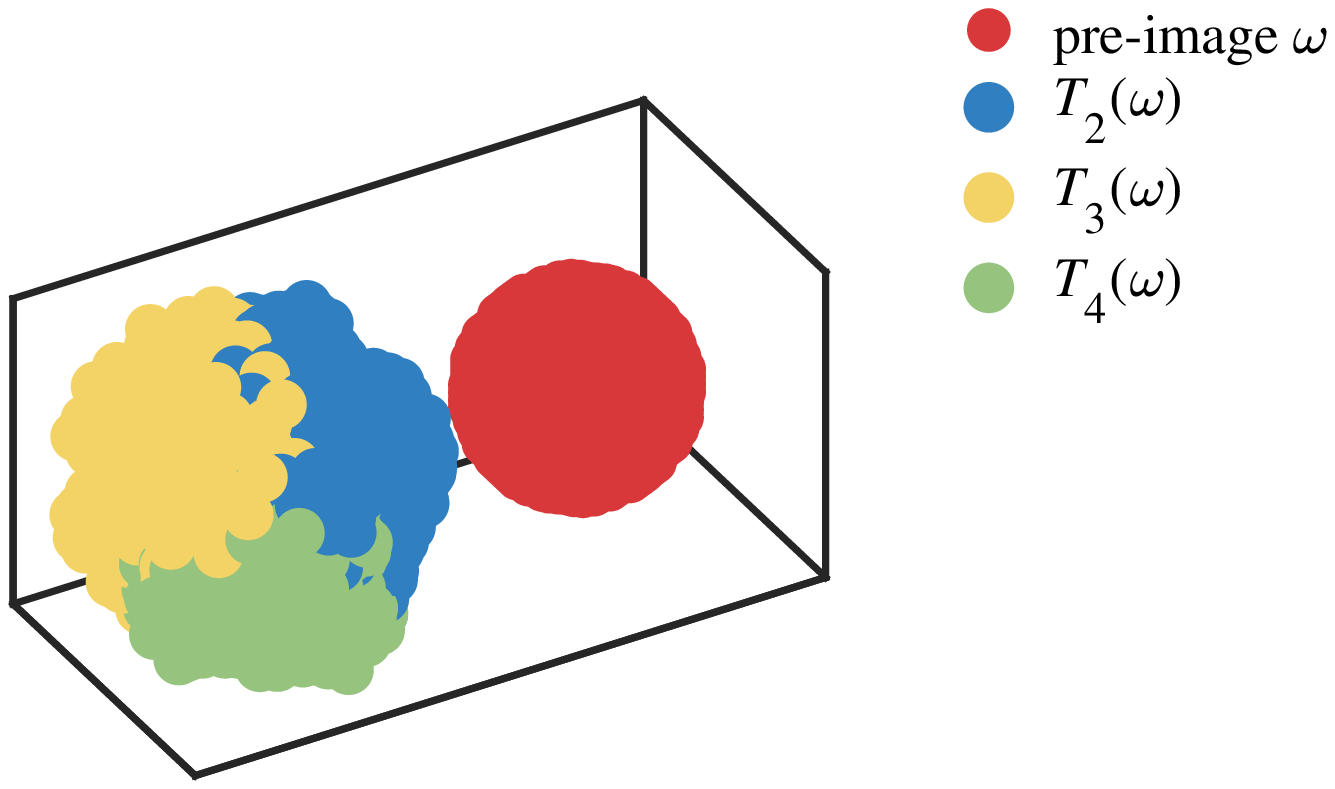}}
		\caption{The approximate OT maps between electron positions for the $2$D and $3$D systems given by the \sklalmgr~method. The pre-image $\omega$ in red stands for the positions of the first electron, while the areas in other colors are the associated positions of the other electrons, respectively, determined by $\{\scrT_i\}_{i=2}^{N_e}$. (a) System $5$. (b) System $6$. (c) System $7$. (d) System $8$.}
		\label{fig:approximate transport maps 2D/3D}
	\end{figure}
	We illustrate the approximate OT maps $\{\scrT_i\}_{i=2}^{N_e}$ through their images of the barycenters of the finite elements within some given subregion $\omega\subseteq\Omega$.
	
	\par The results in Table \ref{tab:2D/3D obj} and Figure \ref{fig:sce potential 2D/3D} showcase the convergence of the \cmg~optimization framework, also manifesting the effectiveness of the prolongation operator therein. The approximate OT maps in Figure \ref{fig:approximate transport maps 2D/3D} are in line with physical intuitions. In particular, for the three-Gaussian $2$D system 5, Figure \ref{fig:approximate transport maps 2D/3D} (a) implies that if the first electron is around one Gaussian center, the other two electrons will go near the other two Gaussian centers, respectively. For the four-Gaussian $2$D system 6, Figure \ref{fig:approximate transport maps 2D/3D} (b) shows that if one electron revolves around but stays away from the top Gaussian center, the other three electrons will stay near the three Gaussian centers, respectively, with one of which surrounded by the first one but keeping away from each other. For the three-Gaussian $3$D system 7, the results depicted in Figure \ref{fig:approximate transport maps 2D/3D} (c) are analogous to those in Figure \ref{fig:approximate transport maps 2D/3D} (a). For the two-Gaussian $3$D system 8, Figure \ref{fig:approximate transport maps 2D/3D} (d) indicates that if one electron lies around the Gaussian center $[1,0,0]^\T$, the other three will be located around the other center, with their positions trisecting a sphere. In Figure \ref{fig:approximate transport maps 2D/3D}, the positions of the electrons determined by the approximate OT maps $\{\scrT_i\}_{i=2}^{N_e}$ conform to the repulsive law and are improved along our optimization process. Notably in Figure \ref{fig:approximate transport maps 2D/3D}, we provide the first visualization of the approximate OT maps between electron positions in $3$D contexts. 

	
	
	\begin{remark}
		There have been extensive works dedicated to the numerical solutions of the MMOT \eqref{eqn:MMOT}. 
			The authors of \cite{mendl2013kantorovich} consider the Kantorovich dual of the MMOT, whose number of inequality constraints increases exponentially with $N_e$. 
			The authors of \cite{benamou2016numerical} investigate the entropy regularized MMOT, where curse of dimensionality still resides. 
			The authors of \cite{khoo2019convex} derive a convex semidefinite programming relaxation for the discretized $N$-representability form of the MMOT, where the problem size is independent from the value of $N_e$ but the gap induced by the convex relaxation remains elusive. 
			The authors of \cite{alfonsi2022constrained} exploit a moment-constrained relaxation of the MMOT, which admits sparse optimal solutions but entails careful selections of parameters and test functions to achieve satisfactory approximations. 
			In both \cite{chen2014numerical} and our work, the reformulation of the MMOT under the Monge-like ansatz is adopted, where the problem size increases linearly with respect to $N_e$. However, at the moment, the ansatz is provably true only for special (e.g., two-electron/one-dimensional/special radially symmetric) cases \cite{friesecke2023strong}. 
			
			\par The considerable distinctions in models make a fair numerical comparison difficult and go beyond our scope. Nevertheless, we shall note that the model under the Monge-like ansatz explicitly characterizes the electron-electron couplings. Therefore, in contrast with others, it enables the evaluation of solution qualities through approximate OT maps between electron positions; see Figure \ref{fig:approximate transport maps 2D/3D}. Additionally, in comparison with \cite{chen2014numerical} where the authors only consider two-electron systems (i.e., H$_2$ molecule) with discretization size $K\sim3000$, we solve larger-scale problems ($K\sim10^5$) and simulate systems with more electrons ($N_e=3\sim7$). 
\end{remark}

\subsection{Scalability tests}\label{subsec:scalability}

Finally, we conduct scalability tests for the \klalm~and \sklalm~methods with respect to $K$ and $N_e$. The system under simulation shares the same normalized single-particle density $\rho$ with system 1 in Table \ref{tab:system info}, yet with varying $N_e$. We employ equimass discretization.

\par For the tests with respect to $K$, we fix $N_e=3$ and consider $K\in\{90,180,360,720\}$. For each value of $K$, the two methods are called with 10 random trials. The proximal parameter is set to $\mu_i^{(t)}\equiv0.05$. The stopping parameters are $tol=10^{-3}$ and $t_{\max}=+\infty$. The achieved err\_obj, err\_sce, and required T averaged over 10 trials are gathered in Figure \ref{fig:scalability} (a). 
\begin{figure}[!t]
\centering
\subfloat[Scalability with respect to $K$]{\includegraphics[width=\linewidth]{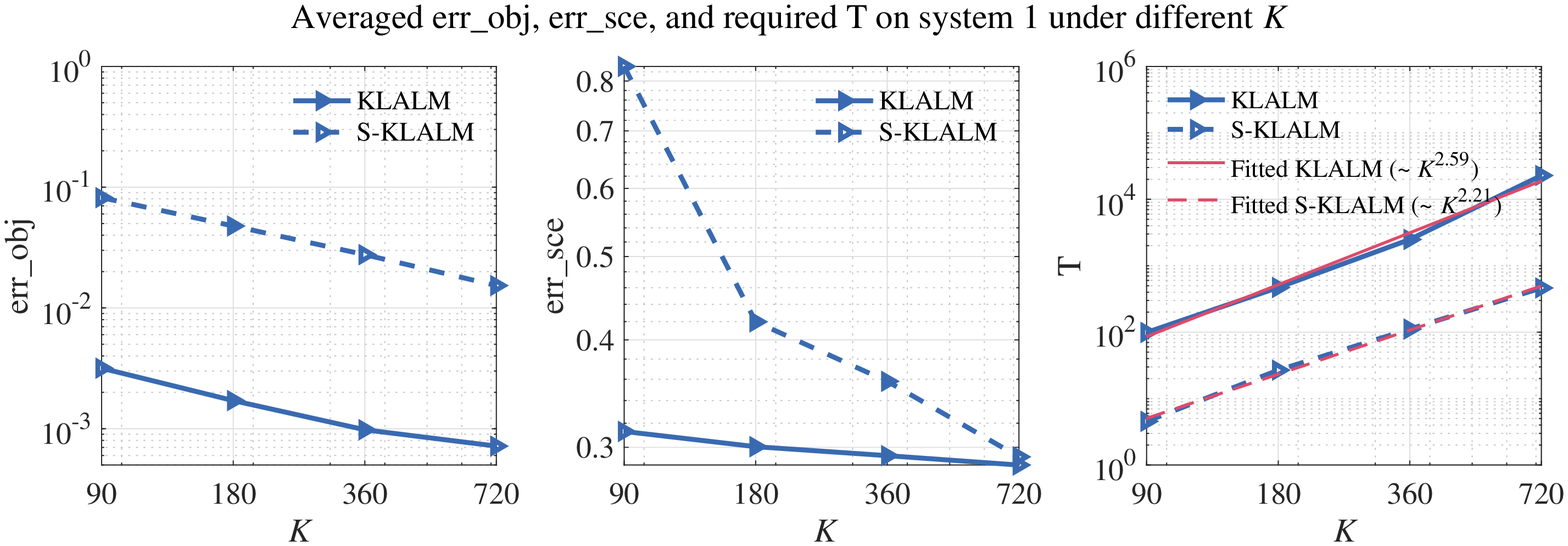}}\\
\subfloat[Scalability with respect to $N_e$]{\includegraphics[width=\linewidth]{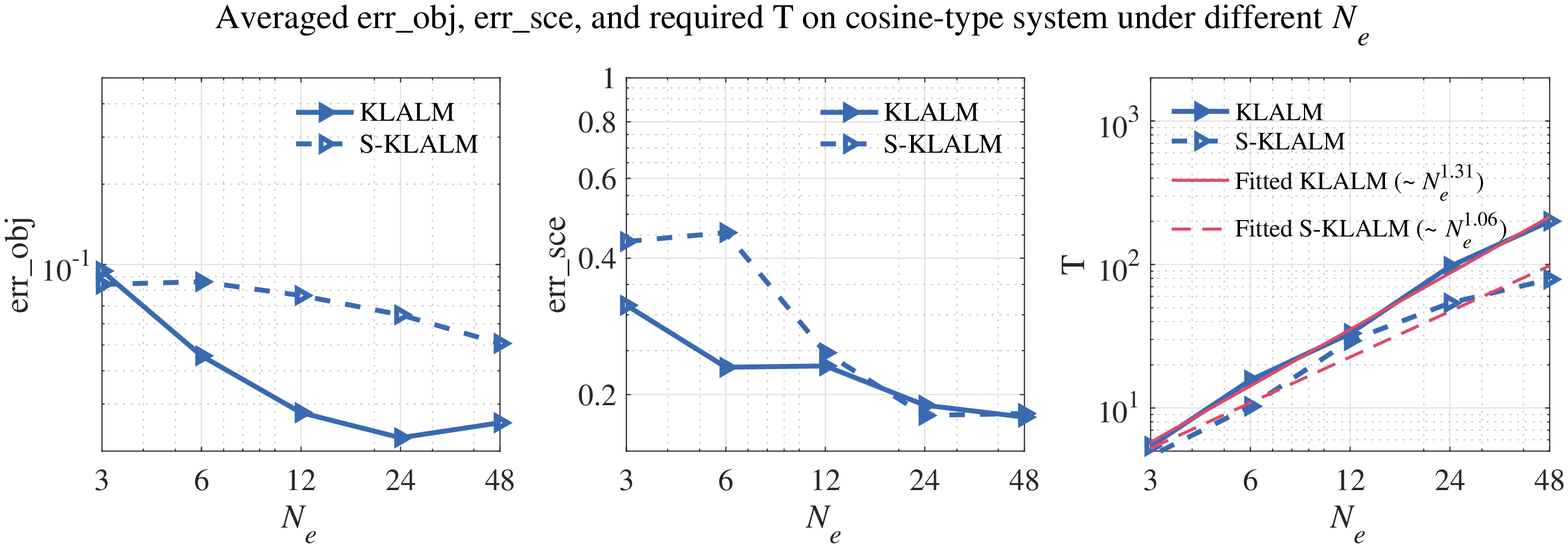}}
\caption{The achieved err\_obj, err\_sce, and required T averaged over 10 trials for each value of $K$ and $N_e$ given by the \klalm~and S-\klalm~methods on the cosine-type system (equimass discretization). The blue solid and dashed lines represent the results of the \klalm~and S-\klalm~methods, respectively. The pink solid and dashed lines denote the fitted relations between T and $K$ (or $N_e$) when using the \klalm~and \sklalm~methods, respectively. Left: err\_obj. Middle: err\_sce. Right: T. (a) Scalability with respect to $K$. For the \klalm~method, $\text{T}\sim K^{2.6}$. For the S-\klalm~method, $\text{T}\sim K^{2.2}$. (b) Scalability with respect to $N_e$. For the \klalm~method, $\text{T}\sim N_e^{1.3}$. For the \sklalm~method, $\text{T}\sim N_e^{1.1}$.}
\label{fig:scalability}
\end{figure}
In light of Table \ref{tab:complexities}, we fit the linear relation between $\log(\text{T})$ and $\log(K)$ with linear least squares and obtain the following:
\begin{align*}
\text{\klalm:}\quad\log(\text{T})&~\approx2.59\log(K)-7.22;\\
\text{\sklalm:}\quad\log(\text{T})&~\approx2.21\log(K)-8.31.
\end{align*}

\par For the tests with respect to $N_e$, we fix $K=144$ and vary $N_e$ in $\{3,6,12,24,48\}$. For each value of $N_e$, the two methods are called with 10 random trials. We fix the proximal parameter $\mu_i^{(t)}\equiv 20/\log(K)$ because $\snorm{\vvv_i^{(t)}}_\infty$ grows with $N_e$. The stopping parameters are $tol=5\times10^{-3}$ and $t_{\max}=+\infty$. The achieved err\_obj, err\_sce, and required T averaged over 10 trials are collected in Figure \ref{fig:scalability} (b). Likewise, we obtain the following relations:
\begin{align*}
\text{\klalm:}\quad\log(\text{T})&~\approx1.31\log(N_e)+0.32;\\
\text{\sklalm:}\quad\log(\text{T})&~\approx1.06\log(N_e)+0.49.
\end{align*}

\par The obtained scalings of the \klalm~and \sklalm~methods with respect to $N_e$ are in rough accordance with the computational complexities summarized in Table \ref{tab:complexities}, where $\tau=0.5$. The marked deviation in the scalings of the \klalm~and \sklalm~methods with respect to $K$ can be ascribed to the implementation in software as well as the difference in the numbers of iterations for each value of $K$.

\section{Conclusions}\label{sec:conclusion}

\par Leveraging the tools from OT and the importance sampling technique, we introduce novel \bcd-type methods, the (S-)\eralm~and (S-)\klalm~methods, for the multi-block optimization problems over the transport polytopes. These methods enjoy highly scalable schemes for the subproblems and save considerable expenditure for the calculations and storage in large-scale contexts. 
With the theory of randomized matrix sparsification, we establish for the \eralm~and \seralm~methods that the average stationarity violations tend to 0 as the problem size increases to $+\infty$ (with probability going to $1$). To the best of our knowledge, our work is the first attempt in applying the matrix entrywise sampling technique to multi-block nonconvex settings with theoretical guarantees. In the simulations of strongly correlated electron systems, though lacking convergence analysis, the (S-)\klalm~methods exhibit desirable robustness to the choices of proximal parameters;
compared with the \klalm~method, the \sklalm~method trades accuracy for efficiency. Both the \klalm~and \sklalm~methods are further brought together under a cascadic multigrid optimization framework, pursuing a decent balance for large-scale simulations. The numerical results conform to both theoretical predictions and physical intuitions. The better scalability of the sampling-based methods allows the first visualization of the approximate OT maps for the 3D systems.

\par For future work, we are eager to provide some insights into the theoretical advantages of importance sampling, as evidenced by the numerical results in section \ref{subsubsec:different gamma}. 
It is also of special interest in optimization to investigate the convergence properties for the proposed KL divergence-based methods by exploring routes that do not rest on the local Lipschitz smoothness of the proximal terms. 
Computationally, further acceleration can be gained via other techniques, such as support identification \cite{lee2023accelerating}. In terms of quantum physics application, it is necessary to study the solution landscape for problem \eqref{eqn:MPGCC l1}. Our numerical findings suggest that the objective errors at stationary points decrease as the discretization becomes finer.

\appendix

\section{Convergence of the {\sc eralm} method}\label{appsec:convergence eralm}

\par This part includes the proofs of the results concerning the \eralm~method. We first adapt a lemma from \cite{kerdoncuff2021sampled}, characterizing the objective error induced by entropy regularization on an arbitrary OT problem. 

\begin{lemma}\label{lemma:entropy error}
Let $W\in\R^{m\times n}$, $\mba\in\R^m$, $\mbb\in\R^n$, $\lambda>0$, and $\calI\subseteq\{(j,k):j=1,\ldots,m,~k=1,\ldots,n\}$. Suppose that $T'$, $T''\in\R^{m\times n}$ are the optimal solutions of 
$$\min_T~\inner{W,T},~\st~T\in\calU(\mba,\mbb),~T_{\calI^c}=0$$
and
$$\min_T~\inner{W,T}+\lambda h(T),~\st~T\in\calU(\mba,\mbb),~T_{\calI^c}=0,$$
respectively. Then $0\le\inner{W,T''-T'}\le-\lambda h(\mba\mbb^\T)$. 
\end{lemma}

\begin{proof}
See the proof of \cite[Lemma 1]{kerdoncuff2021sampled}, which leverages the fact that $0\ge h(T)\ge h(\mba\mbb^\top)$ for all $T\in\calU(\mba,\mbb)$.
\end{proof}

\par Next, we provide an upper bound for the size of $\calU(\mba_i,\mbb_i)$ ($i\in\{1,\ldots,N\}$).

\begin{lemma}\label{lemma:diameter}
For $i\in\{1,\ldots,N\}$, there holds 
$$\snorm{T-T'}\le2d_i,~~\text{for all}~T,T'\in\calU(\mba_i,\mbb_i).$$
\end{lemma}

\begin{proof}
It suffices to bound the norm of an arbitrary member in $\calU(\mba_i,\mbb_i)$. Note that, for any $T\in\calU(\mba_i,\mbb_i)$,
\begin{align}
\snorm{T}&=\sqrt{\sum_{j=1}^{m_i}\sum_{t=1}^{n_i}t_{jk}^2}\le\sqrt{\sum_{j=1}^{m_i}\lrbracket{\sum_{t=1}^{n_i}t_{jk}}^2}=\snorm{T\one_{n_i}}=\norm{\mba_i}\nonumber\\
&=\sqrt{\sum_{j=1}^{m_i}a_{i,j}^2}\le\sqrt{m_i\snorm{\mba_i}_\infty^2}=\sqrt{m_i}\snorm{\mba_i}_\infty.\nonumber
\end{align}
Similarly, we have $\snorm{T}\le\sqrt{n_i}\snorm{\mbb_i}_\infty$. Then by triangle inequality and the definition of $d_i$ in Theorem \ref{theorem:convergence of eralm}, we complete the proof.
\end{proof}

\par Based upon Assumption \ref{assume:Lip grad} and Lemma \ref{lemma:diameter}, the residual function $R_i$ in equation \eqref{eqn:residual function} enjoys a Lipschitz-like property, which will be useful for the convergence proof. 

\begin{lemma}\label{lemma:Lip residual}
Let $X'$, $X''\in\bigtimes_{i=1}^N\calU(\mba_i,\mbb_i)$. Suppose that Assumption \ref{assume:Lip grad} holds and $X_i'=X_i''$ for some $i\in\{1,\ldots,N\}$. Then $\abs{R_i(X')-R_i(X'')}\le2d_iL\snorm{X'-X''}$. 
\end{lemma}

\begin{proof}
Let $\bar X_i'$ and $\bar X_i''$ be the optimal solutions of 
$$\min_{X_i\in\calU(\mba_i,\mbb_i)}~\inner{\nabla_if(X'),X_i}\quad\text{and}\quad\min_{X_i\in\calU(\mba_i,\mbb_i)}~\inner{\nabla_if(X''),X_i},$$
respectively. Simple algebraic calculations yield
\begin{align*}
&R_i(X')=\inner{\nabla_if(X'),X_i'-\bar X_i'}~(\text{equation \eqref{eqn:residual function}})\\
&=\inner{\nabla_if(X''),X_i'-\bar X_i'}+\inner{\nabla_if(X')-\nabla_if(X''),X_i'-\bar X_i'}\\
&\le\inner{\nabla_if(X''),X_i'-\bar X_i'}+2d_iL\snorm{X'-X''}~(\text{Assumption \ref{assume:Lip grad} and Lemma \ref{lemma:diameter}})\\
&=\inner{\nabla_if(X''),X_i''-\bar X_i'}+2d_iL\snorm{X'-X''}~(\text{since}~X_i'=X_i'')\\
&\le\inner{\nabla_if(X''),X_i''-\bar X_i''}+2d_iL\snorm{X'-X''}~(\text{the definition of }\bar X_i'')\\
&=R_i(X'')+2d_iL\snorm{X'-X''}.~(\text{equation}~\eqref{eqn:residual function})
\end{align*}
Similarly, one can show $R_i(X'')\le R_i(X')+2d_iL\snorm{X'-X''}$. These two together give the desired result.
\end{proof}

\par With the above tools in place, we are ready to establish an upper bound for the average residual over the iterate sequence generated by the \eralm~method. 

\begin{proof}[Proof of Theorem \ref{theorem:convergence of eralm}]
For any $i\in\{1,\ldots,N\}$ and $t\ge0$,
\begin{align*}
&f(X_{\le i}^{(t+1)},X_{>i}^{(t)})\le f(X_{<i}^{(t+1)},X_{\ge i}^{(t)})+\inner{C_i^{(t)},X_i^{(t+1)}-X_i^{(t)}}+\frac{L}{2}\snorm{X_i^{(t+1)}-X_i^{(t)}}^2\\
=&~f(X_{<i}^{(t+1)},X_{\ge i}^{(t)})+\alpha\inner{C_i^{(t)},\tilde X_i^{(t+1)}-X_i^{(t)}}+\frac{\alpha^2L}{2}\snorm{\tilde X_i^{(t+1)}-X_i^{(t)}}^2\\
\le&~f(X_{<i}^{(t+1)},X_{\ge i}^{(t)})+\alpha\inner{C_i^{(t)},\tilde X_i^{(t+1)}-X_i^{(t)}}+2(d_i\alpha)^2L\\
\le&~f(X_{<i}^{(t+1)},X_{\ge i}^{(t)})-\alpha\lambda h(\mba_i\mbb_i^\T)-\alpha R_i(X_{<i}^{(t+1)},X_{\ge i}^{(t)})+2(d_i\alpha)^2L,
\end{align*}
where the first inequality uses Assumption \ref{assume:Lip grad}, the second one uses Lemma \ref{lemma:diameter}, the last one relies on Lemma \ref{lemma:entropy error} and the definition \eqref{eqn:residual function} of the residual function $R_i$. The above relation further gives
\begin{equation}
\alpha R_i(X_{<i}^{(t+1)},X_{\ge i}^{(t)})\le f(X_{<i}^{(t+1)},X_{\ge i}^{(t)})-f(X_{\le i}^{(t+1)},X_{>i}^{(t)})-\alpha\lambda h(\mba_i\mbb_i^\T)+2(d_i\alpha)^2L.
\label{eqn:residual ub 1}
\end{equation}
Note that
\begin{align*}
&\abs{R_i(X_{<i}^{(t+1)},X_{\ge i}^{(t)})-{R_i(X^{(t)})}}^2\\
\le&~4(d_iL)^2\snorm{X^{(t+1)}-X^{(t)}}^2=4(d_iL)^2\sum_{i=1}^N\snorm{X_i^{(t+1)}-X_i^{(t)}}^2\\
=&~4(d_iL\alpha)^2\sum_{i=1}^N\snorm{\tilde X_i^{(t+1)}-X_i^{(t)}}^2\le16\bar d^4L^2\alpha^2N,
\end{align*}
where the first inequality follows from Lemma \ref{lemma:Lip residual} and the last one is due to Lemma \ref{lemma:diameter}. Combining the above inequality and relation \eqref{eqn:residual ub 1} yields
\begin{align*}
&~\alpha R_i(X^{(t)})=\alpha\lrsquare{R_i(X_{<i}^{(t+1)},X_{\ge i}^{(t)})+\lrbracket{{R_i(X^{(t)})-R_i(X_{<i}^{(t+1)},X_{\ge i}^{(t)})}}}\\
\le&~f(X_{<i}^{(t+1)},X_{\ge i}^{(t)})-f(X_{\le i}^{(t+1)},X_{>i}^{(t)})-\alpha\lambda h(\mba_i\mbb_i^\T)+4\bar d^2\alpha^2L\sqrt{N}+2(d_i\alpha)^2L\\
\le&~f(X_{<i}^{(t+1)},X_{\ge i}^{(t)})-f(X_{\le i}^{(t+1)},X_{>i}^{(t)})+\alpha \lambda\bar h+2\bar d^2\alpha^2L(2\sqrt{N}+1).
\end{align*}
Summing the above inequality over $i$ from $1$ to $N$ and dividing both sides by $\alpha$, one obtains
$$R(X^{(t)})\le \frac{f(X^{(t)})-f(X^{(t+1)})}{\alpha}+N\lambda\bar h+2\bar d^2LN(2\sqrt{N}+1)\alpha.$$
Summing the above relation over $t$ from 0 to $t_{\max}-1$ and dividing the both sides by $t_{\max}$, we complete the proof after noting $f(X^{(t_{\max})})\ge\underline{f}$, the definition of $\alpha$ in equation \eqref{eqn:optimal step size full}, and $2N(2\sqrt{N}+1)<(2N+1)^2$. Incidentally, the lower bound for $t_{\max}$ in equation \eqref{eqn:optimal step size full} ensures $\alpha\le1$.
\end{proof}

\begin{proof}[Proof of Corollary \ref{corollary:asymptotic full}]
By Assumption \ref{assumption:discretization} (i)-(iii), it is not hard to derive that 
\begin{align*}
\bar d&=\max_{i=1}^N\min\{\sqrt{m_i}\norm{\mba_i}_\infty,\sqrt{n_i}\norm{\mbb_i}_\infty\}=\Theta\lrbracket{\max_{i=1}^N\min\lrbrace{\frac{1}{\sqrt{m_i}},\frac{1}{\sqrt{n_i}}}},\\
\bar h&=\max_{i=1}^N\sum_{j,k}a_{i,j}b_{i,k}(1-\log a_{i,j}b_{i,k})=\Theta\lrbracket{\sum_{i=1}^N\log m_in_i}.
\end{align*}
\par Based on the above bounds, the first term on the right-hand side of inequality \eqref{eq:thm1} goes to $0$ if $t_{\max}=\Omega(\sum_{i=1}^N(m_i+n_i)^\eta)$ with $\eta>\theta$ and $M$ independent from $\{m_i\}_{i=1}^N$ and $\{n_i\}_{i=1}^N$, and the second one goes to $0$ if $\lambda=o(1/\sum_{i=1}^N\log m_in_i)$. The proof is complete. 
\end{proof}

\section{Convergence of the {\sc S-eralm} method}\label{appsec:convergence seralm}

\par We define the following auxiliary sequences: 
\begin{align*}
\bar X_i^{(t+1)}&\in\argmin_{X_i}~\inner{C_i^{(t)},X_i},~\st~X_i\in\calU(\mba_i,\mbb_i),~i=1,\ldots,N,~t\ge0,\\
\breve{X}_i^{(t+1)}&=\argmin_{X_i}~\inner{C_i^{(t)},X_i}+\hat\lambda h(X_i),~\st~X_i\in\calU(\mba_i,\mbb_i),~i=1,\ldots,N,~t\ge0.
\end{align*}

\begin{proof}[Proof of Theorem \ref{theorem:convergence of seralm}]
For any $i\in\{1,\ldots,N\}$ and $t\ge0$,
\begin{align}
\label{eqn:recursive sample}
&~\lrbracket{f(X_{\le i}^{(t+1)},X_{>i}^{(t)})-f(X_{<i}^{(t+1)},X_{\ge i}^{(t)})}/\alpha-2d_i^2L\alpha\\
\le&~\lrbracket{f(X_{\le i}^{(t+1)},X_{>i}^{(t)})-f(X_{<i}^{(t+1)},X_{\ge i}^{(t)})}/\alpha-L\alpha\snorm{\tilde X_i^{(t+1)}-X_i^{(t)}}^2/2\nonumber\\
=&~\lrbracket{f(X_{\le i}^{(t+1)},X_{>i}^{(t)})-f(X_{<i}^{(t+1)},X_{\ge i}^{(t)})-L\snorm{X_i^{(t+1)}-X_i^{(t)}}^2/2}/\alpha\nonumber\\
\le&~\inner{C_i^{(t)},X_i^{(t+1)}-X_i^{(t)}}/\alpha=\inner{C_i^{(t)},\tilde X_i^{(t+1)}-X_i^{(t)}}\nonumber\\
=&~\inner{\hat C_i^{(t)},\tilde X_i^{(t+1)}}+\inner{C_i^{(t)}-\hat C_i^{(t)},\tilde X_i^{(t+1)}}-\inner{C_i^{(t)},X_i^{(t)}}\nonumber\\
=&~\underbrace{\inner{\hat C_i^{(t)},\tilde X_i^{(t+1)}}-\inner{C_i^{(t)},\bar X_i^{(t+1)}}}_{I_1}+\underbrace{\inner{C_i^{(t)}-\hat C_i^{(t)},\tilde X_i^{(t+1)}}}_{I_2}-R_i(X_{<i}^{(t+1)},X_{\ge i}^{(t)}),\nonumber
\end{align}
where the first inequality follows from Lemma \ref{lemma:diameter}, the second one uses Assumption \ref{assume:Lip grad}, and the last equality uses the definition \eqref{eqn:residual function} of the residual function $R_i$.

\par Next, we seek to bound $I_1$ and $I_2$ in relation \eqref{eqn:recursive sample}. For the latter, 
\begin{align}
&\inner{C_i^{(t)}-\hat C_i^{(t)},\tilde X_i^{(t+1)}}\le d_i\snorm{\hat C_i^{(t)}-C_i^{(t)}}=d_i\hat\lambda\sqrt{\sum_{j,k:(j,k)\in\calI_i^{(t)}}\log^2\big(n_{s,i}\cdot p_{i,jk}^{(t)}\big)}\nonumber\\
&\quad\le d_i\hat\lambda\sqrt{|\calI_i^{(t)}|\log^2\frac{1}{(1-\gamma)w_i  n_{s,i}}}=d_i\hat\lambda\sqrt{|\calI_i^{(t)}|}\log\frac{1}{(1-\gamma)w_i  n_{s,i}},\nonumber
\end{align}
where the first inequality comes from the proof of Lemma \ref{lemma:diameter} and the second one uses the formula \eqref{eqn:subsampling probability} and Assumption \ref{assumption:kernel and sampling} (ii), the first equality follows from equation \eqref{eqn:subprob cost matrix sampling} and Assumption \ref{assumption:kernel and sampling} (ii). Following Hoeffding's inequality, for any $\iota>0$,
$$\PP\lrbracket{|\calI_i^{(t)}|\ge n_{s,i}+\iota\cdot m_in_i}\le\exp\lrbracket{-2\iota^2m_in_i}.$$
Therefore, with probability no less than $1-\exp(-2\iota^2m_in_i)$, we have
\begin{equation}
\inner{C_i^{(t)}-\hat C_i^{(t)},\tilde X_i^{(t+1)}}\le d_i\hat{\lambda}\sqrt{n_{s,i}+\iota\cdot m_in_i}\log\frac{1}{(1-\gamma)w_in_{s,i}}.
\label{eqn:sampling error 1}
\end{equation}
Regarding the former term $I_1$, 
\begin{align}
\label{eqn:sampling error 2}
&\inner{\hat C_i^{(t)},\tilde X_i^{(t+1)}}-\inner{C_i^{(t)},\bar X_i^{(t+1)}}\\
\le&~\inner{\hat C_i^{(t)},\tilde X_i^{(t+1)}}-\inner{C_i^{(t)},\breve{X}_i^{(t+1)}}-\hat\lambda h(\mba_i\mbb_i^\T)\nonumber\\
=&~\lrsquare{\inner{\hat C_i^{(t)},\tilde X_i^{(t+1)}}+\hat\lambda h(\tilde X_i^{(t+1)})}-\lrsquare{\inner{C_i^{(t)},\breve{X}_i^{(t+1)}}+\hat\lambda h(\breve{X}_i^{(t+1)})}\nonumber\\
&~+\hat\lambda\lrsquare{h(\breve{X}_i^{(t+1)})-h(\tilde X_i^{(t+1)})}-\hat\lambda h(\mba_i\mbb_i^\T)\nonumber\\
\le&~q_i(\tilde\uu_i^{(t,\star)},\tilde\vvv_i^{(t,\star)};\hat\lambda,\hat\Psi_i^{(t)})-q_i(\tilde\uu_i^{(t,\star)},\tilde\vvv_i^{(t,\star)};\hat\lambda,\Psi_i^{(t)})-2\hat\lambda h(\mba_i\mbb_i^\T)\nonumber\\
\le&~\hat c_2\hat\lambda\frac{\snorm{\hat\Psi_i^{(t)}-\Psi_i^{(t)}}_2}{\snorm{\Psi_i^{(t)}}_2}\abs{1-\frac{\snorm{\hat\Psi_i^{(t)}-\Psi_i^{(t)}}_2}{\snorm{\Psi_i^{(t)}}_2}}^{-1}-2\hat\lambda h(\mba_i\mbb_i^\T),\nonumber
\end{align}
where the first inequality uses Lemma \ref{lemma:entropy error} (with $\calI=\calI_i^{(t)}$), the second one leverages $h(\breve{X}_i^{(t+1)})\le0$, $h(\tilde X_i^{(t+1)})\ge h(\mba_i\mbb_i^\T)$, and the definition \eqref{eqn:dual of eralm} together with the strong duality and the optimality of $(\tilde\uu_i^{(t,\star)},\tilde\vvv_i^{(t,\star)})$ in problem \eqref{eqn:dual of seralm}, the last one follows from \cite{li2023efficient}. To bound the first term on the right-hand side of the inequality \eqref{eqn:sampling error 2}, we resort to \cite[Theorem 3.1]{achlioptas2007fast}. Specifically, for any $j\in\{1,\ldots,m_i\}$, $k\in\{1,\ldots,n_i\}$, it holds by equation \eqref{eqn:subsampling probability} and Assumption \ref{assumption:kernel and sampling} that $$\E\lrbracket{\frac{\hat{\psi}_{i,jk}^{(t)}}{\snorm{\Psi_i^{(t)}}_2}} = n_{s,i}\cdot p_{i,jk}^{(t)} \cdot \frac{\psi_{i,jk}^{(t)}}{n_{s,i}\cdot p_{i,jk}^{(t)}} \cdot \frac{1}{\snorm{\Psi_i^{(t)}}_2} =\frac{\psi_{i,jk}^{(t)}}{\snorm{\Psi_i^{(t)}}_2},$$
\begin{align}
\var\lrbracket{\frac{\hat{\psi}_{i,jk}^{(t)}}{\snorm{\Psi_i^{(t)}}_2}}&<\frac{(\psi_{i,jk}^{(t)})^2}{n_{s,i}\cdot p_{i,jk}^{(t)}\snorm{\Psi_i^{(t)}}_2^2}\le\frac{1}{(1-\gamma)w_i  n_{s,i}\snorm{\Psi_i^{(t)}}_2^2}\nonumber\\
&\le\frac{c_1^2(m_i+n_i)^{-2\nu}}{(1-\gamma)w_i  n_{s,i}}\le\frac{c_1^2\log^4(1+\eps)}{8(m_i+n_i)\log^4(m_i+n_i)}.\nonumber
\end{align}
In addition, $\hat{\psi}_{i,jk}^{(t)}/\snorm{\Psi_i^{(t)}}_2$ lies in an interval with length no larger than
$$\begin{aligned}
\frac{\psi_{i,jk}^{(t)}}{n_{s,i}\cdot p_{i,jk}^{(t)}\snorm{\Psi_i^{(t)}}_2}&\le\frac{1}{(1-\gamma)w_i  n_{s,i}\snorm{\Psi_i^{(t)}}_2}\le\frac{c_1(m_i+n_i)^{-\nu}}{(1-\gamma)w_i  n_{s,i}}\\
&\le\frac{c_1\log^4(1+\eps)}{8(m_i+n_i)^{1-\nu}\log^4(m_i+n_i)}.
\end{aligned}$$
Then, by Theorem 3.1 in \cite{achlioptas2007fast},
$$\PP\lrbracket{\frac{\snorm{\hat\Psi_i^{(t)}-\Psi_i^{(t)}}_2}{\snorm{\Psi_i^{(t)}}_2}\ge c_1(1+\eps+\zeta)\frac{\log^2(1+\eps)}{\log^2(m_i+n_i)}}<2\exp\lrbracket{-\frac{16\zeta^2}{\eps^4}\log^4(m_i+n_i)}$$
holds for any $\zeta>0$ and $m_i+n_i\ge152$. That is to say, with probability no less than $1-2\exp\big(-16\zeta^2\log^4(m_i+n_i)/\eps^4\big)$,
$$\frac{\snorm{\hat\Psi_i^{(t)}-\Psi_i^{(t)}}_2}{\snorm{\Psi_i^{(t)}}_2}<c_1(1+\eps+\zeta)\frac{\log^2(1+\eps)}{\log^2(m_i+n_i)},$$
and 
\begin{align}
\label{eqn:high prob}
&\frac{\snorm{\hat\Psi_i^{(t)}-\Psi_i^{(t)}}_2}{\snorm{\Psi_i^{(t)}}_2}\abs{1-\frac{\snorm{\hat\Psi_i^{(t)}-\Psi_i^{(t)}}_2}{\snorm{\Psi_i^{(t)}}_2}}^{-1}\\
<&~\frac{c_1(1+\eps+\zeta)\log^2(1+\eps)}{\log^2(m_i+n_i)-c_1(1+\eps+\zeta)\log^2(1+\eps)}.\nonumber
\end{align}

\par Combining relations \eqref{eqn:recursive sample} to \eqref{eqn:high prob}, we have, with probability no less than 
$$\lrsquare{1-2\exp\lrbracket{-\frac{16\zeta^2}{\eps^4}\log^4(m_i+n_i)}}\lrsquare{1-\exp\lrbracket{-2\iota^2m_in_i}},$$
that
\begin{align*}
f(X_{\le i}^{(t+1)},X_{>i}^{(t)})\le& ~f(X_{<i}^{(t+1)},X_{\ge i}^{(t)})-\alpha R_i(X_{<i}^{(t+1)},X_{\ge i}^{(t)})+2(d_i\alpha)^2L\\
&~-2\alpha\hat\lambda h(\mba_i\mbb_i^\T)+d_i\alpha\hat\lambda\sqrt{n_{s,i}+\iota\cdot m_in_i}\log\frac{1}{(1-\gamma)w_i  n_{s,i}}\\
&~+\alpha\hat\lambda\frac{c_1\hat c_2(1+\eps+\zeta)\log^2(1+\eps)}{\log^2(m_i+n_i)-c_1(1+\eps+\zeta)\log^2(1+\eps)},
\end{align*}
which implies
\begin{align*}
\alpha R_i(X_{<i}^{(t+1)},X_{\ge i}^{(t)})\le&~ f(X_{<i}^{(t+1)},X_{\ge i}^{(t)})-f(X_{\le i}^{(t+1)},X_{>i}^{(t)})+2(d_i\alpha)^2L\\
&~-2\alpha\hat\lambda h(\mba_i\mbb_i^\T)+d_i\alpha\hat\lambda\sqrt{n_{s,i}+\iota\cdot m_in_i}\log\frac{1}{(1-\gamma)w_i  n_{s,i}}\\
&~+\alpha\hat\lambda\frac{c_1\hat c_2(1+\eps+\zeta)\log^2(1+\eps)}{\log^2(m_i+n_i)-c_1(1+\eps+\zeta)\log^2(1+\eps)}.
\end{align*}
Following similar arguments for proving Theorem \ref{theorem:convergence of eralm} and noticing the definition of $c_3$, one obtains the desired result.
\end{proof}

\begin{proof}[Proof of Corollary \ref{corollary:asymptotic sample}]
Assumption \ref{assumption:discretization} implies bounds for $\bar d$, $\bar h$ (see the proof of Corollary \ref{corollary:asymptotic full} in Appendix \ref{appsec:convergence eralm}) and also $w_i=\Theta(1/m_in_i)$ ($i=1,\ldots,N$). 

\par The first term on the right-hand side of inequality \eqref{eq:thm2} tends to $0$ if $t_{\max}=\Theta(\sum_{i=1}^N(m_i+n_i)^\eta)$ with $\eta>\theta$ and $M$ independent from $\{m_i\}_{i=1}^N$ and $\{n_i\}_{i=1}^N$. For a fixed $\iota>0$, since
\begin{align*}
&~\sum_{i=1}^N\sqrt{n_{s,i}+\iota\cdot m_in_i}\log\frac{1}{(1-\gamma)w_in_{s,i}}\\
=&~\calO\lrbracket{\sum_{i=1}^N\lrsquare{\frac{\sqrt{m_in_i}}{(m_i+n_i)^{\nu-1/2}}\log^2(m_i+n_i)+\sqrt{m_in_i}}\log\frac{(m_i+n_i)^{2\nu-1}}{\log^4(m_i+n_i)}},
\end{align*}
with the choices of $\{n_{s,i}\}_{i=1}^N$, the second and third terms tend to $0$ if
$$\hat\lambda=o\lrbracket{\frac{1}{\sum_{i=1}^N\sqrt{m_in_i}\log(m_i+n_i)}}$$
and $\eps$, $\nu$, $\gamma$ are independent from $\{m_i\}_{i=1}^N$ and $\{n_i\}_{i=1}^N$. Incidentally, the choices of $\{n_{s,i}\}_{i=1}^N$ do not conflict with Assumption \ref{assumption:kernel and sampling} (ii) by virtue of Remark \ref{remark:justification on sample size}. Since $c_1$, $c_2$, and $\hat c_2$ are also independent from $\{m_i\}_{i=1}^N$ and $\{n_i\}_{i=1}^N$, the last term vanishes as $\sum_{i=1}^N(m_i+n_i)\to+\infty$. Finally, the probability is not less than 
\begin{align*}
&~\prod_{i=1}^N\lrbrace{\lrsquare{1-2\exp\lrbracket{-\frac{16\zeta^2}{\eps^4}\log^4(m_i+n_i)}}\lrsquare{1-\exp\lrbracket{-2\iota^2m_i n_i}}}^{t_{\max}}\\
\ge&~\prod_{i=1}^N\lrbrace{\lrsquare{1-\frac{2}{(m_i+n_i)^{16\zeta^2/\eps^4}}}\lrsquare{1-\exp\lrbracket{-2\iota^2m_i n_i}}}^{t_{\max}},
\end{align*}
which, by Assumption \ref{assumption:discretization} (iv), goes to $1$ as $\sum_{i=1}^N(m_i+n_i)\to+\infty$ after choosing $\zeta>0$ such that $\zeta>\sqrt{\eta}\eps^2/4$. The proof is completed.
\end{proof}


\begin{bibdiv}
	\begin{biblist}
		
		\bib{achlioptas2013near}{article}{
			author={Achlioptas, D.},
			author={Karnin, Z.~S.},
			author={Liberty, E.},
			title={Near-optimal entrywise sampling for data matrices},
			date={2013},
			conference={
				title={Advances in Neural Information Processing Systems},
			},
			book={
				volume={26},
				editor={Burges, C.~J.},
				editor={Bottou, L.},
				editor={Welling, M.},
				editor={Ghahramani, Z.},
				editor={Weinberger, K.~Q.},
				publisher={Curran Associates, Inc.},
			},
			pages={1565\ndash 1573},
		}
		
		\bib{achlioptas2007fast}{article}{
			author={Achlioptas, D.},
			author={McSherry, F.},
			title={Fast computation of low-rank matrix approximations},
			date={2007},
			journal={J. ACM},
			volume={54},
			number={2},
			pages={Art. 9, 19},
			doi={10.1145/1219092.1219097},
		}
		
		\bib{ahookhosh2021multi}{article}{
			author={Ahookhosh, M.},
			author={Hien, L. T.~K.},
			author={Gillis, N.},
			author={Patrinos, P.},
			title={Multi-{B}lock {B}regman proximal alternating linearized
				minimization and its application to orthogonal nonnegative matrix
				factorization},
			date={2021},
			journal={Comput. Optim. Appl.},
			volume={79},
			number={3},
			pages={681\ndash 715},
			doi={10.1007/s10589-021-00286-3},
		}
		
		\bib{ai2021optimal}{article}{
			author={Ai, M.},
			author={Wang, F.},
			author={Yu, J.},
			author={Zhang, H.},
			title={Optimal subsampling for large-scale quantile regression},
			journal={J. Complex.},
			volume={62},
			pages={101512},
			date={2021},
			publisher={Elsevier},
			issn = {0885-064X},
			doi = {10.1016/j.jco.2020.101512},
		}
		
		\bib{alfonsi2022constrained}{article}{
			author={Alfonsi, A.},
			author={Coyaud, R.},
			author={Ehrlacher, V.},
			title={Constrained overdamped {L}angevin dynamics for symmetric
				multimarginal optimal transportation},
			date={2022},
			journal={Math. Models Methods Appl. Sci.},
			volume={32},
			number={03},
			pages={403\ndash 455},
			doi={10.1142/S0218202522500105},
		}
		
		\bib{alfonsi2021approximation}{article}{
			author={Alfonsi, A.},
			author={Coyaud, R.},
			author={Ehrlacher, V.},
			author={Lombardi, D.},
			title={Approximation of optimal transport problems with marginal moments
				constraints},
			date={2021},
			journal={Math. Comp.},
			volume={90},
			number={328},
			pages={689\ndash 737},
			doi={10.1090/mcom/3568},
		}
		
		\bib{altschuler2019massively}{article}{
			author={Altschuler, J.},
			author={Bach, F.},
			author={Rudi, A.},
			author={Niles-Weed, J.},
			title={Massively scalable {S}inkhorn distances via the {N}ystr\"om method},
			conference={
				title={Advances in Neural Information Processing Systems},
			},
			book={
				editor={Wallach, H.},
				editor={Larochelle, H.},
				editor={Beygelzimer, A.},
				editor={d'Alch\'e-Buc, F.},
				editor={Fox, E.},
				editor={Garnett, R.},
				volume={32},
				publisher={Curran Associates, Inc.},
			},
			pages={4427\ndash 4437},
			date={2019},
		}
		
		\bib{arjovsky2017wasserstein}{article}{
			author={Arjovsky, M.},
			author={Chintala, S.},
			author={Bottou, L.},
			title={Wasserstein generative adversarial networks},
			date={2017},
			conference={
				title={Proceedings of the 34th International Conference on Machine Learning},
			},
			book={
				editor={Precup, D.},
				editor={Teh, Y.~W.},
				series={Proceedings of Machine Learning Research},
				volume={70},
				publisher={PMLR},
			},
			pages={214\ndash 223},
		}
		
		\bib{beck2015cyclic}{article}{
			author={Beck, A.},
			author={Pauwels, E.},
			author={Sabach, S.},
			title={The cyclic block conditional gradient method for convex
				optimization problems},
			date={2015},
			journal={SIAM J. Optim.},
			volume={25},
			number={4},
			pages={2024\ndash 2049},
			doi={10.1137/15M1008397},
		}
		
		\bib{benamou2002monge}{article}{
			author={Benamou, J.-D.},
			author={Brenier, Y.},
			author={Guittet, K.},
			title={The {M}onge-{K}antorovitch mass transfer and its computational fluid
				mechanics formulation},
			note={ICFD Conference on Numerical Methods for Fluid Dynamics (Oxford,
				2001)},
			journal={Internat. J. Numer. Methods Fluids},
			volume={40},
			date={2002},
			number={1-2},
			pages={21--30},
			issn={0271-2091},
			doi={10.1002/fld.264},
		}
		
		
		\bib{benamou2016numerical}{article}{
			author={Benamou, J.-D.},
			author={Carlier, G.},
			author={Nenna, L.},
			title={A numerical method to solve multi-marginal optimal transport
				problems with {C}oulomb cost},
			conference={
				title={Splitting Methods in Communication, Imaging, Science, and
					Engineering},
			},
			book={
				editor={Glowinski, R.},
				editor={Osher, S. J.},
				editor={Yin, W.},
				series={Sci. Comput.},
				publisher={Springer, Cham},
			},
			isbn={978-3-319-41587-1},
			isbn={978-3-319-41589-5},
			date={2016},
			pages={577--601},
		}
		
		\bib{bertsekas2016nonlinear}{book}{
			author={Bertsekas, D. P.},
			title={Nonlinear {P}rogramming},
			series={Athena Scientific Optimization and Computation Series},
			edition={3},
			publisher={Athena Scientific, Belmont, MA},
			date={2016},
			pages={xviii+861},
			isbn={978-1-886529-05-2},
			isbn={1-886529-05-1},
		}
		
		\bib{bigot2012consistent}{arXiv}{
			author={Bigot, J.},
			author={Klein, T.},
			title={Consistent estimation of a population barycenter in the {W}asserstein space},
			archiveprefix={arXiv},
			eprint={1212.2562},
			primaryclass={math.ST},
			date={2012}
		}
		
		
		\bib{bolte2014proximal}{article}{
			author={Bolte, J.},
			author={Sabach, S.},
			author={Teboulle, M.},
			title={Proximal alternating linearized minimization for nonconvex and
				nonsmooth problems},
			date={2014},
			journal={Math. Program.},
			volume={146},
			number={1-2},
			series={A},
			pages={459\ndash 494},
			doi={10.1007/s10107-013-0701-9},
		}
		
		\bib{borzi2008multigrid}{article}{
			author={Borz\`{i}, A.},
			author={Hohenester, U.},
			title={Multigrid optimization schemes for solving Bose-Einstein
				condensate control problems},
			journal={SIAM J. Sci. Comput.},
			volume={30},
			date={2008},
			number={1},
			pages={441--462},
			issn={1064-8275},
			doi={10.1137/070686135},
		}
		
		\bib{braun2022conditional}{arXiv}{
			author={Braun, G.},
			author={Carderera, A.},
			author={Combettes, C. W.},
			author={Hassani, H.},
			author={Karbasi, A.},
			author={Mokhtari, A.},
			author={Pokutta, S.},
			title={Conditional gradient methods},
			archiveprefix={arXiv},
			eprint={2211.14103},
			primaryclass={math.OC},
			date={2022}
		}
		
		\bib{braverman2021near}{article}{
			author={Braverman, V.},
			author={Krauthgamer, R.},
			author={Krishnan, A.~R.},
			author={Sapir, S.},
			title={Near-optimal entrywise sampling of numerically sparse matrices},
			date={2021},
			conference={
				title={Proceedings of the 34th Conference on Learning Theory},
			},
			book={
				editor={Belkin, M.},
				editor={Kpotufe, S.},
				series={Proceedings of Machine Learning Research},
				volume={134},
				publisher={PMLR},
			},
			pages={759\ndash 773},
		}
		
		\bib{bregman1967relaxation}{article}{
			author={Br\`egman, L. M.},
			title={A relaxation method of finding a common point of convex sets and
				its application to the solution of problems in convex programming},
			language={Russian},
			journal={\v{Z}. Vy\v{c}isl. Mat i Mat. Fiz.},
			volume={7},
			date={1967},
			pages={620--631},
			issn={0044-4669},
		}
		
		\bib{brenier1997homogenized}{article}{
			author={Brenier, Y.},
			title={A homogenized model for vortex sheets},
			journal={Arch. Rational Mech. Anal.},
			volume={138},
			date={1997},
			number={4},
			pages={319--353},
			issn={0003-9527},
			doi={10.1007/s002050050044},
		}
		
		\bib{buttazzo2012optimal}{article}{
			author={Buttazzo, G.},
			author={{De Pascale}, L.},
			author={Gori-Giorgi, P.},
			title={Optimal-transport formulation of electronic density-functional
				theory},
			date={2012},
			journal={Phys. Rev. A},
			volume={85},
			number={6},
			pages={062502},
			doi={10.1103/PhysRevA.85.062502}
		}
		
		\bib{carlier2015numerical}{article}{
			author={Carlier, G.},
			author={Oberman, A.},
			author={Oudet, E.},
			title={Numerical methods for matching for teams and {W}asserstein
				barycenters},
			journal={ESAIM Math. Model. Numer. Anal.},
			volume={49},
			date={2015},
			number={6},
			pages={1621--1642},
			issn={2822-7840},
			doi={10.1051/m2an/2015033},
		}
		
		\bib{chen2019extended}{article}{
			author={Chen, C.},
			author={Li, M.},
			author={Liu, X.},
			author={Ye, Y.},
			title={Extended {ADMM} and {BCD} for nonseparable convex minimization models
				with quadratic coupling terms: {C}onvergence analysis and insights},
			journal={Math. Program.},
			volume={173},
			date={2019},
			number={1-2},
			pages={37--77},
			issn={0025-5610},
			doi={10.1007/s10107-017-1205-9},
		}
		
		
		\bib{chen2014numerical}{article}{
			author={Chen, H.},
			author={Friesecke, G.},
			author={Mendl, C.~B.},
			title={Numerical methods for a {K}ohn-{S}ham density functional model
				based on optimal transport},
			date={2014},
			journal={J. Chem. Theory Comput.},
			volume={10},
			number={10},
			pages={4360\ndash 4368},
			doi={10.1021/ct500586q}
		}
		
		\bib{chen2017efficient}{article}{
			author={Chen, J.},
			author={Garc\'{\i}a-Cervera, C. J.},
			title={An efficient multigrid strategy for large-scale molecular
				mechanics optimization},
			journal={J. Comput. Phys.},
			volume={342},
			date={2017},
			pages={29--42},
			issn={0021-9991},
			doi={10.1016/j.jcp.2017.04.035},
		}
		
		\bib{colombo2015multimarginal}{article}{
			author={Colombo, M.},
			author={{De Pascale}, L.},
			author={{Di Marino}, S.},
			title={Multimarginal optimal transport maps for one-dimensional
				repulsive costs},
			date={2015},
			journal={Canad. J. Math.},
			volume={67},
			number={2},
			pages={350\ndash 368},
			doi={10.4153/CJM-2014-011-x},
		}
		
		\bib{cotar2013density}{article}{
			author={Cotar, C.},
			author={Friesecke, G.},
			author={Kl{\"u}ppelberg, C.},
			title={Density functional theory and optimal transportation with
				{C}oulomb cost},
			date={2013},
			journal={Comm. Pure Appl. Math.},
			volume={66},
			number={4},
			pages={548\ndash 599},
			doi={10.1002/cpa.21437},
		}
		
		\bib{cuturi2013sinkhorn}{article}{
			author={Cuturi, M.},
			title={Sinkhorn distances: {L}ightspeed computation of optimal
				transport},
			date={2013},
			conference={
				title={Advances in Neural Information Processing Systems},
			},
			book={
				editor={Burges, C.~J.},
				editor={Bottou, L.},
				editor={Welling, M.},
				editor={Ghahramani, Z.},
				editor={Weinberger, K.~Q.},
				volume={26},
				publisher={Curran Associates, Inc.},
			},
			pages={2292\ndash 2300}
		}
		
		\bib{cuturi2014fast}{article}{
			author={Cuturi, M.},
			author={Doucet, A.},
			title={Fast computation of {W}asserstein barycenters},
			date={2014},
			conference={
				title={Proceedings of the 31st International Conference on Machine Learning},
			},
			book={
				editor={Xing, E. P.},
				editor={Jebara, T.},
				series={Proceedings of Machine Learning Research},
				volume={32},
				number={2},
				publisher={PMLR},
			},
			pages={685\ndash 693},
		}
		
		\bib{dagotto2005complexity}{article}{
			author={Dagotto, E.},
			title={Complexity in strongly correlated electronic systems},
			journal={Science},
			volume={309},
			number={5732},
			pages={257\ndash 262},
			date={2005},
			publisher={American Association for the Advancement of Science},
			doi={10.1126/science.1107559}
		}
		
		\bib{di2020optimal}{arXiv}{
			author={{Di Marino}, S.},
			author={Gerolin, A.},
			title={Optimal transport losses and {S}inkhorn algorithm with
				general convex regularization},
			date={2020},
			eprint={2007.00976},
			archiveprefix={arXiv},
			primaryclass={math.OC}
		}
		
		\bib{driggs2021stochastic}{article}{
			author={Driggs, D.},
			author={Tang, J.},
			author={Liang, J.},
			author={Davies, M.},
			author={Sch\"{o}nlieb, C.-B.},
			title={A stochastic proximal alternating minimization for nonsmooth and
				nonconvex optimization},
			journal={SIAM J. Imaging Sci.},
			volume={14},
			date={2021},
			number={4},
			pages={1932--1970},
			doi={10.1137/20M1387213},
		}
		
		\bib{drineas2011note}{article}{
			author={Drineas, P.},
			author={Zouzias, A.},
			title={A note on element-wise matrix sparsification via a matrix-valued
				{B}ernstein inequality},
			date={2011},
			journal={Inform. Process. Lett.},
			volume={111},
			number={8},
			pages={385\ndash 389},
			doi={10.1016/j.ipl.2011.01.010},
		}
		
		\bib{dvurechensky2018computational}{article}{
			author={Dvurechensky, P.},
			author={Gasnikov, A.},
			author={Kroshnin, A.},
			title={Computational optimal transport: {C}omplexity by accelerated
				gradient descent is better than by {S}inkhorn's algorithm},
			date={2018},
			conference={
				title={Proceedings of the 35th International Conference on Machine Learning},
			},
			book={
				editor={Dy, J.},
				editor={Krause, A.},
				series={Proceedings of Machine Learning Research},
				volume={80},
				publisher={PMLR},
			},
			pages={1367\ndash 1376},
		}
		
		\bib{elvira2021advances}{article}{
			author={Elvira, V.},
			author={Martino, L.},
			title={Advances in importance sampling},
			date={2021},
			journal={Wiley Statist. Ref. Stat. Ref. Online},
			pages={1\ndash 14},
			doi = {10.1002/9781118445112.stat08284},
		}
		
		
		\bib{fercoq2016optimization}{article}{
			author={Fercoq, O.},
			author={Richt\'{a}rik, P.},
			title={Optimization in high dimensions via accelerated, parallel, and
				proximal coordinate descent},
			journal={SIAM Rev.},
			volume={58},
			date={2016},
			number={4},
			pages={739--771},
			issn={0036-1445},
			doi={10.1137/16M1085905},
		}
		
		\bib{filatov2015spin}{article}{
			author={Filatov, M.},
			title={Spin-restricted ensemble-referenced {K}ohn-{S}ham method: {B}asic
				principles and application to strongly correlated ground and excited states
				of molecules},
			date={2015},
			journal={Wiley Interdiscip. Rev. Comput. Mol. Sci.},
			volume={5},
			number={1},
			pages={146\ndash 167},
			doi={10.1002/wcms.1209}
		}
		
		
		\bib{friesecke2023strong}{article}{
			author={Friesecke, G.},
			author={Gerolin, A.},
			author={Gori-Giorgi, P.},
			title={The strong-interaction limit of density functional
				theory},
			date={2023},
			conference={
				title={Density {F}unctional {T}heory: {M}odeling, {M}athematical {A}nalysis, {C}omputational {M}ethods, and {A}pplications},
			},
			book={
				editor={Canc\`es, E.},
				editor={Friesecke, G.},
				series={Mathematics and Molecular Modeling},
				publisher={Springer},
			},
			pages={183\ndash 266},
			doi={10.1007/978-3-031-22340-2\_4}
		}
		
		
		\bib{friesecke2022genetic}{article}{
			author={Friesecke, G.},
			author={Schulz, A.~S.},
			author={V\"ogler, D.},
			title={Genetic column generation: {F}ast computation of high-dimensional
				multimarginal optimal transport problems},
			date={2022},
			journal={SIAM J. Sci. Comput.},
			volume={44},
			number={3},
			pages={A1632\ndash A1654},
			doi={10.1137/21M140732X},
		}
		
		
		\bib{geng2016label}{article}{
			author={Geng, X.},
			title={Label distribution learning},
			journal={IEEE Trans. Knowl. Data Eng.},
			volume={28},
			number={7},
			pages={1734\ndash 1748},
			date={2016},
			doi={10.1109/TKDE.2016.2545658}
		}
		
		
		\bib{hertrich2022inertial}{article}{
			author={Hertrich, J.},
			author={Steidl, G.},
			title={Inertial stochastic {PALM} and applications in machine learning},
			journal={Sampl. Theory Signal Process. Data Anal.},
			volume={20},
			date={2022},
			number={1},
			pages={Paper No. 4, 33},
			issn={2730-5716},
			doi={10.1007/s43670-022-00021-x},
		}
		
		\bib{hosseini2022intrinsic}{article}{
			author={Hosseini, B.},
			author={Steinerberger, S.},
			title={Intrinsic sparsity of {K}antorovich solutions},
			language={English, with English and French summaries},
			journal={C. R. Math. Acad. Sci. Paris},
			volume={360},
			date={2022},
			pages={1173--1175},
			issn={1631-073X},
			doi={10.5802/crmath.392},
		}
		
		\bib{hu2023global}{article}{
			author={Hu, Y.},
			author={Chen, H.},
			author={Liu, X.},
			title={A global optimization approach for multi-marginal optimal
				transport problems with {C}oulomb cost},
			date={2023},
			journal={SIAM J. Sci. Comput.},
			volume={45},
			number={3},
			pages={A1214\ndash A1238},
			doi={10.1137/21M1455164},
		}
		
		\bib{hu2023convergence}{article}{
			author={Hu, Y.},
			author={Liu, X.},
			title={The convergence properties of infeasible inexact proximal
				alternating linearized minimization},
			journal={Sci. China Math.},
			volume={66},
			date={2023},
			number={10},
			pages={2385--2410},
			issn={1674-7283},
			doi={10.1007/s11425-022-2074-7},
		}
		
		\bib{hu2023exactness}{article}{
			author={Hu, Y.},
			author={Liu, X.},
			title={The exactness of the $\ell_1$ penalty function for a class of
				mathematical programs with generalized complementarity constraints},
			date={2023},
			journal={Fundam. Res.},
			status={in press},
			doi={10.1016/j.fmre.2023.04.006}
		}
		
		\bib{kantorovich1942transfer}{article}{
			author={Kantorovich, L. V.},
			title={On the translocation of masses},
			journal={C. R. (Doklady) Acad. Sci. URSS (N.S.)},
			date={1942},
			pages={199--201},
		}
		
		\bib{kerdoncuff2021sampled}{article}{
			author={Kerdoncuff, T.},
			author={Emonet, R.},
			author={Sebban, M.},
			title={Sampled {G}romov {W}asserstein},
			date={2021},
			journal={Mach. Learn.},
			volume={110},
			number={8},
			pages={2151\ndash 2186},
			doi={10.1007/s10994-021-06035-1},
		}
		
		\bib{khoo2020semidefinite}{article}{
			author={Khoo, Y.},
			author={Lin, L.},
			author={Lindsey, M.},
			author={Ying, L.},
			title={Semidefinite relaxation of multimarginal optimal transport for
				strictly correlated electrons in second quantization},
			date={2020},
			journal={SIAM J. Sci. Comput.},
			volume={42},
			number={6},
			pages={B1462\ndash B1489},
			doi={10.1137/20M1310977},
		}
		
		\bib{khoo2019convex}{article}{
			author={Khoo, Y.},
			author={Ying, L.},
			title={Convex relaxation approaches for strictly correlated density
				functional theory},
			date={2019},
			journal={SIAM J. Sci. Comput.},
			volume={41},
			number={4},
			pages={B773\ndash B795},
			doi={10.1137/18M1207478},
		}
		
		\bib{kullback1951information}{article}{
			author={Kullback, S.},
			author={Leibler, R.~A.},
			title={On information and sufficiency},
			date={1951},
			journal={Ann. Math. Statistics},
			volume={22},
			number={1},
			pages={79\ndash 86},
			doi={10.1214/aoms/1177729694},
		}
		
		\bib{kundu2017recovering}{article}{
			author={Kundu, A.},
			author={Drineas, P.},
			author={Magdon-Ismail, M.},
			title={Recovering {PCA} and sparse {PCA} via hybrid-$(\ell_1,\ell_2)$
				sparse sampling of data elements},
			date={2017},
			journal={J. Mach. Learn. Res.},
			volume={18},
			number={75},
			pages={1\ndash 34},
			url={http://jmlr.org/papers/v18/16-258.html},
		}
		
		\bib{lacoste2013block}{article}{
			author={Lacoste-Julien, S.},
			author={Jaggi, M.},
			author={Schmidt, M.},
			author={Pletscher, P.},
			title={Block-coordinate {Frank-Wolfe} optimization for structural
				{SVMs}},
			date={2013},
			conference={
				title={Proceedings of the 30th International Conference on Machine Learning},
			},
			book={
				editor={Dasgupta, S.},
				editor={McAllester, D.},
				series={Proceedings of Machine Learning Research},
				volume={28},
				publisher={PMLR},
			},
			pages={53\ndash 61},
		}
		
		\bib{lee2023accelerating}{article}{
			author={Lee, C.-P.},
			title={Accelerating inexact successive quadratic approximation for
				regularized optimization through manifold identification},
			journal={Math. Program.},
			volume={201},
			date={2023},
			number={1-2},
			pages={599--633},
			issn={0025-5610},
			doi={10.1007/s10107-022-01916-2},
		}
		
		\bib{li2023importance}{article}{
			author={Li, M.},
			author={Yu, J.},
			author={Li, T.},
			author={Meng, C.},
			title={Importance sparsification for {S}inkhorn algorithm},
			date={2023},
			journal={J. Mach. Learn. Res.},
			volume={24},
			number={247},
			pages={1\ndash 44},
			url={https://jmlr.org/papers/v24/22-1311.html}
		}
		
		\bib{li2023efficient}{article}{
			author={Li, M.},
			author={Yu, J.},
			author={Xu, H.},
			author={Meng, C.},
			title={Efficient approximation of {G}romov-{W}asserstein distance using
				importance sparsification},
			journal={J. Comput. Graph. Statist.},
			volume={32},
			date={2023},
			number={4},
			pages={1512--1523},
			issn={1061-8600},
			doi={10.1080/10618600.2023.2165500},
		}
		
		\bib{li2019provable}{arXiv}{
			author={Li, Q.},
			author={Zhu, Z.},
			author={Tang, G.},
			author={Wakin, M.~B.},
			title={Provable {B}regman-divergence based methods for nonconvex
				and non-{L}ipschitz problems},
			archiveprefix={arXiv},
			eprint={1904.09712},
			primaryclass={math.OC},
			date={2019},
		}
		
		
		
		
		\bib{liu1996metropolized}{article}{
			author={Liu, J.~S.},
			title={Metropolized independent sampling with comparisons to rejection
				sampling and importance sampling},
			date={1996},
			journal={Stat. Comput.},
			volume={6},
			number={2},
			pages={113\ndash 119},
			doi={10.1007/BF00162521}
		}
		
		\bib{liu2008monte}{book}{
			author={Liu, J.~S.},
			title={Monte Carlo Strategies in Scientific Computing},
			series={Springer Series in Statistics},
			publisher={Springer, New York},
			date={2008},
			pages={xvi+343},
			isbn={978-0-387-76369-9},
			isbn={0-387-95230-6},
		}
		
		\bib{liu2021multilevel}{article}{
			author={Liu, J.},
			author={Yin, W.},
			author={Li, W.},
			author={Chow, Y. T.},
			title={Multilevel optimal transport: {A} fast approximation of
				Wasserstein-1 distances},
			journal={SIAM J. Sci. Comput.},
			volume={43},
			date={2021},
			number={1},
			pages={A193--A220},
			issn={1064-8275},
			doi={10.1137/18M1219813},
		}
		
		
		\bib{luo1993convergence}{article}{
			author={Luo, Z.-Q.},
			author={Tseng, P.},
			title={On the convergence rate of dual ascent methods for linearly
				constrained convex minimization},
			date={1993},
			journal={Math. Oper. Res.},
			volume={18},
			number={4},
			pages={846\ndash 867},
			doi={10.1287/moor.18.4.846},
		}
		
		\bib{ma2019optimal}{article}{
			author={Ma, M.},
			author={Wang, X.},
			author={Duan, Y.},
			author={Frey, S.~H.},
			author={Gu, X.},
			title={Optimal mass transport based brain morphometry for patients with
				congenital hand deformities},
			date={2019},
			journal={Vis. Comput.},
			volume={35},
			pages={1311\ndash 1325},
			doi={10.1007/s00371-018-1543-5}
		}
		
		\bib{ma2014statistical}{article}{
			author={Ma, P.},
			author={Mahoney, M.},
			author={Yu, B.},
			title={A statistical perspective on algorithmic leveraging},
			date={2014},
			conference={
				title={Proceedings of the 31st International Conference on Machine Learning},
			},
			book={
				editor={Xing, E.~P.},
				editor={Jebara, T.},
				series={Proceedings of Machine Learning Research},
				volume={32},
				publisher={PMLR},
			},
			pages={91\ndash 99},
		}
		
		
		\bib{mendl2013kantorovich}{article}{
			author={Mendl, C.~B.},
			author={Lin, L.},
			title={Kantorovich dual solution for strictly correlated electrons in
				atoms and molecules},
			date={2013},
			journal={Phys. Rev. B},
			volume={87},
			number={12},
			pages={125106},
			doi={10.1103/PhysRevB.87.125106}
		}
		
		\bib{meng2019large}{article}{
			author={Meng, C.},
			author={Ke, Y.},
			author={Zhang, J.},
			author={Zhang, M.},
			author={Zhong, W.},
			author={Ma, P.},
			title={Large-scale optimal transport map estimation using projection
				pursuit},
			date={2019},
			conference={
				title={Advances in Neural Information Processing Systems},
			},
			book={
				editor={Wallach, H.},
				editor={Larochelle, H.},
				editor={Beygelzimer, A.},
				editor={d\textquotesingle Alch\'{e}-Buc, F.},
				editor={Fox, E.},
				editor={Garnett, R.},
				volume={32},
				publisher={Curran Associates, Inc.},
			},
			pages={8118\ndash 8129}
		}
		
		\bib{monge1781memoire}{article}{
			author={Monge, G.},
			title={M\'emoire sur la th\'eorie des d\'eblais et des remblais},
			journal={Histoire de l'Acad\'emie Royale des Sciences},
			pages={666\ndash 704},
			date={1781}
		}
		
		
		\bib{owen2013monte}{book}{
			author={Owen, A.~B.},
			title={Monte {C}arlo {T}heory, {M}ethods and {E}xamples},
			publisher={Stanford University},
			date={2013},
			url={http://statweb.stanford.edu/~owen/mc/},
		}
		
		
		
		\bib{pele2009fast}{article}{
			author={Pele, O.},
			author={Werman, M.},
			title={Fast and robust earth mover's distances},
			conference={
				title={IEEE 12th International Conference on Computer Vision},
			},
			book={
				publisher={IEEE},
			},
			date={2009},
			pages={460\ndash 467},
			doi={10.1109/ICCV.2009.5459199}
		}
		
		\bib{peyre2019computational}{article}{
			author={Peyr{\'e}, G.},
			author={Cuturi, M.},
			title={Computational optimal transport: {W}ith applications to data
				science},
			date={2019},
			journal={Found. Trends Mach. Learn.},
			volume={11},
			number={5-6},
			pages={355\ndash 607},
			doi={10.1561/2200000073}
		}
		
		
		\bib{rubner1997earth}{article}{
			author={Rubner, Y.},
			author={Guibas, L. J.},
			author={Tomasi, C.},
			title={The earth mover's distance, multi-dimensional scaling, and color-based image retrieval},
			conference={
				title={Proceedings of the ARPA Image Understanding Workshop},
			},
			book={
				publisher={ARPA},
			},
			pages={661\ndash 668},
			date={1997}
		}
		
		\bib{seidl1999strong}{article}{
			author={Seidl, M.},
			title={Strong-interaction limit of density-functional theory},
			date={1999},
			journal={Phys. Rev. A},
			volume={60},
			number={6},
			pages={4387},
			doi={10.1103/PhysRevA.60.4387}
		}
		
		
		
		\bib{seidl2000simulation}{article}{
			author={Seidl, M.},
			author={Perdew, J.~P.},
			author={Kurth, S.},
			title={Simulation of all-order density-functional perturbation theory,
				using the second order and the strong-correlation limit},
			date={2000},
			journal={Phys. Rev. Lett.},
			volume={84},
			number={22},
			pages={5070},
			doi={10.1103/PhysRevLett.84.5070}
		}
		
		\bib{seidl1999strictly}{article}{
			author={Seidl, M.},
			author={Perdew, J.~P.},
			author={Levy, M.},
			title={Strictly correlated electrons in density-functional theory},
			date={1999},
			journal={Phys. Rev. A},
			volume={59},
			number={1},
			pages={51},
			doi={10.1103/PhysRevA.59.51}
		}
		
		\bib{shannon1948mathematical}{article}{
			author={Shannon, C. E.},
			title={A mathematical theory of communication},
			journal={Bell System Tech. J.},
			volume={27},
			date={1948},
			pages={379--423, 623--656},
			issn={0005-8580},
			doi={10.1002/j.1538-7305.1948.tb01338.x},
		}
		
		\bib{sinkhorn1967concerning}{article}{
			author={Sinkhorn, R.},
			author={Knopp, P.},
			title={Concerning nonnegative matrices and doubly stochastic matrices},
			date={1967},
			journal={Pacific J. Math.},
			volume={21},
			number={2},
			pages={343\ndash 348},
			doi={10.2307/2314570},
		}
		
		\bib{sun2020efficiency}{article}{
			author={Sun, R.},
			author={Luo, Z.-Q.},
			author={Ye, Y.},
			title={On the efficiency of random permutation for {ADMM} and coordinate
				descent},
			journal={Math. Oper. Res.},
			volume={45},
			date={2020},
			number={1},
			pages={233--271},
			issn={0364-765X},
			doi={10.1287/moor.2019.0990},
		}
		
		\bib{villani2003topics}{book}{
			author={Villani, C.},
			title={Topics in Optimal Transportation},
			publisher={American Mathematical Society},
			date={2003},
			volume={58},
			doi={10.1090/gsm/058},
		}
		
		
		\bib{wang2021comparative}{article}{
			author={Wang, H.},
			author={Zou, J.},
			title={A comparative study on sampling with replacement vs {P}oisson sampling in optimal subsampling},
			date={2021},
			conference={
				title={Proceedings of the 24th International Conference on Artificial Intelligence and Statistics},
			},
			book={
				editor={Banerjee, A.},
				editor={Fukumizu, K.},
				series={Proceedings of Machine Learning Research},
				volume={130},
				publisher={PMLR},
			},
			pages={289\ndash 297},
		}
		
		\bib{wang2022sampling}{article}{
			author={Wang, J.},
			author={Zou, J.},
			author={Wang, H.},
			title={Sampling with replacement vs {P}oisson sampling: {A} comparative study
				in optimal subsampling},
			journal={IEEE Trans. Inform. Theory},
			volume={68},
			date={2022},
			number={10},
			pages={6605--6630},
			issn={0018-9448},
		}
		
		\bib{wright2015coordinate}{article}{
			author={Wright, S. J.},
			title={Coordinate descent algorithms},
			journal={Math. Program.},
			volume={151},
			date={2015},
			number={1},
			pages={3--34},
			issn={0025-5610},
			doi={10.1007/s10107-015-0892-3},
		}
		
		\bib{xia2018cascadic}{article}{
			author={Xia, Q.},
			author={Shi, T.},
			title={A cascadic multilevel optimization algorithm for the design of composite structures with curvilinear fiber based on {S}hepard interpolation},
			journal={Compos. Struct.},
			volume={188},
			pages={209--219},
			date={2018},
			doi={10.1016/j.compstruct.2018.01.013}
		}
		
		\bib{xie2020fast}{article}{
			author={Xie, Y.},
			author={Wang, X.},
			author={Wang, R.},
			author={Zha, H.},
			title={A fast proximal point method for computing exact {W}asserstein
				distance},
			date={Jul. 22--25, 2020},
			conference={
				title={Proceedings of the 35th Uncertainty in Artificial Intelligence
					Conference},
			},
			book={
				editor={Adams, R.~P.},
				editor={Gogate, V.},
				series={Proceedings of Machine Learning Research},
				volume={115},
				publisher={PMLR},
			},
			pages={433\ndash 453},
		}
		
		\bib{xu2019learning}{article}{
			author={Xu, L.},
			author={Sun, H.},
			author={Liu, Y.},
			title={Learning with batch-wise optimal transport loss for {3D} shape
				recognition},
			date={2019},
			conference={
				title={Proceedings of the IEEE/CVF Conference on Computer Vision and Pattern Recognition},
			},
			book={
				publisher={IEEE},
			},
			booktitle={Proceedings of the IEEE/CVF Conference on Computer Vision and
				Pattern Recognition},
			pages={3333\ndash 3342},
		}
		
		
		\bib{yang2022bregman}{article}{
			author={Yang, L.},
			author={Toh, K.-C.},
			title={Bregman proximal point algorithm revisited: {A} new inexact
				version and its inertial variant},
			date={2022},
			journal={SIAM J. Optim.},
			volume={32},
			number={3},
			pages={1523\ndash 1554},
			doi={10.1137/20M1360748},
		}
		
		\bib{yu2022optimal}{article}{
			author={Yu, J.},
			author={Wang, H.},
			author={Ai, M.},
			author={Zhang, H.},
			title={Optimal distributed subsampling for maximum quasi-likelihood
				estimators with massive data},
			date={2022},
			journal={J. Amer. Statist. Assoc.},
			volume={117},
			number={537},
			pages={265\ndash 276},
			doi={10.1080/01621459.2020.1773832},
		}
		
		\bib{zhao2018label}{article}{
			author={Zhao, P.},
			author={Zhou, Z.},
			title={Label distribution learning by optimal transport},
			conference={
				title={Proceedings of the AAAI Conference on Artificial Intelligence},
			},
			book={
				editor={McIlraith, S. A.},
				editor={Weinberger, K. Q.},
				volume={32},
				number={1},
				publisher={AAAI Press}
			},
			date={2018},
			pages={4506\ndash 4513},
			doi={10.1609/aaai.v32i1.11609}
		}
		
	\end{biblist}
\end{bibdiv}

\end{document}